\documentclass[pdflatex,sn-mathphys-num]{sn-jnl}

\newgeometry{top=2.5cm, bottom=2.5cm, left=3cm, right=3cm}

\usepackage{graphicx}
\usepackage{multirow}
\usepackage{amsmath,amssymb,amsfonts}
\usepackage{amsthm}
\usepackage{mathrsfs}
\usepackage[title]{appendix}
\usepackage{xcolor}
\usepackage{textcomp}
\usepackage{manyfoot}
\usepackage{booktabs}
\usepackage{algorithm}
\usepackage{algorithmicx}
\usepackage{algpseudocode}
\usepackage{listings}
\usepackage{hyperref}

\usepackage{tikz}
\usetikzlibrary{trees, calc, decorations.pathreplacing}

\usepackage{eucal}          
\usepackage{amscd}          
\usepackage{latexsym}
\usepackage{mathtools}      
\usepackage{dsfont}

\usepackage[all]{xy}

\usepackage{tcolorbox}

\usepackage{float}          
\usepackage{multicol}

\usepackage{subcaption}

\usepackage{epigraph}

\theoremstyle{thmstyleone}
\newtheorem{thm}{Theorem}[section]

\newtheorem{lem}[thm]{Lemma}

\newtheorem{prop}[thm]{Proposition}
\newtheorem{proposition}[thm]{Proposition}

\newtheorem{cor}[thm]{Corollary}

\theoremstyle{thmstyletwo}
\newtheorem{example}[thm]{Example}

\theoremstyle{thmstylethree}
\newtheorem{definition}[thm]{Definition}

\newtheorem*{theorem*}{Theorem}
\newtheorem*{corollary*}{Corollary}

\newtheorem*{Satz*}{Satz}

\newtheorem*{Proposal*}{Proposal}

\begin{document}

\title[Título corto]{Ultrametric Graphons and Hierarchical Community Networks: Spectral Theory and Applications}

\author*[1]{\fnm{Ángel Alfredo} \sur{Morán Ledezma}}\email{angel.ledezma@kit.edu}

\affil*[1]{\orgdiv{Geodetic Institute}, \orgname{Karlsruhe Institute of Technology},
  \orgaddress{\street{Englerstr. 7}, \city{Karlsruhe}, \postcode{76131},
  \state{Karlsruhe}, \country{Germany}}}

\abstract{
We develop a theory of ultrametric graphons as limiting objects for 
random networks with nested hierarchical community structure. A graphon 
$W:[0,1]^2\to[0,1]$ is called ultrametric if $W(x,y)=w(d(x,y))$, 
where $d$ is an ultrametric on $[0,1]$ induced by a family of nested 
partitions and $w$ is a positive kernel. The resulting random graphs 
exhibit a hierarchical community structure in which the density of 
connections is governed by the ultrametric distance between vertices.

The Laplacian $L_d^k$ of the deterministic graph sampled from an 
ultrametric graphon is itself an ultrametric Laplacian, whose 
eigenvalues and spectral projectors admit completely explicit closed-form 
expressions in terms of the community sizes and inter-community connection 
densities. We show that the normalized eigenvalues and spectral projectors 
of the random Laplacian $L_r^k$ are arbitrarily close to those of $L_d^k$ 
with high probability as $k\to\infty$; the explicit formulas for $L_d^k$ 
therefore provide closed-form analytical approximations for the spectrum and spectral projectors
of $L_r^k$. We then develop the following applications. A sign structure theorem for the 
empirical spectral projectors provides a rigorous generalization of the 
Fiedler vector criterion to hierarchical networks with arbitrarily many 
communities. We further establish a detectability threshold in spectral 
community detection for one-level hierarchical graphons, governed by a 
threshold $p^*=\min_i\rho_i$, where $\rho_i$ is the normalized spectral 
gap of the $i$-th sub-community. For random walks, we construct a limiting pseudo-inverse Laplacian 
operator $L_W^+$ and establish its almost sure convergence from the 
pseudo-inverse of the random Laplacian $L_r^k$ in \(L^2\) norm. 
Since the hitting and commute times of the continuous-time Markov chain 
are expressed in terms of the pseudo-inverse Laplacian, this convergence 
implies that both collapse almost surely, in the large-graph limit, to 
quantities depending only on the expected degrees of the endpoints, 
losing all information about the hierarchical community structure. Finally, we apply the framework to the SIS epidemic model on hierarchical 
community networks, deriving explicit closed-form stability conditions for 
the disease-free equilibrium in terms of the cluster sizes and 
inter-community connectivities. These reveal a fundamental tension between 
homogeneous and heterogeneous community structures: global cooperation is 
optimal in the homogeneous case, whereas targeted intervention on the most 
connected sub-community is substantially more effective in the heterogeneous 
case, as confirmed by numerical experiments.}


\maketitle

\tableofcontents

\section{Introduction}

In recent years, networks have become one of the most important mathematical abstractions across the sciences. Due to their generality as models of pairwise interactions among abstract entities, they appear ubiquitously across a wide range of applications \cite{networksnewman}. The study of the properties of networks relies on the body of graph theory, whose tools are used to understand and reveal the characteristics of systems modeled as networks. Community (also called cluster or module) structure  and hierarchical organization are among the most pervasive properties in real-world networks. From ecological networks such as food webs and occurrence networks, through the modular organization of the brain's functional connectome, to the nested hierarchical structure of social systems, the topology of real-world networks is rarely characterized by a single scale. Rather, it is typical to observe a hierarchical organization composed of nested communities realizing several levels of resolution \cite{GirvanNewman2002, NewmanGirvan2004, Fortunato2010, Ravasz2002, SalesPardo2007, Clauset2008, Schaub2023, Dreveton2025, Rezvani2022, Meunier2009, Meunier2010, Wierzbinski2023, Simon1962, BorzoneMas2023}.

The literature on hierarchical community structure in networks has largely focused on methods for detecting such communities, yielding insights into the functional organization of the underlying systems \cite{GirvanNewman2002, NewmanGirvan2004, Fortunato2010, Ravasz2002, SalesPardo2007, Clauset2008, Schaub2023, Dreveton2025, Rezvani2022, Porter2009}.In this context, a fundamental question is not only whether hierarchical community structure exists in a given network, but whether it can be reliably detected. Understanding the limits of detectability is of both theoretical and practical importance, as it determines the regime in which algorithms can be expected to recover the underlying community structure \cite{Schaub2023, Dreveton2025, PhaseTransitionsSpectralCD}.

Random walks on graphs are of great use in different areas of science. From measuring centrality in networks \cite{BrinPage1998, LangvilleMeyer2004, LambiottRosvall2012} to diffusive and spreading processes \cite{MasudaPorterLambiotte2017}, they are a fundamental tool for analysis. Their behaviour is deeply influenced by the underlying network structure, in particular, by the presence of community organization at multiple scales. 

On the other hand, a rigorous graphon-theoretic framework that captures multi-scale hierarchical structure and its spectral consequences for dynamical processes has not been fully developed.  In one direction, Markov processes on graphs have been used as a tool for community detection, yielding multiscale methods that go beyond classical modularity maximization or spectral clustering \cite{Lambiotte2015, Schaub2012, Patelli2020, LiZhang2013}. In the other direction, one can fix a modular structure and study the consequences for the dynamics. This approach has been explored in the context of epidemic spreading: for SIS models on two-scale community networks, it has been shown that community structure lowers the epidemic threshold and that communities act as reservoirs of infection \cite{Bonaccorsi2014, LiuHu2005, Nadini2018}, while for SIR models on hierarchical modular networks, non-universal power-law growth of prevalence has been reported \cite{Odor2021}. However, these works are largely restricted to one or two levels of community structure, and closed-form analytical expressions for epidemic thresholds in terms of the full multi-level hierarchical organization of the network are not available. A unified analytical framework that captures the consequences of arbitrarily nested community structure on Markovian dynamics, and that yields explicit analytical expressions, remains an open problem, which this work addresses.

An important feature of many real-world networks exhibiting hierarchical community structure is that they belong to the dense connectivity regime. For example, in structural brain connectomes it has been shown that removal of weak connections is inconsequential for graph-theoretical analysis, and that reducing connectivity density further introduces instability in graph theoretical analyses \cite{Civier2019}. This observation motivates the study of hierarchical community networks in the dense regime, which is precisely the setting of the present work.

This work proposes a mathematical framework for the study of hierarchical community networks in the dense regime, grounded in the theory of graphons. We introduce the notion of an \emph{ultrametric graphon}: by modeling hierarchical community structure via a graphon whose edge probabilities are governed by an ultrametric on \([0,1]\), the framework simultaneously captures the nested community organization and the dense connectivity regime described above. The central analytical advantage of this approach lies in the explicit solvability of the associated Laplacian: as established in \cite{MoranLedezma2026}, the spectral theory of ultrametric Laplacians on finite spaces admits completely explicit eigenvalues, spectral projectors, and heat kernel formulas. In the graphon limit, this solvability passes to the large-graph regime, yielding closed-form analytical expressions for spectral quantities, random walk functionals, and epidemic thresholds that are not accessible through existing frameworks.

\subsection{Graphons as limit of dense graphs}

Graphons are a promising nonparametric generative model for the study of large networks. The theory, initiated by Lovász and Szegedy \cite{LovaszSzegedy2006} and further developed in \cite{BorgsI2008, BorgsII2012}, provides a rigorous framework for the study of limits of dense graph sequences. A graphon \(
W:[0,1]^2\rightarrow[0,1]\) encodes the asymptotic edge probability between vertices as a function of their latent positions in \([0,1]\), and serves as a generative model for random networks. The monograph \cite{LovaszBook2012} provides a comprehensive treatment of the theory, while considerable attention has been devoted to graphon estimation \cite{GaoLuZhou2015, KloppTsybakov2017}. 

In parallel, graphon operators, integral operators associated with the kernel
\(W\), and their associated dynamics, such as random walks and reaction-diffusion equations, have been studied as infinite-dimensional analogues of their discrete counterparts \cite{PetitLambiotte2021, BramburgerHolzer2023}. Since the spectrum of Laplacian-type operators encodes both topological and dynamical properties of networks, considerable attention has been devoted to the spectral theory of graphon Laplacians \cite{vizuete2021laplacian}. The problem of modularity maximization has also been extended to the graphon setting, providing a graphon approach in the task of finding communities in dense networks \cite{KlimmJonesSchaub2022}.

\subsection{Ultrametric spaces and ultrametric Laplacians}
An ultrametric space $(X, d)$ is a metric space satisfying the 
\emph{strong triangle inequality}:
\begin{equation*}
    d(x, z) \leq \max\{d(x, y), d(y, z)\} \quad \text{for all } x, y, z \in X.
\end{equation*}
Equivalently, every point in a ball is its center, and the open balls 
form a nested, dendogram-compatible structure, an ultrametric tree 
\cite{MoranLedezma2026}. Beginning with the pioneering work of Rammal, 
Toulouse, and Virasoro \cite{rammal1986ultrametricity}, ultrametric structures have been recognized as the natural 
abstraction for systems with hierarchical organization. In recent work \cite{MoranLedezma2026}, the present author developed a unified spectral theory for finite ultrametric Laplacians in the context of phylogenetic trees, yielding closed-form expressions for eigenvalues, heat kernels, and spectral projectors. As noted in \cite{Dreveton2025}, hierarchical 
community detection consists in finding a tree of communities where deeper 
levels of the hierarchy reveal finer-grained structures. It is therefore 
natural to model a hierarchically nested random network via an underlying 
ultrametric tree (space), whose structure directly encodes the community 
hierarchy at every scale. Ultrametric analysis has found applications across diverse scientific domains, 
including evolutionary biology and phylogenetics \cite{MoranLedezma2026}, protein folding and glass relaxation, \cite{MoranLedezma2025}, graph theory and 
network analysis \cite{bradleymoranshape}, and the modeling of 
hierarchical neural dynamics \cite{ZZ2023}.
  \section{Main results and organization}
This paper develops a theory of ultrametric graphons as limiting objects
for random networks with nested hierarchical community structure. A
graphon $W:[0,1]^2\to[0,1]$ is called \emph{ultrametric} if
$W(x,y)=w(d(x,y))$, where $d$ is an ultrametric on $[0,1]$ induced by
a family of nested partitions. The resulting random graphs exhibit a
hierarchical community structure in which the density of connections
between vertices is governed by their ultrametric distance: for example, of \(w\) is non-increasing, vertices
within the same cluster at a fine scale are densely connected, while
vertices belonging to different coarse-scale clusters are sparsely
connected. The spectral theory of ultrametric Laplacians, developed in
Section~\ref{sec:spectraltheoryultra}, provides the analytical
foundation for all subsequent results.

The first set of results concerns the spectral theory of the Laplacian
$L_d^k$ associated with the deterministic graph sampled from an
ultrametric graphon. Since $L_d^k$ is an ultrametric Laplacian in the
sense of Section~\ref{sec:spectraltheoryultra}, its spectrum is
completely explicit: each internal node $I$ of the ultrametric tree modeling the community structure
contributes an eigenvalue $\lambda(I)$ with multiplicity $|C(I)|-1$,
and the associated spectral projector $E_I$ encodes the community
structure at that level of the hierarchy. We show that the normalized
eigenvalues and spectral projectors of the random Laplacian $L_r^k$
converge to those of $L_d^k$ with high probability as the number of
vertices grows, via an adaptation of the Davis--Kahan theorem and
Weyl's inequality to the ultrametric setting. This yields explicit closed-form expressions, in terms of the community 
sizes and inter-community connection densities, for the limiting 
eigenvalues and spectral projectors of the random Laplacian $L_r^k$.

A central result is the \emph{sign structure theorem for spectral
projectors}: for sufficiently large graphs sampled from an ultrametric
graphon, the sign pattern of the empirical spectral projector
$\hat{V}\hat{V}^\top$ associated with a cluster $I$ reveals the
community structure induced by the children of $I$ with high
probability, positive entries for pairs within the same child cluster,
negative entries for pairs in different child clusters. This
constitutes a rigorous generalization of the Fiedler vector criterion
to hierarchical networks with arbitrarily many communities and
non-trivial eigenvalue multiplicities.

We then establish a \emph{spectral threshold} for one-level hierarchical graphons. We show that the detectability of
the community structure via spectral methods is governed by a threshold
$p^* = \min_i \rho_i$, where each $\rho_i$ is the normalized spectral
gap of the $i$-th sub-community. Via Cheeger-type inequalities, $\rho_i$
admits a structural interpretation: it quantifies the effective
connectivity of the sub-network $G_i$, being sensitive both to the
presence of bottlenecks (small conductance) and to the absence of
highly connected vertices (bounded local degree). When the
inter-community edge density $w(h([0,1]))$ falls below $p^*$, spectral
clustering correctly identifies the community structure; when it exceeds
$p^*$, the Fiedler matrix loses all community information. This
generalizes the 
results presented in ~\cite{PhaseTransitionsSpectralCD}
from two balanced communities to an arbitrary number of communities of
heterogeneous sizes.

A second set of results concerns random walks on hierarchical community
networks. We construct a limiting pseudo-inverse Laplacian operator
$L_W^+$ and show that the pseudo-inverse of the random Laplacian
$L_r^k$ converges to $L_W^+$ in \(L^2\) norm almost surely.
As a consequence, the mean first passage times and the commute times of the continuous-time Markov chain
collapse in the large-graph limit to the sum of the inverses of the
expected degrees of the endpoints, losing all information about the
ultrametric community structure. This extends the commute-time collapse
phenomenon, previously established for discrete-time random walks by
von Luxburg et al.~\cite{vonluxburg2014hitting}, to hierarchical
community networks, with a proof that is entirely independent and based
on the spectral theory of the hierarchical pseudo-inverse Laplacian.

Finally, we apply the framework to the analysis of epidemic spreading
on hierarchical community networks via the SIS model. Explicit
closed-form conditions on the graphon structure are derived for the
stability and instability of the disease-free equilibrium, expressed
directly in terms of the sizes and inter-community connectivities of
the clusters at each level of the hierarchy. These analytical
expressions are not accessible through existing frameworks and reveal a
fundamental tension between homogeneous and heterogeneous community
structures: in homogeneous networks, global cooperation, reducing
connectivity uniformly across all communities, is the optimal
intervention strategy, whereas in heterogeneous networks, targeted
intervention on the most connected sub-community is substantially more
effective. These conclusions are supported by numerical experiments on
randomly generated ultrametric graphons with controlled levels of
community heterogeneity.

\section{Spectral Theory of Ultrametric Laplacians}\label{sec:spectraltheoryultra}

The theory of ultrametric Laplacians in the finite setting was fully developed in \cite{MoranLedezma2026}. In this section we review the most important results for this manuscript. 

An ultrametric space is a metric space \((X, d)\) in which the metric \(d\) satisfies the strong triangle inequality:
\[
d(x, z) \leq \max\{ d(x, y), d(y, z) \} \quad \text{for all } x, y, z \in X.
\]

An ultrametric space equipped with a measure \(m\) is called a \emph{measure} ultrametric space. In particular in the finite setting this is just a positive function on the leaves. A \emph{probability measure} is a measure satisfying \[\sum_{x\in X}m(x)=1.\]

Given a finite ultrametric space \((X, d,m)\) with topological tree \(T\) and measure \(m\), consider the operator
\[
L_X u(x) = \sum_{y\in X} w(d(x, y)) \big(u(y) - u(x)\big) \, m(y),
\]
where \(w : [0, \infty) \to \mathbb{R}\) is a given function. This operator acts on functions \(u : X \to \mathbb{R}\). We refer to $L_T$ as the \emph{ultrametric Laplacian operator} attached to the triple $(X,d,m)$.

\begin{definition}\label{def:ultrametriclaplacianmatrix}
The \emph{ultrametric Laplacian matrix} associated to the operator \(L_X\) is the \(|X| \times |X|\) matrix \(L= (L_{x,y})_{x, y \in X}\) whose entries are given by
\[
L_{x,y} =
\begin{cases}
w(d(x, y))\, m(y) & \text{if } y \neq x, \\
- \sum_{z \neq x} w(d(x, z))\, m(z) & \text{if } y = x.
\end{cases}
\]
\end{definition}
This matrix corresponds to the matrix representation of \(L_X\) with respect to the canonical basis \(\{e_y\}_{y \in X}\), where \(X\) is equipped with the canonical inner product and \(e_y(x) = \delta_{x,y}\) where \(\delta_{x,y}\) is the Kronecker delta, defined by
\[
\delta_{x,y} =
\begin{cases}
1 & \text{if } x = y, \\
0 & \text{otherwise}.
\end{cases}
\]

\subsection{The topological tree.}

Every finite ultrametric space \((X,d)\) admits a canonical tree representation. The closed balls of \((X,d)\) form a nested family under inclusion, and this nesting defines a rooted tree \(T\) refer as the \emph{topological tree} of \((X,d)\), whose leaves are the elements of \(X\) and whose internal nodes correspond bijectively to balls. Importantly, the tree \(T\) is not merely a combinatorial object: each internal node \(n\) harries a height \(h(n):=diam(B_n)\), making this tree an ultrametric tree in the sense of Section \(2\) of \cite{MoranLedezma2026}. The distance between two leaves is then recovered as \(d(x,y)=h(x\wedge y)\) where \(x\wedge y\) is the least common ancestor, \(LCA\) for short. Conversely, any ultrametric tree \(T\) induces an ultrametric on its leaves via this formula. This correspondence is a bijection \cite{MoranLedezma2026}. Henceforth, the terms \emph{internal node} and \emph{ball} are used interchangeably, and we write \(B_n\) for the ball associated to \(n\in T\). We denote by \(C(n)\) the set of children nodes of an internal node \(n\).

\begin{figure}[H]
    \centering
    \includegraphics[width=\textwidth]{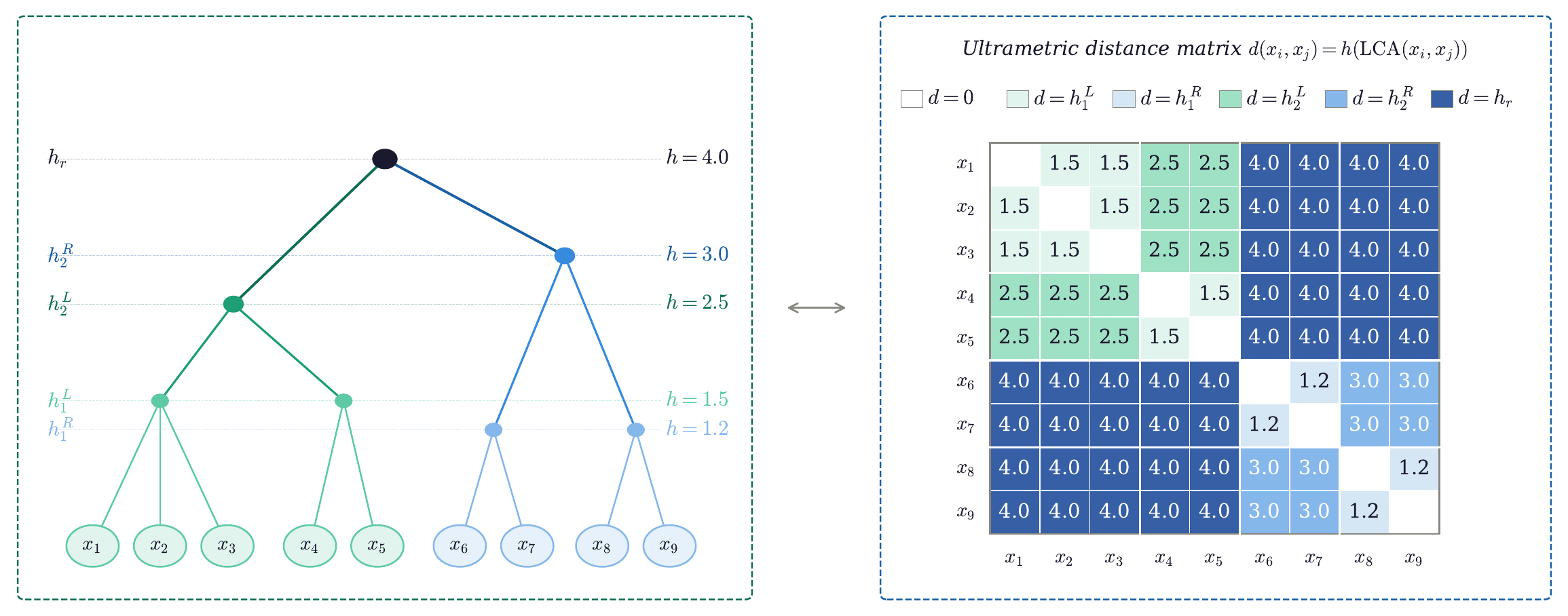}
    \caption{Every finite ultrametric space can be represented as a tree, and vice versa. This correspondence allows one to use the language of ultrametric spaces and the language of trees interchangeably: each ball is in bijection with an internal node, and the diameter of each ball is encoded in the height of the corresponding node.}
    \label{fig:ultraequivalence}
\end{figure}

\subsection{The spectrum and eigen-projectors.}

Let \(n\in T\) be an internal node, define the functions 
\begin{equation}
\label{eq:eigenvectorsultrametric}\varphi_{B_n,l}=\frac{\mathbf{1}_{B_l}}{m(B_l)}-\frac{\mathbf{1}_{B_{n}}}{m(B_{n})},
\quad\text{for each child }l\in C(n).
\end{equation}
Define
\[\mathcal{V}_n:=\left\{ \psi: \psi|_{B_l,} \, \text{is constant for all \(l\in C(n)\)},\, \sum_{l\in C(n)} m(B_l)\psi|_{B_l,}=0, \, Supp \, \psi \subset B_n\right\}.\]
It is easy to see that for a node \(n\) the set of functions \(\varphi_{B_n,l}\) span \(\mathcal{V}_n\).

The dimension of this space is \(|C(n)|-1
\). The following proposition characterize the eigenvalues of the ultrametric Laplacian in terms of the topology and measure on the ultrametric space \((X,d,m)\). 

\begin{prop} \label{prop:eigenvalueofultrametricoperator}
    Let \(\varphi_{B_n,l}\) defined as in equation \ref{eq:eigenvectorsultrametric}. Then \(\varphi_{B_n,l}\) is an eigenfunction of \(L_X\), with eigenvalue 
    \begin{equation} \label{eigenvaluereal}
        \begin{split}
            \lambda_n&=-\sum_{l\in \gamma_r(n)} m(B_l)\left[w(diam(B_l))-w(diam(B_{F(l)})\right]\\
            &=-\sum_{y\in X\setminus B_n} w(d(x_0,y))m(y) - w(diam(B_n))m(B_n),
        \end{split}
    \end{equation}
    where \(x_0\in B_n\), with multiplicity \(m_n:=|C(n)|-1\).
\end{prop}
\begin{proof}
    Let \(n\in T\setminus X\) an internal node. Let \[
\varphi_{B_n,l}=\frac{\mathbf{1}_{B_l}}{m(B_l)}-\frac{\mathbf{1}_{B_{n}}}{m(B_{n})},
\quad\text{for a given child }l\in C(n).
\] define the real number
\(\lambda_n:=-\sum_{y\in X\setminus B_n} k(d(x_0,y))m(y) - k(diam(B_n))m(B_n).\) Define \(a=\frac{1}{m(B_l)}-\frac{1}{m(B_{n})}\) and \(b=-\frac{1}{m(B_{n})}\). Let \(x\in B_l\), therefore, \(\varphi_{B_n,l}(x)=a\), and 
\[\sum_{y\in B_l}k(d(x,y))(\varphi_{B_n,l}(y)-\varphi_{B_n,l}(x))m(y)=0.\] Therefore
\[L_X \,\varphi_{B_n,l}(x)=\sum_{y\in B_{n}\setminus B_l}k(d(x,y))(b-a)m(y)+\sum_{y\in X\setminus B_n}k(d(x,y))(0-a)m(y).\]
Since \(x\in B_l\), for any \(y\in B_n\setminus B_l\), \(k(d(x,y))=k(diam(B_n))\) and 
\begin{equation*}
    \begin{split}
        L_X \,\varphi_{B_n,l}(x)&=k(diam(B_n))(b-a)m(B_n\setminus B_l)-a\sum_{y\in X\setminus B_n}k(d(x,y))m(y)\\
        &=a\left[-k(diam(B_n))\left(1-\frac{b}{a}\right)m(B_n\setminus B_l)+\sum_{y\in X\setminus B_n}k(d(x,y))m(y).\right]
    \end{split}
\end{equation*}
In a similar way, for \(x\in B_n\setminus B_l\) where \(\varphi_{B_n,l}(x)=b\), the following holds
\[L_X \,\varphi_{B_n,l}(x)=b\left[-k(diam(B_n))\left(1-\frac{a}{b}\right)m( B_l)+\sum_{y\in X\setminus B_n}k(d(x,y))m(y).\right]\]
A direct computation leads to the following equalities
\[\left(1-\frac{b}{a}\right)m(B_n\setminus B_l)=m(B_n),\] and 
\[\left(1-\frac{a}{b}\right)m(B_l)=m(B_n).\]
Hence, in both cases 
\[L_X \, \varphi_{B_l,n}(x)=\varphi_{B_l,n}(x) \lambda_n.\]
Lastly, if \(x\in X\setminus B_n\) then 
\[L_X \, \varphi_{B_l,n}(x)=\sum_{y\in X}k(d(x,y)) \varphi_{B_l,n}(y)m(y)=0,\]
since \(\varphi_{B_l,n}\) has mean zero.
\end{proof}
Therefore there is a one by one correspondence between the eigenvalues with multiplicity and the balls of the ultrametric space (and therefore the internal nodes of the topological tree). Hence \(\lambda_n=\lambda(B_n)\). \newline

The eigen-projectors of the ultrametric Laplacian can be expressed in terms of the functions \ref{eq:eigenvectorsultrametric}.

\begin{prop} \label{prop:eigenprojectorsultra}
    The projection kernel associated with the eigenvalue \(\lambda(B_n)\) , denoted by \(E_{B_n}:X\times X\rightarrow\mathbb{R}\) satisfies \(E_{I_n}(x,y)=0\) if \(x\notin B_n\)  or \(y\notin B_n\)and, if \(x\in B_n\) and \(y\in B_n\) then
\begin{equation}\label{eq:projectorgraphon}
    E_{B_n}(x,y)=\frac{1_{B_j}(x)}{m(B_n)}-\frac{1_{B_n}(x)}{m(B_n)},
\end{equation}
where \(B_j\subset B_n\) is the only  child ball of \(B_n\) containing \(y\). 
\end{prop}
\begin{proof}
    Let \(E_n(x,y)\) the projection kernel to the eigen-space \(\mathcal{V}_n\) , that is, for each \(f\in L^2(m)\),
\[\sum_{y\in X} E_n(x,y)f(y)m(y)-f(x) \, \perp \, \mathcal{V}_n.\]
In particular for \(\delta_z\) where \(z\in B_j\subset B_n\) and \(j\in C(n)\), 
\[E(x,z)m(z)=\sum_{y\in X} E_n(x,y)\delta_z(y)m(y)=m(z)\left(\frac{\textbf{1}_{B_j}}{m(B_j)}(x)-\frac{\textbf{1}_{B_{n}}}{m(B_{n})}(x)\right)=m(z)\varphi_{B_j}(x),\]
where the last equality follows from the fact that \(m(z)\varphi_{B_j}-\delta_z \, \perp \, \varphi_{B_l}\), for all \(l\in C(n)\). Consequently 
\begin{equation*}\label{projectionkernel}
    E_n(x,z)=\varphi_{z\in B_j}(x).
\end{equation*}
\end{proof}
\begin{figure}[H]
    \centering
    \includegraphics[width=\textwidth]{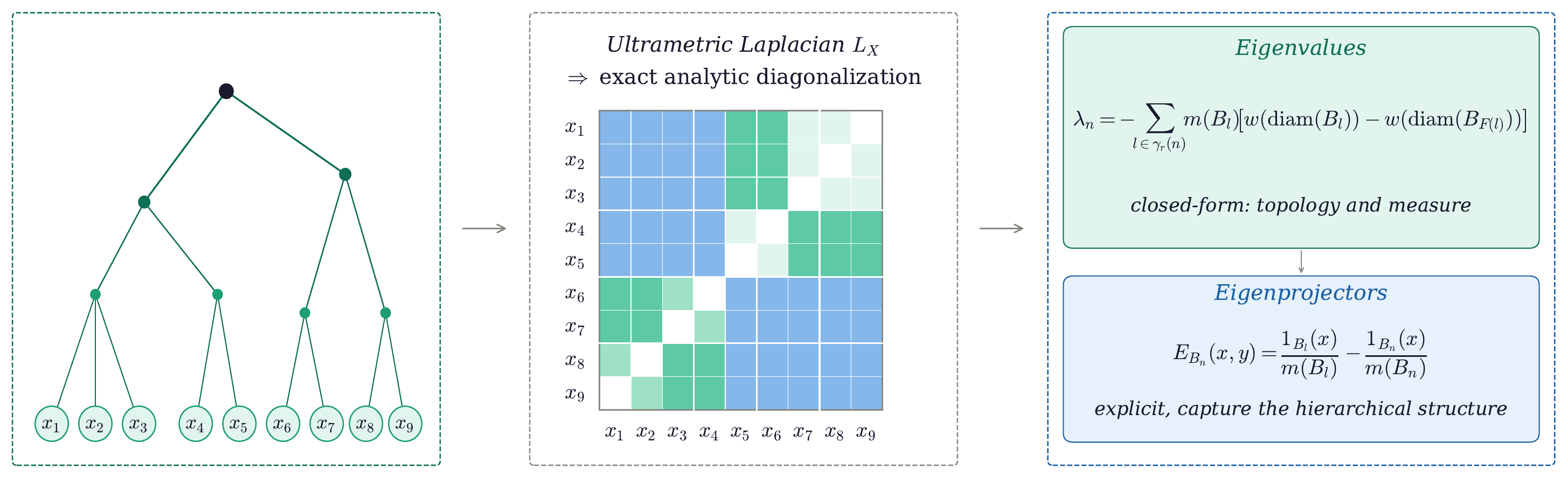}
    \caption{Any ultrametric space \((X,d)\) is equivalently described by an ultrametric tree (left). The distance function via a kernel \(w(d(x,y))\) and a measure \(m\) defines an ultrametric Laplacian \(L_X\) (center) that admits an exact spectral decomposition with explicit eigenvalues \(\lambda_n\) and eigenprojectors of the form \(\varphi_{B_n,l}\)  .}
    \label{fig:ultrametricpipeline}
\end{figure}
Ultrametric Laplacians enjoy a remarkable analytical advantage: their eigenvalues and eigen-projectors admit fully explicit closed-form expressions in terms of the topology and measure of the underlying space. This is summarized in Figure \ref{fig:ultrametricpipeline}. Any measure ultrametric space \((X,d,m)\) can be represented as an ultrametric tree given rise to an ultrametric Laplacian \(L_X\) via a kernel function \(k>0\). The eigenvalues \(\lambda_n=\lambda(B_n)\) and eigen-projectors \(E_n(x,y)\) are completely determined from the hierarchical structure of the space.

\section{Graphons a quick introduction.} \label{sec:graponintro}

In this work we adopt the simplest representation of \emph{graphon}, namely  a symmetric Lebesgue-measurable function  \(W:[0,1]^2 \rightarrow [0,1]\). An intuitive way to think of a graphon is as the limiting object of the heatmap of an adjacency matrix. This is often called the \emph{pixel picture} \cite{Glasscock2015}: as the number of vertices tends to infinity, the image is rescaled to always occupy the unit square. In this limit, the vertices become infinitesimal: each pair \(x,y\in [0,1]\) play the role of vertex indices while \(W(x,y)\) represent the associated weight.

\begin{figure}[H]
    \centering
    \includegraphics[width=\textwidth]{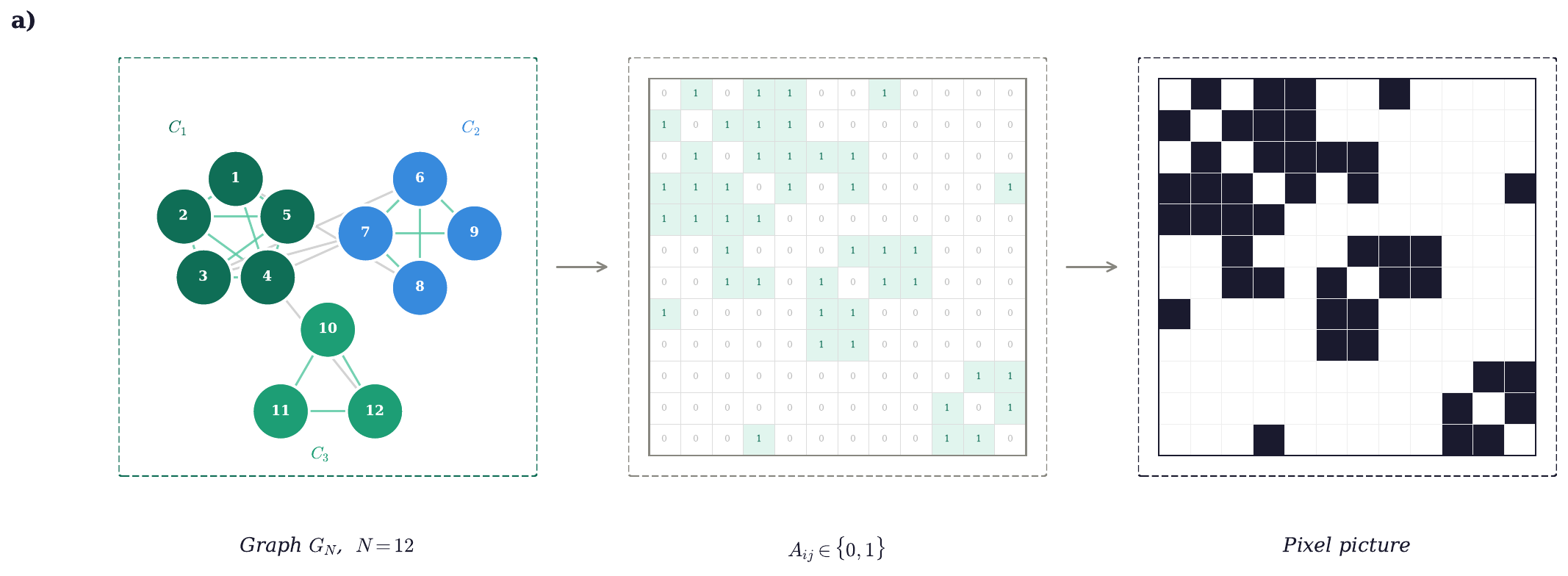}\\[0.8em]
    \includegraphics[width=\textwidth]{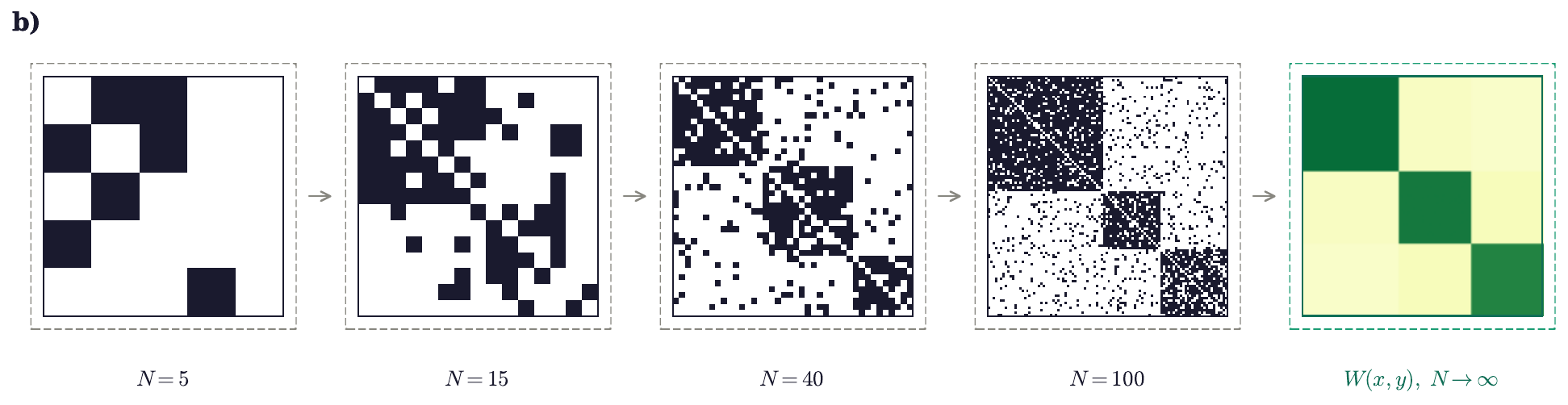}
    \caption{A graph can be represented by its adjacency matrix or equivalently as a pixel plot, as shown in panel \(a)\). A graphon arises as the limiting object of a sequence of pixel plots as the number of vertices tends to infinity; in this sense, a graphon is the limit of a sequence of dense graphs, as illustrated in panel \(b)\).}
    \label{fig:graphonpipeline}
\end{figure}

Since permuting node labels leaves the graph structure unchanged but modifies its adjacency matrix, a graphon is properly an equivalence class: \(W\) and \(W(\phi(x),\phi(y))\) represent the same graphon for any measure-preserving bijection \(\phi:[0,1]\rightarrow [0,1]\). 

A graphon \(W\) also serves as a random graph model. Given \(N\) points \(x_i\) selected uniformly at random from the interval \([0,1]\), one constructs a random graph on vertices \(\{x_i\}\) via the weights \(W(x_i,x_j)\) by constructing an adjacency matrix \(A=[\xi_{ij}]\), whose elements \(\xi_{ij}\in \{0,1\}\), with \(i>j\) are independent random variables with probability distribution 
\begin{equation} \label{eq:distributiongraphon}
    \mathbb{P}(\xi_{ij}=1)=W(x_i,x_j), \qquad \xi_{ij}=\xi_{ji},
\end{equation}
for \(i>j\) and \(\xi_{ii}=0\) for all \(1\leq i\leq N\). Alternatively, one may take a uniform partition \(I_i=[i/N,(i+1)/N)\) and set \(x_i=i/N\). 

Therefore, for any given number of vertices \(N>0\) a graphon \(W\) has attached a weighted deterministic graph with adjacency matrix denoted by \[A_d^N=[W(x_i,x_j)]_{i,j=1}^{N}\] and a simple random graph with adjacency matrix denoted by \[A_r^N = [\xi_{ij}]_{i,j=1}^N, \quad \xi_{ij} = \xi_{ji}, \quad \xi_{ii}=0, \quad \mathbb{P}(\xi_{ij}=1)=W(x_i,x_j) \ \text{for } i<j.\]

The following classic result holds \cite{LOVASZ2006933}.

\begin{thm}\label{teorem:convergencegraphons}
Let $W:[0,1]^2\to[0,1]$ be a graphon and let $\{x_i\}_{i=1}^{N}$ be a collection of points with $x_i$ selected from each interval $I_i=\left[\frac{i-1}{N},\frac{i}{N}\right)$ of the uniform partition of $[0,1]$. Define the step-function graphon $W_N:[0,1]^2\to\{0,1\}$ by
\[
W_N(x,y) = \xi_{ij}, \quad (x,y)\in I_i\times I_j,
\]
where $\xi_{ij}\in\{0,1\}$ are independent with $\mathbb{P}(\xi_{ij}=1)=W(x_i,x_j)$, and $\xi_{ij}=\xi_{ji}$, $\xi_{ii}=0$. Then
\[
\delta_\Box(W_N, W)\xrightarrow{N\to\infty} 0 \quad \text{a.s},
\]
where the cut metric is defined as
\[
\delta_\Box(U,W) = \inf_{\phi\in\mathcal{L}}\sup_{S,T\subseteq[0,1]}\left|\int_{S\times T}\bigl(U(\phi(x),\phi(y))-W(x,y)\bigr)\,dx\,dy\right|,
\]
and $\mathcal{L}$ denotes the space of Lebesgue-measurable bijections of $[0,1]$.
\end{thm}
Therefore, every graphon is the limit of simple graphs (in bijection with step-function graphons). 

%Since we are interested in community detection methods and Markovian processes on networks, the natural question is how the Laplacian of the random graphs \(A_r^N\) behaves as \(N\rightarrow \infty\), and whether the spectral information of these Laplacians is related to the spectral information of \(W\). 

\section{Ultrametric Graphons and hierarchical community networks}
We now initiate the theory of ultrametric graphons in order to model hierarchical community networks, that is, networks where the density of connections is organized in clusters, where each cluster could in turn contain distinguishable sub-clusters, and so on. Such networks have been studied in \cite{girvan2002community}, for example. The nested community structure in network theory is an important phenomenon appearing in many real networks, and this work aims to provide an analytical framework to study such networks and the dynamical phenomena on them as the number of vertices grows. Although the community structures produced by our framework are related to those of the hierarchical stochastic block model, our approach is developed entirely within the language of ultrametric spaces and graphon theory, independently of the SBM formalism. Our main aim is to relate the theory of hierarchical Laplacians with graphon theory.

Several technical lemmas and parts of certain proofs rely on, or are adapted from, results in \cite{BramburgerHolzer2023}, which will be cited at the appropriate moments throughout the text. We also note that the relationship between kernels $w$ (and their associated non-local operators) and their discretizations (together with their corresponding Laplacian matrices) has already been studied from the perspective of $p$-adic analysis in \cite{Zuniga2022, ZunigaNetworks}, independently of graphon theory.

\subsection{Finite ultrametrics on the unit interval}

Let $[0,1]$ be equipped with a finite family of nested partitions $\{\Upsilon_\ell : M\ge \ell\ge 1\}$, where $\Upsilon_1=\{[0,1]\}$ and each $\Upsilon_\ell$ is a finite partition of $[0,1]$ into intervals of the following form:
\[
[0,1]
= [a^{(\ell)}_0,a^{(\ell)}_1)
  \,\sqcup\, [a^{(\ell)}_1,a^{(\ell)}_2)
  \,\sqcup\, \cdots \,\sqcup\,
    [a^{(\ell)}_{k_\ell-1},a^{(\ell)}_{k_\ell}],
\]
where \(a^{(\ell)}_{k_\ell}=1\). Moreover, each interval $I\in\Upsilon_\ell$ is the disjoint union of $n_{\ell,I}\ge2$ intervals from $\Upsilon_{\ell+1}$. Any interval in \(\Upsilon_{\ell}\) has assigned the index \(\ell\), which will be referred to as the \emph{level} of the interval, and when we are only interested in the level of a given interval we use the notation \(I_{\ell}\) for an arbitrary interval at level \(\ell\). For every $x\in[0,1]$, denote by $I_\ell(x)\in\Upsilon_\ell$ the unique interval of level $\ell$ containing $x$. For every interval \(I_\ell\) we assign a \emph{height} \(h(I_\ell)>0\) such that \(h(I_\ell)<h(I_r)\) if \(I_\ell\subset I_r\).

We define the level of the pair \((x,y)\in [0,1]^2\) by
\[
\ell(x,y):=\max\{M\ge \ell\ge0 : y\in I_\ell(x)\}
\]
for $x\neq y$. Define the \emph{least common ancestor} interval as the unique interval \(I_{\ell(x,y)}\) such that \(x,y\in I_{\ell(x,y)}\)  denoted \(x\wedge y\) or \(\mathrm{LCA}(x,y)\).

We define the function \(d:[0,1]\times [0,1] \rightarrow \mathbb{R}^+\),
\begin{equation} \label{eq:ultrametricgraphon}
    d(x,y):=
\begin{cases}
0, & x=y,\\
h(I_{\ell(x,y)}), & x\neq y.
\end{cases}
\end{equation}
This setup allows us to construct infinitely many finite ultrametric spaces in the following way. Take a finite sample of points \(X:=\{x_n\}\subset [0,1]\) where at least two points of each \(I_M\) are sampled, with \(I_M \in \Upsilon_M\) and \(\Upsilon_M\) the last partition; then \((X,d)\) is a finite ultrametric space. The balls of this ultrametric space are in one-to-one correspondence with the intervals of \(\{\Upsilon_\ell : M\ge \ell\ge 0\}\). In terms of the topological tree, the intervals of the family of partitions correspond to the internal nodes, whereas the sampled points \(\{x_n\}\) are the leaves; see Figure~\ref{fig:infinitesimaltree}.

All terminology introduced for finite ultrametric spaces carries over to this setting, with intervals playing the role of internal nodes and sampled points playing the role of leaves. In particular, given two intervals \(I_k\) and \(I_m\), their least common ancestor, denoted \(I_k\wedge I_m\) or \(\mathrm{LCA}(I_k,I_m)\), is the maximal level interval containing both of them. The notation \(I_{F(m)}\) represents the unique interval at level \(m-1\) containing \(I_m\), that is, the father interval. Finally, \(\gamma_r(I_n)\) is the unique sequence of intervals \(I_i\), \(i=1,2,\ldots,n\), satisfying \(I_n\subset I_{n-1}\subset \cdots \subset I_1=[0,1]\).

\begin{figure}[H]
    \centering
\includegraphics[width=\textwidth]{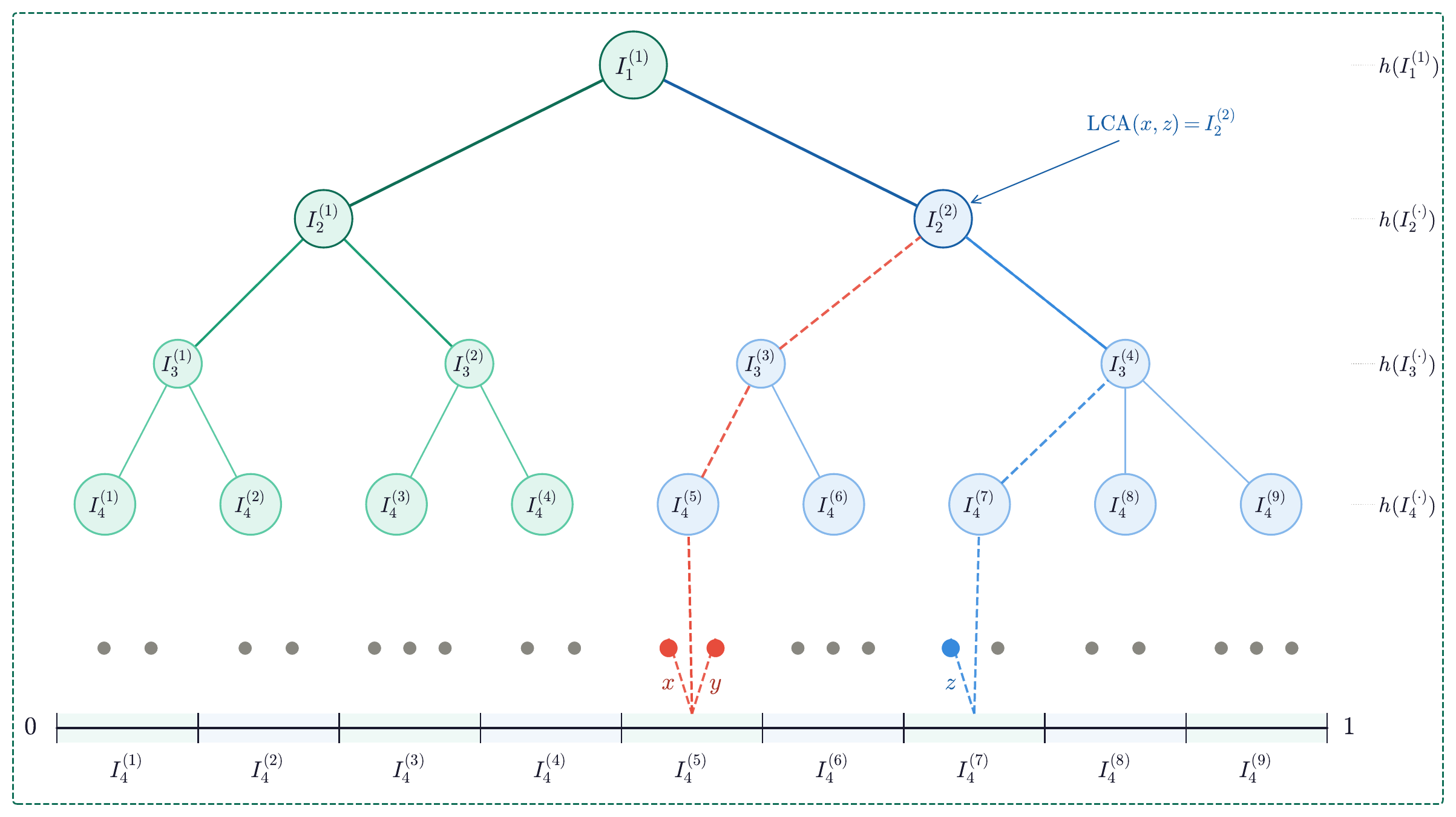}
    \caption{The intervals in the family of nested partitions $\{\Upsilon_\ell\}$ are identified with the internal nodes of a rooted tree. Points in $[0,1]$ are potential leaves; a finite ultrametric space $(X,d)$ is obtained by sampling at least two points per interval of the finest partition $\Upsilon_N$, at which point the sampled points become the leaves of the tree and the resulting balls are in bijection with the intervals of $\{\Upsilon_\ell\}$. Here, $x$ and $y$ belong to the same interval $I^{(5)}_4$, while $z$ belongs to $I^{(7)}_4$, so that $d(x,y) \leq d(x,z) = d(y,z)$, illustrating the ultrametric inequality. The heights $h(I^{(\cdot)}_\ell)$ are schematic; in general they vary across intervals of the same level.}
\label{fig:infinitesimaltree}
\end{figure}

\subsection{Ultrametric graphons and nested community networks.}
We begin with the main definition of this section.
\begin{definition}
    A graphon \(W:[0,1]^2\rightarrow [0,1]\) is called (finite-)ultrametric if
    \begin{equation*}
        W(x,y)=w(d(x,y)),
    \end{equation*}
    \noindent where \(d\) is the ultrametric attached to the family of partitions as in \eqref{eq:ultrametricgraphon} and \(w:[0,\infty)\rightarrow[0,1]\) is a positive function.
\end{definition}

We now describe which type of random networks comes from an ultrametric graphon. As mentioned before, we can construct a deterministic graph and a random graph by taking $N$ points uniformly distributed in $[0,1]$. Equivalently, given the ultrametric $d$ determined by the nested partitions $\{\Upsilon_\ell\}_{M \ge \ell\ge 1}$, we may sample points in a way adapted to this filtration and the Lebesgue measure denoted by \(\mu\): for the last level $\ell=M$, each interval $I_{M}^{(m)}\in\Upsilon_M$ is chosen to have a rational Lebesgue length $\mu(I_{M}^{(m)})$. We choose an integer \(N_k:=kN\), for \(k\geq 1\), as the number of sampled points, where $N$ equals the least common multiple of the denominators of all these rational lengths; we sample exactly $N_k\,\mu(I_{M}^{(m)})$ equidistant points from each interval $I_{M}^{(m)}$. This produces a deterministic set of sample locations $\{x_i\}$ whose distribution respects the relative sizes of the ultrametric intervals and is therefore compatible with the Lebesgue measure: since each interval $I_{M}^{(m)}$ receives exactly $N_k\,\mu(I_{M}^{(m)})$ sample points, the resulting empirical measure (normalized counting measure)
\[
\mu_{N_k} \;:=\; \frac{1}{N_k} \sum_{i=1}^{N_k} \delta_{x_i},
\]
assigns to intervals of the last level $I_{M}^{(k)}$ the mass
\[
\mu_{N_k}(I_{M}^{(k)}) 
    \;=\; \frac{N_k\,\mu(I_{M}^{(k)})}{N_k}
    \;=\; \mu(I_{M}^{(k)}).
\]
The sampling construction above is compatible with the standard graphon sampling. Indeed, by subdividing each interval \(I_M^{(k)}\) into \(N_k\,\mu(I_M^{(k)})\) equidistant sub-intervals \(I_M^{(k,i)}\) of equal Lebesgue measure
\[
\mu\bigl(I_M^{(k,i)}\bigr) = \frac{\mu(I_M^{(k)})}{N_k\,\mu(I_M^{(k)})} = \frac{1}{N_k},
\]
and selecting one point \(x_i\in I_M^{(k,i)}\), the resulting partition \(\{I_M^{(k,i)}\}\) is a uniform partition of \([0,1]\) into \(N_k\) intervals of equal length, which is precisely the standard graphon sampling construction of Section \ref{sec:graponintro}.

Moreover, $\mu_{N_k}$ agrees with the Lebesgue measure on all intervals of the partition $\Upsilon_\ell$: since any interval at level 
$\ell$ decomposes as a union of intervals at level $\ell+1$, and so on up to the last 
level, the measure of any interval satisfies
\[
\mu_{N_k}(I_\ell)=\sum_{I_M^{(k,i)}:\,x_i\in I_\ell} 
\mu\bigl(I_M^{(k,i)}\bigr)=\mu(I_\ell).
\]

In what follows, we adopt this construction, since it makes the analytical treatment more tractable; moreover, this implies that the graphs attached to this sampling method converge to the ultrametric graphon \(W\) in the sense of Theorem~\ref{teorem:convergencegraphons}. In particular, the size of each 
community in the resulting network is controlled by the Lebesgue measure 
$\mu(I_\ell)$ of the corresponding interval: larger intervals produce 
proportionally more vertices, so that $\mu$ encodes the relative sizes of the 
communities at every level of the hierarchy.\newline

For a given interval \(I_\ell\in \Upsilon_\ell\) define the set \(C(I_\ell):=\{I_{\ell+1}\in \Upsilon_{\ell+1}: I_\ell=\bigsqcup I_{\ell+1}\}\) as the set of children intervals of \(I_\ell\), following the underlying tree structure attached to the filtration. The number \(|C(I_\ell)|\) is referred to as the \emph{number of children} of the interval \(I_\ell\). For any \(x,y\in [0,1]\) in different intervals at the same level, \(x\in I_\ell^{(i)}\) and \(y\in I_\ell^{(j)}\), the probability of a possible edge is controlled by \(w(d(x,y))=w(h(I_{\ell(x,y)}))\); this parameter determines the density of connections between vertices belonging to different clusters. In this way, clusters become distinguishable through this density, and each level of the filtration generates a family of clusters whose inter-cluster connectivity is governed by the ultrametric distance on the unit interval.

Assume for example that \(w(h(I_\ell))\) increases as the heights decrease (i.e., as \(\ell\) increases): vertices in finer clusters are more densely connected among themselves. Conversely, the inter-cluster density between coarser clusters is lower, making them the most distinguishable ones in the hierarchy. The first level of clusters is determined by the number of children of \(I_1=[0,1]\), each of which decomposes further via \(|C(I_2^{(j)})|\), and so on. The communities of these networks are therefore \textbf{hierarchical and nested}.

\begin{example}
\label{example:graphonultra}
    We now fix a concrete ultrametric structure on the unit interval $[0,1]$ as a 
concrete illustrative example of a hierarchical community network. We start from 
a rooted tree with three children
\[
A,\; B,\; C,
\]
where the interval $A$ has two children $A_1,A_2$, the interval $B$ has two 
children $B_1,B_2$, and the interval $C$ has three children $C_1,C_2,C_3$. 
Up to this resolution, this gives the finite rooted tree
\[
\text{root} \longrightarrow \{A,B,C\}, \qquad
A\to\{A_1,A_2\},\; B\to\{B_1,B_2\},\; C\to\{C_1,C_2,C_3\}.
\]
Each of these seven intervals is further subdivided into two equal parts, 
producing the fourteen final-level intervals
\[
I_{A_1^a}, I_{A_1^b},\;
I_{A_2^a}, I_{A_2^b},\;
I_{B_1^a}, I_{B_1^b},\;
I_{B_2^a}, I_{B_2^b},\;
I_{C_1^a}, I_{C_1^b},\;
I_{C_2^a}, I_{C_2^b},\;
I_{C_3^a}, I_{C_3^b}.
\]
Thus, the unit interval admits the final decomposition
\[
[0,1)
= \bigsqcup_{X\in\{A_1,A_2,B_1,B_2,C_1,C_2,C_3\}}
   \bigl(I_{X^a}\sqcup I_{X^b}\bigr).
\]
The resulting tree structure is depicted in Figure \ref{fig:treegraphon}.

\begin{figure}[H]
\centering
\includegraphics[scale=0.3]{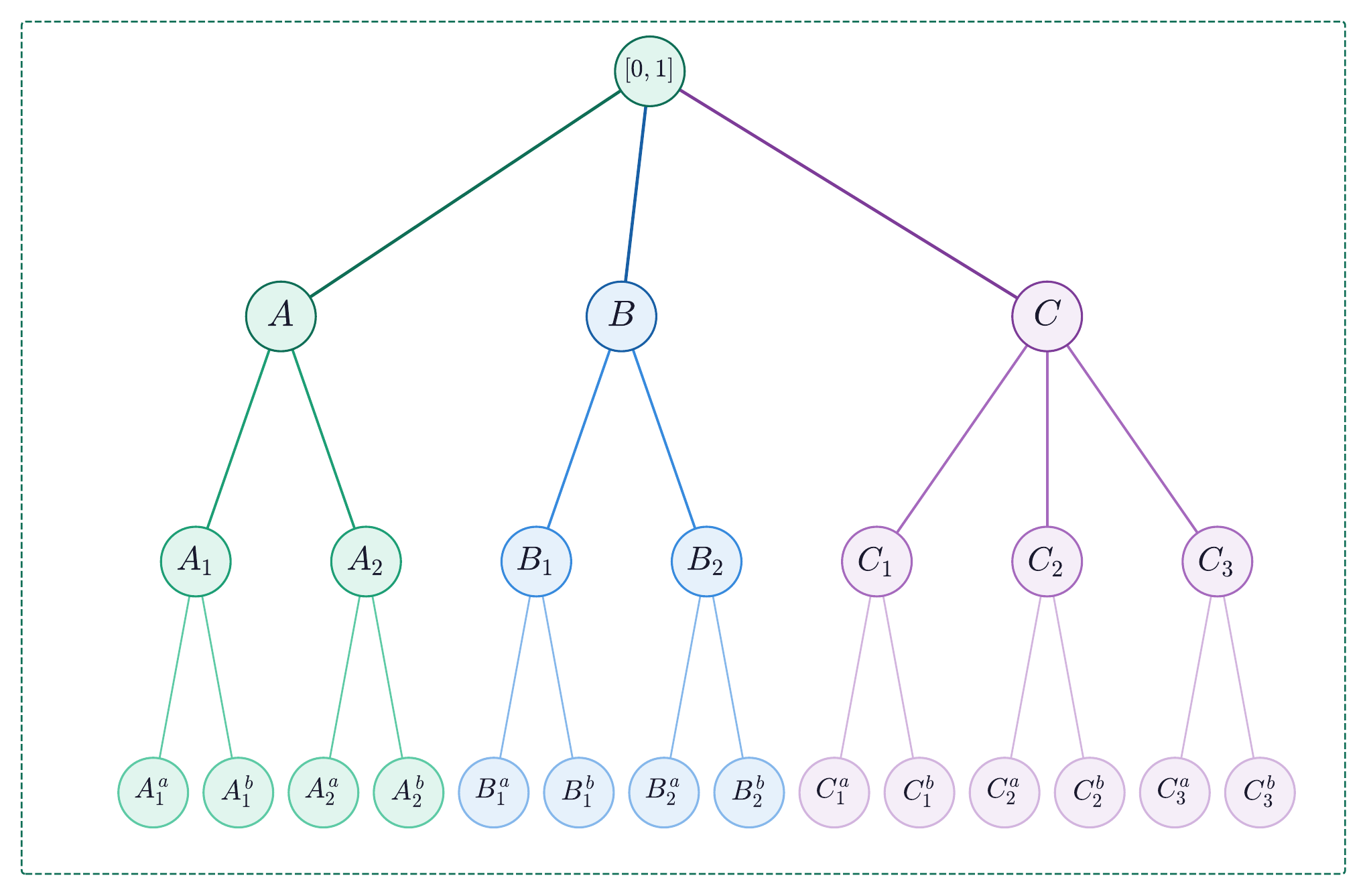}
\caption{Filtration with resolution up to level \(3\).}
\label{fig:treegraphon}
\end{figure}

We then define an ultrametric \(d:[0,1]\times[0,1]\to[0,\infty)\) that reflects this hierarchical structure, the numerical values are display in Figure \ref{fig:ultraexample}. 

\begin{figure}[H]
    \centering
    \includegraphics[scale=0.4]{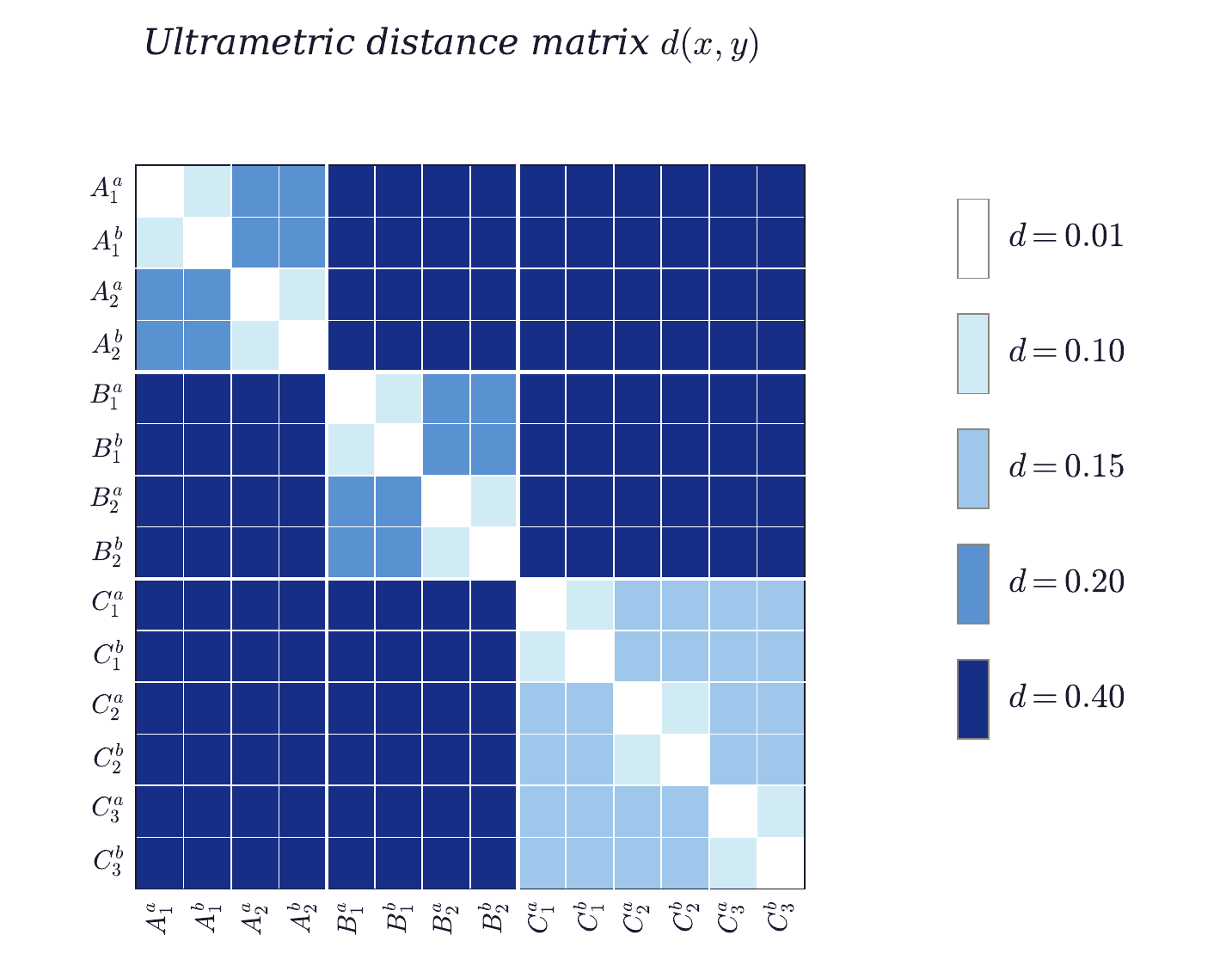}
    \caption{Heat map of distance function \(d\).}
    \label{fig:ultraexample}
\end{figure}

Given this ultrametric, we define an ultrametric graphon by prescribing a radial 
kernel
\[
w(d) = \exp\!\left(-\frac{d}{0.1}\right), \qquad d\geq0,
\]
and setting
\[
W(x,y) := w\bigl(d(x,y)\bigr)
        = \exp\!\left(-\frac{d(x,y)}{0.1}\right).
\]
Note that for points lying in the same interval at level $4$, one has 
$d(x,y)=0.01$, so the corresponding edge probability becomes
\[
W(x,y)=\exp\!\left(-\frac{0.01}{0.1}\right)\approx 0.9.
\]

\begin{figure}[H]
    \centering
    \includegraphics[width=0.9\linewidth]{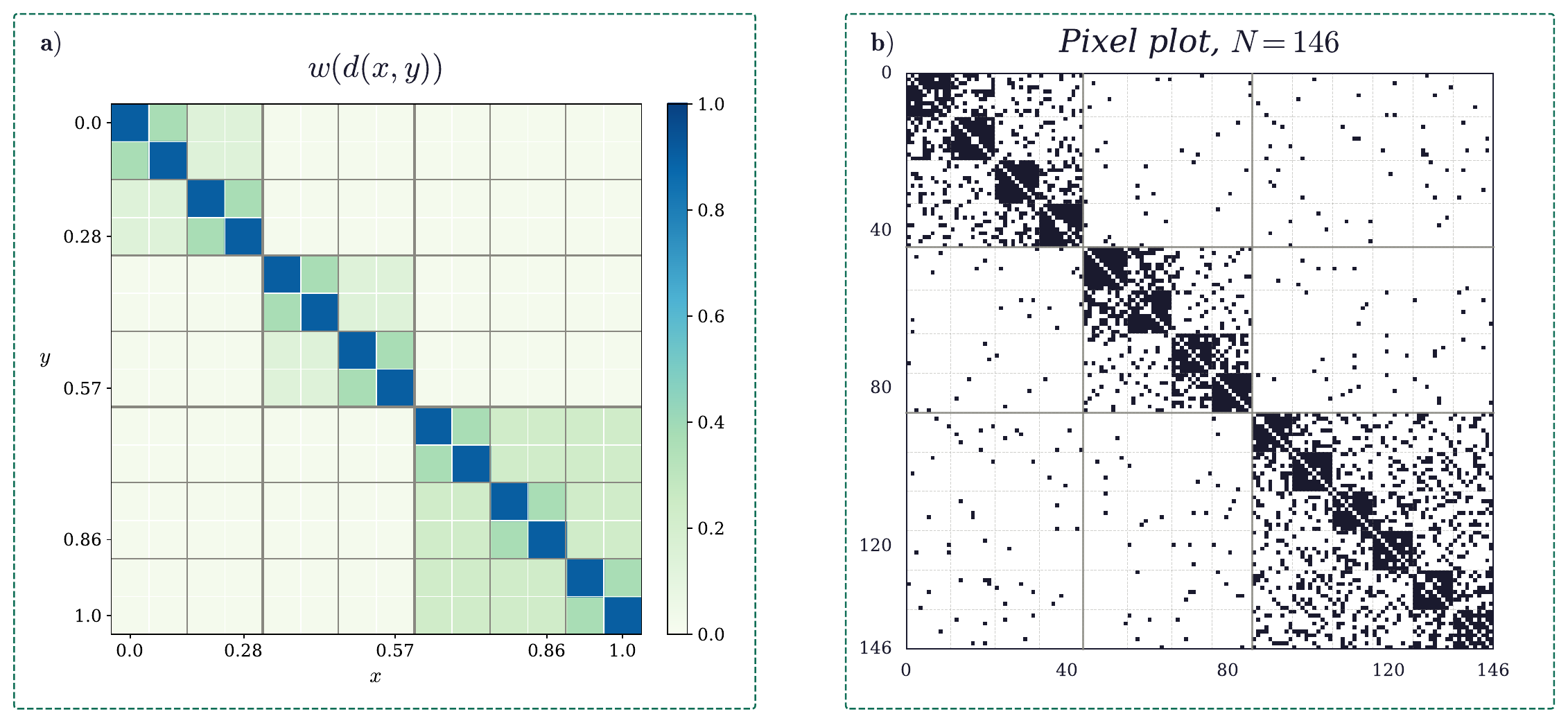}
    \caption{Ultrametric kernel and Adjacency graph. a) The heat map of the ultrametric exponential kernel and b) a random graph adjacency pixel plot of a realization with 146 vertices}
    \label{fig:randomgraphexample}
\end{figure}

Figure \ref{fig:randomgraphexample}, shows the corresponding heat map, the probability of connection between two pixels inside the yellow blocks is almost \(1\), the clusters generated by the partition are reflected in the pixel plot, here the density of the connections make all clusters up to resolution of level \(3\) distinguishable. 

The resulting random graphs exhibit a nested community structure, with strongly connected sub-communities inside each interval, intermediate connectivity within each branch \(A,B,C\), and very weak connectivity between different clusters. In Figure \ref{fig:combinatorialgraphcluster} below an schematic figure is presented. Formally, the result is a combinatorial graph, here we exaggerate the nodes depending in the level in order to distinguish the communities in the resulting network. 
\begin{figure}[H]
    \centering
    \includegraphics[width=0.8\linewidth]{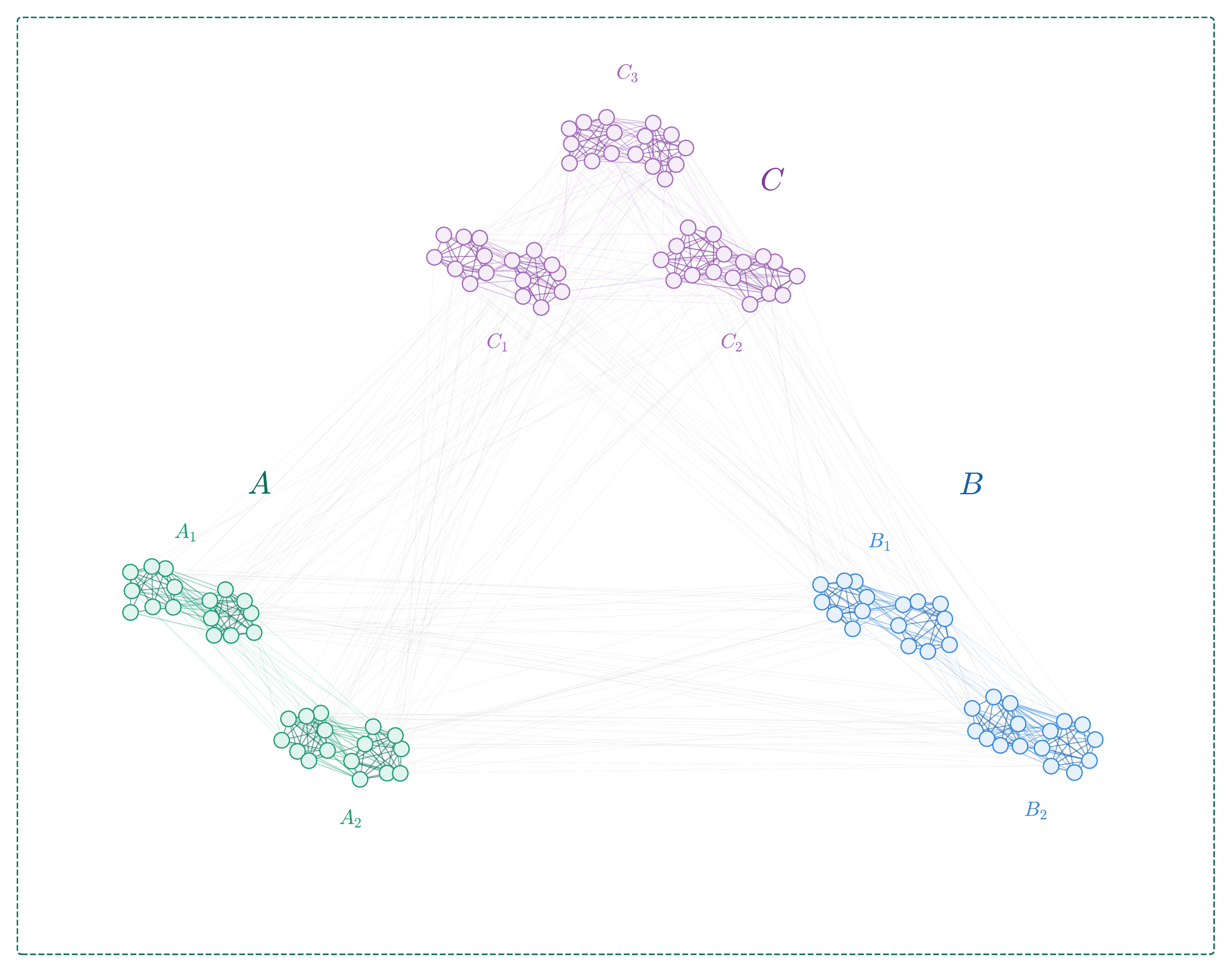}
    \caption{The resulting combinatorial graph attached to the adjacency pixel plot of Figure \ref{fig:randomgraphexample} b). The edges are prolonged according to the level of the filtration. }
    \label{fig:combinatorialgraphcluster}
\end{figure}

\end{example}

\section{Ultrametric graphons and graph Laplacians: The eigenvalue problem and the Fiedler matrix}

In this section we study the spectral properties of the Laplacians attached to 
the random graphs generated by the ultrametric graphons and how the topology of 
the underlying ultrametric tree relates to them in the limit. We also study the 
effectiveness of spectral community detection in terms of these results. First, 
we recall the sampling method presented in the previous section. Let $M$ be the 
last level of the partition. We take $N$ to be the LCM of the denominators of the lengths $\mu(I_{M}^{(k)})$, 
and sample $N_k\,\mu(I_{M}^{(k)})$ equi-spaced points $x_i\in I_{M}^{(k)}$, 
where $N_k=kN$, with $k\geq1$. Each point $x_i$ is sampled from an interval 
$I_{M}^{(k,i)}\subset I_{M}^{(k)}$ of Lebesgue measure $\mu\bigl(I_{M}^{(k,i)}\bigr)=\frac{1}{N_k}$, 
so that in total we have $N_k$ intervals $I_{M}^{(k,i)}$. 

As described before, the set 
of points $X_k:=\{x_i\}$ equipped with the function $d:[0,1]^2 \rightarrow 
\mathbb{R}^+$ induced by the partition is a finite ultrametric space. The matrices 
$A_d^{k}=[W(d(x_i,x_j))]_{i,j}$ and $A_r^{k}=[\xi_{ij}]_{i,j}$ have an 
attached Laplacian matrix
\[
L_d^{k}=A_d^{k}-D_d^{k},
\]
where $D_d^{k}$ is the degree matrix of $A_d^{k}$, that is, the diagonal matrix 
whose entries are the row sums of $A_d^{k}$. We define $L_r^{k}$ analogously. 
An important consequence of the fact that $A_d^{k}$ depends on a given 
ultrametric is the following.

\begin{lem}
The Laplacian matrix $L_d^{k}$ is a finite ultrametric Laplacian matrix.
\end{lem}

\begin{proof}
Consider the finite ultrametric space $(X_k,d,m)$ where $m$ is the counting 
measure, so that $m(x)=1$ for all $x\in X_k$. By 
Definition~\ref{def:ultrametriclaplacianmatrix}, the ultrametric Laplacian 
matrix of $(X_k,d,m)$ has entries
\[
L_{x,y} = \begin{cases} w(d(x,y)) & y\neq x,\\ 
-\sum_{z\neq x} w(d(x,z)) & y=x, \end{cases}
\]
which coincides precisely with $L_d^k = A_d^k - D_d^k$, since 
$W(x_i,x_j)=w(d(x_i,x_j))$.
\end{proof}

The eigenvalues and eigenprojectors of $L_d^k$ are therefore completely 
described by the topology of the ultrametric space, the measure, and the kernel 
$w$, which represents the density of connections between clusters. The Laplacian 
matrix $L_d^{k}$ is the matrix representation of the Laplacian operator attached 
to the measured ultrametric space $(X_k,d,m)$ where $m$ is the counting measure. 
The structure of the internal nodes of the associated topological tree is 
independent of $k$; this parameter only increases the number of leaves attached 
to the internal nodes at the last level. \newline

Since \(\mu_{M_k}=\mu\) on each interval of any partition \(\Upsilon_{\ell}\) the spectrum also depends on the size of each cluster, that is, on 
the Lebesgue measure of each interval of the family of partitions. By 
Proposition~\ref{prop:eigenvalueofultrametricoperator}, the nonzero eigenvalues 
of $L_d^{k}$ are given by
\[
\lambda(I_n)=-\sum_{I_m\in \gamma_r(I_n)} m(I_m)
\Bigl(w\bigl(h(I_m)\bigr)-w\bigl(h(I_{F(m)})\bigr)\Bigr),
\]
and therefore
\[
\frac{\lambda(I_n)}{N_k}=-\sum_{I_m\in \gamma_r(I_n)} \mu(I_m)
\Bigl(w\bigl(h(I_m)\bigr)-w\bigl(h(I_{F(m)})\bigr)\Bigr).
\]
Each cluster $I_n$ has an associated eigenvalue; the behavior of 
$\frac{\lambda(I_n)}{N_k}$ depends on three main factors. First, the multiplicity 
of the eigenvalue is $|C(I_n)|-1$, where $|C(I_n)|$ is the number of children 
of $I_n$. Secondly, its value depends on the Lebesgue measure $\mu(I_m)$ of the 
intervals containing $I_n$, which represents the proportion of vertices belonging 
to each cluster. Thirdly, it depends on the difference between the inter-cluster 
edge densities at consecutive levels.

Since $L_d^{k}$ is an ultrametric Laplacian matrix, we can also give a closed 
expression for the spectral projectors. Let $I_n$ be an interval of the family 
of partitions. By \eqref{eq:projectorgraphon}, the projection kernel attached to 
$\lambda(I_n)$ satisfies $E_{I_n}(x,y)=0$ if $x\notin I_n$ or $y\notin I_n$, 
and if $x,y\in I_n$ then
\begin{equation}\label{eq:projectorgraphon}
    E_{I_n}(x,y)=\frac{\mathbf{1}_{I_j}(x)}{m(I_j)}-\frac{\mathbf{1}_{I_n}(x)}{m(I_n)},
\end{equation}
where $I_j\subset I_n$ is the unique child interval of $I_n$ containing $y$.

The next result shows that after normalization, the eigenvalues of $L_d^{k}$ 
and $L_r^{k}$ are arbitrarily close with high probability as the number of 
vertices $N_k$ increases. As a consequence, the explicit solvability of the 
ultrametric Laplacian eigenvalue problem provides an analytical tool to study 
how the topology of a random hierarchical community network influences its 
dynamical properties. This result follows the argument sketched in the proof 
of Theorem~2.8 of \cite{BramburgerHolzer2023}, and is an adaptation of 
Lemma~2.6 therein.

\begin{prop}\label{prop:eigenvalueaprox}
    Let $N_k\geq 2$ and consider the graph Laplacian $L_d^{k}$ with eigenvalues 
    $0=\lambda_1\geq \lambda_2 \geq \dots \geq \lambda_{N_k}$ and its attached 
    random Laplacian $L_r^{k}$ with eigenvalues $0=\hat{\lambda}_1\geq\hat{\lambda}_2
    \geq \dots \geq \hat{\lambda}_{N_k}$. For all $\gamma \in (0,\frac{1}{2})$, 
    there is a constant $C=C(\gamma)$, independent of $k$, such that 
    \[
    \left|\frac{\lambda_i}{N_k}-\frac{\hat{\lambda}_i}{N_k}\right|\leq N_k^{\gamma-\frac{1}{2}},
    \]
    with probability at least $1-2N_ke^{-CN_k^{2\gamma}}$.
\end{prop}

\begin{proof}
    By Lemma~2.6 of \cite{BramburgerHolzer2023}, for all $\gamma \in (0,\frac{1}{2})$, 
    there exists a constant $C=C(\gamma)$, independent of $N_k$, such that
    \[
    \mathbb{P}\!\left(\|L_r^k-L_d^k\|\geq N_k^{\frac{1}{2}+\gamma}\right)
    \leq 2N_k e^{-CN_k^{2\gamma}}.
    \]
    For any pair $M_1$ and $M_2$ of $d\times d$ Hermitian matrices with eigenvalues 
    $\mu_1\geq \cdots \geq \mu_d$ and $\nu_1 \geq \cdots \geq \nu_d$ respectively, 
    Weyl's inequality gives
    \[
    |\mu_j-\nu_j|\leq \max\{|\rho_m|,|\rho_M|\}\leq\|M_1-M_2\|,
    \]
    where $\rho_m$ and $\rho_M$ are the minimum and maximum eigenvalues of $M_1-M_2$ 
    respectively, and $\|\cdot\|$ denotes the operator norm. Hence, for the eigenvalues 
    $\lambda_i$ and $\hat{\lambda}_i$ of $L_d^k$ and $L_r^k$ it holds that
    \[
    |\lambda_i-\hat{\lambda}_i|\leq \|L_r^{k}-L_d^{k}\|\leq N_k^{\frac{1}{2}+\gamma}.
    \]
    Dividing both sides by $N_k$ yields
    \[
    \left|\frac{\lambda_i}{N_k}-\frac{\hat{\lambda}_i}{N_k}\right|\leq N_k^{\gamma-\frac{1}{2}},
    \]
    with probability at least $1-2N_ke^{-CN_k^{2\gamma}}$, as desired.
\end{proof}

We show that the deterministic projection kernels of $L_d^k$ are good approximations 
of the projection kernels attached to $L_r^k$. The following proposition is a 
consequence of a variant of the Davis--Kahan theorem presented in 
\cite{daviskahntheorem}, and is an adaptation of point~$iii)$ in Theorem~2.8 
of \cite{BramburgerHolzer2023}.\begin{prop}\label{prop:eigenvectoraproximation}
    Let $L_d^k$ be the deterministic graph Laplacian matrix with indexed eigenvalues 
    $0=\lambda_1> \lambda_2\geq\cdots\geq \lambda_{N_k}$ and let $L_r^{k}$ be its 
    attached random Laplacian with eigenvalues $0=\hat{\lambda}_1\geq\hat{\lambda}_2
    \geq \dots \geq \hat{\lambda}_{N_k}$. Denote by $\lambda_{n_i}$ the eigenvalues 
    satisfying $\lambda_{n_i}=\lambda(I_n)$ for $i=1,\ldots,m$, where $m$ is the 
    multiplicity of $\lambda(I_n)$. If $V=[v_1,\ldots,v_m]\in \mathbb{R}^{N_k\times m}$ 
    and $\hat{V}=[\hat{v}_1,\ldots,\hat{v}_{m}]\in \mathbb{R}^{N_k\times m}$ have 
    orthonormal columns satisfying $L_d^{k}v_i=\lambda_{n_i}v_i$ and 
    $L_r^{k}\hat{v}_i=\hat{\lambda}_{n_i}\hat{v}_i$, then there exists $\delta>0$, 
    and for any $\gamma \in (0,\frac{1}{2})$ there is a constant $C=C(\gamma)$ such that
    \[
    \|\hat{V}\hat{V}^{T}-VV^{T}\|_F\leq \frac{2\sqrt{2}\,m}{\delta}N_k^{\gamma-\frac{1}{2}},
    \]
    with probability at least $1-2N_ke^{-CN_k^{2\gamma}}$.
\end{prop}
\begin{proof}
    Let $\delta>0$ be a radius around $\lambda(I_n)$ such that
    \[
    \min\{\lambda_n-\lambda_n^{-},\lambda_n^{+}-\lambda_n\}\geq N_k\delta,
    \]
    where $\lambda_n^+$ is the smallest eigenvalue satisfying $\lambda_n^+>\lambda_n$ 
    and $\lambda_n^{-}$ is the largest eigenvalue satisfying $\lambda_n^{-}<\lambda_n$. 
    By the variant of the Davis--Kahan theorem presented in Theorem~2 of 
    \cite{daviskahntheorem}, there exists an orthogonal matrix $\hat{O}\in \mathbb{R}^{m\times m}$ 
    such that
    \[
    \|\hat{V}\hat{O}-V\|_F\leq \frac{\sqrt{2m}\|L_d^k-L_r^k\|}
    {\min\{\lambda_n-\lambda_n^{-},\lambda_n^{+}-\lambda_n\}}
    \leq \frac{\sqrt{2m}}{\delta}N_k^{\gamma-\frac{1}{2}},
    \]
    with probability at least $1-2N_ke^{-CN_k^{2\gamma}}$, where the last inequality 
    follows from
    \[
    \mathbb{P}\!\left(\|L_r^k-L_d^k\|\geq N_k^{\frac{1}{2}+\gamma}\right)
    \leq 2N_k e^{-CN_k^{2\gamma}}.
    \]
    Since the Frobenius norm is sub-multiplicative, invariant under orthogonal 
    transformations, and invariant under transposition, we have
    \begin{equation*}
        \begin{split}
            \|\hat{V}\hat{V}^{T}-VV^{T}\|_F
            =\|(\hat{V}\hat{O})(\hat{V}\hat{O})^{T}-VV^{T}\|_F
            &\leq \|\hat{V}\hat{O}\|_F\|(\hat{V}\hat{O})^T-V^T\|_F
            +\|V^{T}\|_F\|\hat{V}\hat{O}-V\|_F\\
            &= 2\sqrt{m}\|\hat{V}\hat{O}-V\|_F \\
            &\leq \frac{2\sqrt{2}\,m}{\delta}N_k^{\gamma-\frac{1}{2}},
        \end{split}
    \end{equation*}
    which completes the proof.
\end{proof}
We now specialize the preceding results to the ultrametric graphon setting. 
Substituting the explicit eigenvalue formula of 
Proposition~\ref{prop:eigenvalueofultrametricoperator} into 
Proposition~\ref{prop:eigenvalueaprox} yields the following.

\begin{cor}\label{cor:eigenvalueultrametricgraphon}
Let $N_k\geq 2$ and let $W$ be an ultrametric graphon. Under the notation of 
Proposition~\ref{prop:eigenvalueaprox}, for all $\gamma\in(0,\frac{1}{2})$ there 
is a constant $C=C(\gamma)$, independent of $k$, such that
\[
\left|\frac{|\hat{\lambda}_i|}{N_k}-\sum_{I_m\in\gamma_r(I_n)}\mu(I_m)
\Bigl(w\bigl(h(I_m)\bigr)-w\bigl(h(I_{F(m)})\bigr)\Bigr)\right|
\leq N_k^{\gamma-\frac{1}{2}},
\]
with probability at least $1-2N_ke^{-CN_k^{2\gamma}}$, where the sum runs over 
the unique ancestral chain $\gamma_r(I_n)$ of the interval $I_n$ associated 
with $\hat{\lambda}_i$.
\end{cor}
As a first consequence, with high probability, if the multiplicity of the 
deterministic Laplacian for this eigenvalue is $m-1>0$, then we expect to 
have $m-1$ normalized eigenvalues $\frac{\hat{\lambda}_2}{N_k},\ldots,
\frac{\hat{\lambda}_m}{N_k}$ close to the value $-w(h([0,1]))$. If 
$\lambda_2-\lambda_3>0$ is large enough, we can distinguish the first $m$ 
eigenvalues as a well-separated cluster, in which case the number of 
communities at the first level is $m+1$. Since
\[
\lambda_2 - \lambda_3 = m(I_j)\bigl(w(h(I_j)) - w(h([0,1]))\bigr),
\]
where $I_j$ is a child interval of $[0,1]$, the separability of the 
first $m$ eigenvalues depends on two factors: the size of the sub-clusters, 
measured by $m(I_j)$, and the change in connection density between consecutive 
levels, measured by $w(h(I_j)) - w(h([0,1]))$.

It is known that spectral gaps govern the accuracy of spectral community detection 
\cite{communitydetection}; this result extends that principle to hierarchical 
nested communities. More generally, if $I_n$ is a cluster with child interval 
$I_i \subset I_n$, then
\[
\lambda(I_n) - \lambda(I_i) = m(I_i)\bigl(w(h(I_i)) - w(h(I_n))\bigr).
\]
There is therefore a detectability limit at every level of the hierarchy, 
controlled by the size and density of the clusters: once $w(h(I_i)) \to 
w(h(I_n))$, the communities at that level become indistinguishable through the spectrum of the network. In Figure \ref{fig:spectral_convergence} we present the behavior of the spectral convergence for the ultrametric graphon of Example \ref{example:graphonultra}. A few comments are in order regarding this panel. Note that the separation of the eigenvalues sharpens as $N_k$ increases. It is also worth noting the inherent limitation of the information that the spectrum alone can reveal about the community structure. Since clusters $A$ and $B$ are equivalent, their associated sub-trees are isomorphic, or, equivalently, the community densities coincide, their deterministic eigenvalues necessarily agree, and both fall within the same band $\lambda_i / N_k$. The detection of nested communities is therefore constrained by the density structure of the network. On the other hand, notice that the deepest communities are cliques, giving rise to sub-graphs isomorphic to the classical Erdős–Rényi model and producing its characteristic signature in the spectrum.

\begin{figure}[H]
    \centering
    % Fila 1: Pixel plots
    \includegraphics[width=\textwidth]{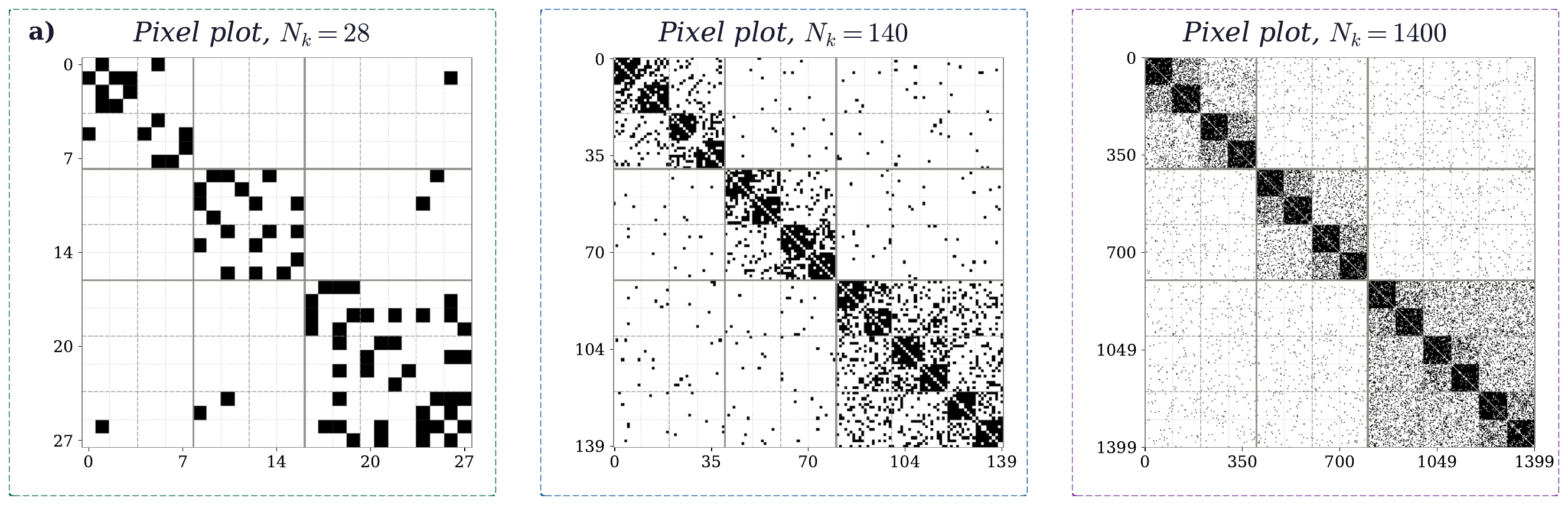}
    
    \vspace{0.4em}
    
    % Fila 2: Eigenvalue scatter plots
    \includegraphics[width=\textwidth]{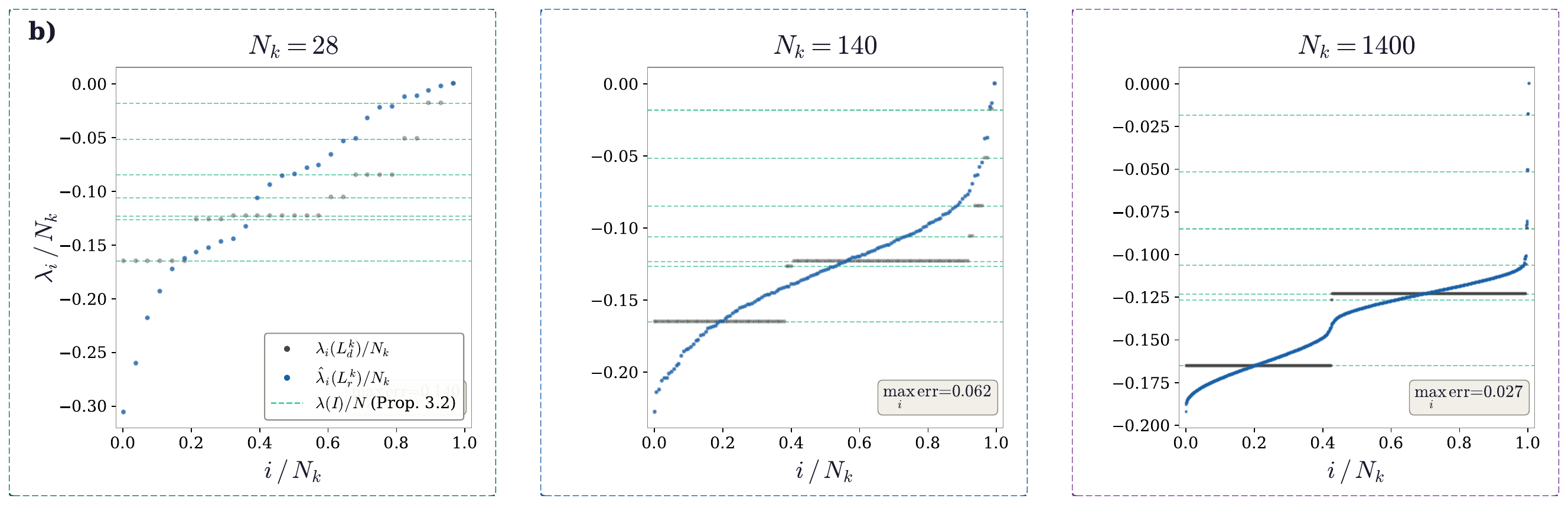}
    
    \vspace{0.4em}
    
    % Fila 3: Paired eigenvalue convergence
    \includegraphics[width=\textwidth]{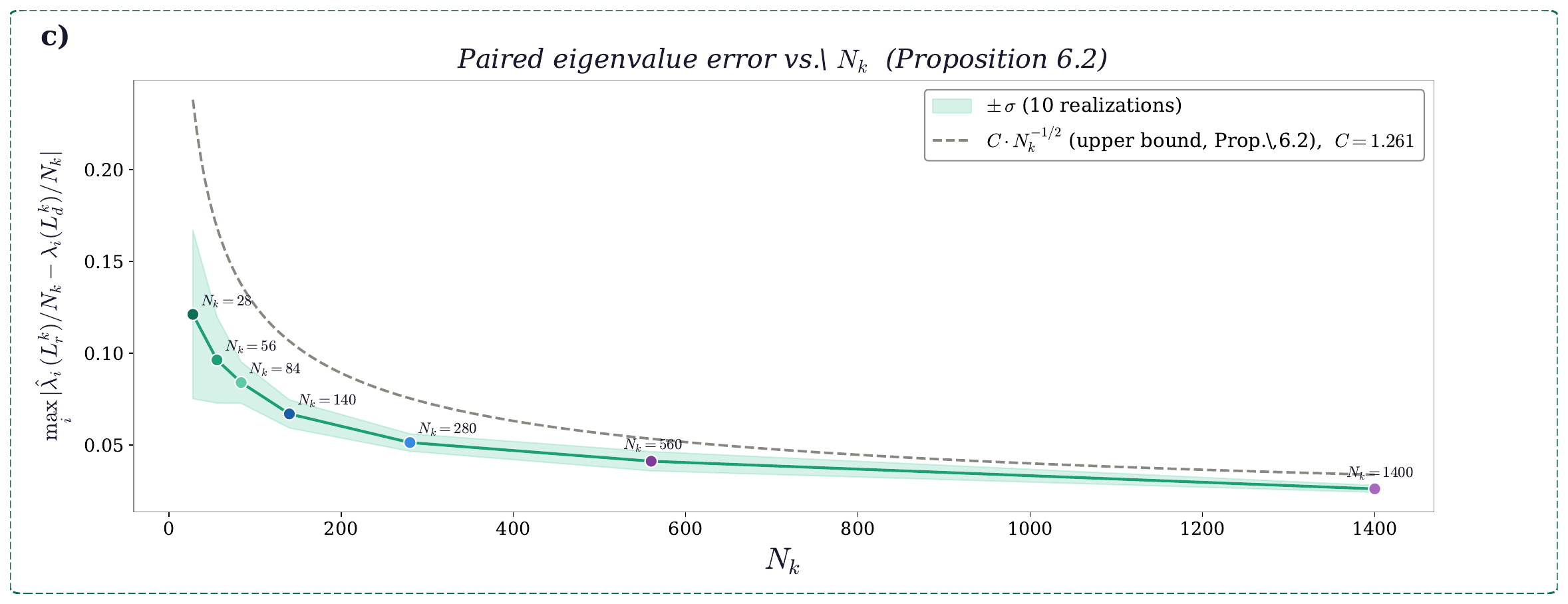}
    
    \caption{Spectral convergence for the ultrametric graphon of Example \ref{example:graphonultra} 
    with kernel $W(x,y) = e^{-d(x,y)/0.1}$ and uniform grid $x_i = i/N_k$.
    \textbf{Top row:} adjacency pixel plots $A_r^k$ for $N_k = 28, 140, 1400$.
    \textbf{Middle row:} normalized paired eigenvalues $\hat\lambda_i(L_r^k)/N_k$ 
    (colored) converging to $\lambda_i(L_d^k)/N_k$ (grey), with theoretical 
    values from Proposition~3.2 shown as dashed lines.
    \textbf{Bottom row:} maximum paired eigenvalue error 
    $\max_i|\hat\lambda_i/N_k - \lambda_i/N_k|$ vs $N_k$, 
    averaged over 10 realizations, with upper bound $C\cdot N_k^{-1/2}$ 
    from Proposition~6.2.}
    \label{fig:spectral_convergence}
\end{figure}

To state the main result of this section we need some preliminary results. 
First, to motivate the result, we ask the following: assuming the first $m$ 
communities are detected via the first $m-1$ nonzero eigenvalues, do the 
eigen-projectors capture the edge structure connecting the clusters? \newline

A classical method of clustering in networks is given by the Fiedler vector. 
The second largest eigenvalue of the Laplacian $L$ of a network, denoted 
$\lambda_2(L)$, is called the \emph{algebraic connectivity}. It satisfies 
$\lambda_2(L)<0$ if and only if the underlying graph is connected. The 
eigenvector corresponding to $\lambda_2(L)$ is known as the \emph{Fiedler 
vector}, typically denoted by $y$ \cite{Fiedler1973}.

A variational characterization of the algebraic connectivity is given by
\[
\lambda_2(L)
=
\max_{\substack{x^\top x = 1 \\ \mathbf{1}^\top x = 0}}
x^\top Lx =
-\min_{\substack{x^\top x = 1 \\ \mathbf{1}^\top x = 0}} \sum_{(i,j)\in E} (x_i-x_j)^2.
\]
A standard spectral clustering procedure for detecting two communities 
\cite{PhaseTransitionsSpectralCD} consists of the following steps:
\begin{enumerate}
    \item Compute the Laplacian $L = A - D$.
    \item Compute the Fiedler vector $y$ associated with $\lambda_2(L)$.
    \item Apply $k$-means clustering to the entries of $y$ to partition 
          the nodes into two groups.
\end{enumerate}
The sign pattern of the Fiedler vector gives a direct criterion for bi-partitioning a graph. We highlight a complementary viewpoint based on its rank-one projector. Under an idealized \emph{perfect separation} regime, 
assume the vertex set splits into two clusters $C_1, C_2$ such that 
$y_i > 0$ for $i \in C_1$ and $y_i < 0$ for $i \in C_2$, where $y$ is 
the unit-norm Fiedler vector. Consider the associated projector
\[
E_F := yy^\top.
\]
Then $(E_F)_{ij} = y_iy_j > 0$ whenever $i,j$ belong to the same cluster, 
and $(E_F)_{ij} < 0$ whenever $i$ and $j$ lie in different clusters. Thus, 
the sign of $E_F$ induces a binary labeling on vertex pairs and, in 
particular, on edges.

The ultrametric spectral analysis gives a generalization of this via the 
following corollary, which specializes Proposition~\ref{prop:eigenvectoraproximation} 
to the ultrametric graphon setting.
\begin{cor}\label{cor:eigenvectorultrametricgraphon}
Let $W$ be an ultrametric graphon with its attached random Laplacian 
$L_r^{k}$ with eigenvalues $0=\hat{\lambda}_1\geq\hat{\lambda}_2\geq 
\dots \geq \hat{\lambda}_{N_k}$. Denote by $\lambda_{n_i}$ the eigenvalues 
of the deterministic Laplacian satisfying $\lambda_{n_i}=\lambda(I_n)$ for 
$i=1,\ldots,m$, where $m$ is the multiplicity of $\lambda(I_n)$. If 
$V=[v_1,\ldots,v_m]\in \mathbb{R}^{N_k\times m}$ and 
$\hat{V}=[\hat{v}_1,\ldots,\hat{v}_{m}]\in \mathbb{R}^{N_k\times m}$ have 
orthonormal columns satisfying $L_d^{k}v_i=\lambda_{n_i}v_i$ and 
$L_r^{k}\hat{v}_i=\hat{\lambda}_{n_i}\hat{v}_i$, then there exists $\delta>0$, 
and for any $\gamma \in (0,\frac{1}{2})$ there is a constant $C=C(\gamma)$ 
such that
\[
\|\hat{V}\hat{V}^{T}-E_{I_n}\|_F\leq \frac{2\sqrt{2}\,m}{\delta}N_k^{\gamma-\frac{1}{2}},
\]
with probability at least $1-2N_ke^{-CN_k^{2\gamma}}$, where
\[
E_{I_n}(x,y)=-\frac{1}{m(I_n)}
\]
if $x,y$ belong to different child clusters of $I_n$, and
\[
E_{I_n}(x,y)=\frac{1}{m(I_j)}-\frac{1}{m(I_n)}
\]
if $x,y$ belong to the same child cluster $I_j$ of $I_n$.
\end{cor}

Therefore, for sufficiently large $N_k$ we expect a sign structure in the 
entries of $\hat{V}\hat{V}^{T}$ attached to a random network sampled from 
an ultrametric graphon: if $\lambda(I)$ is the eigenvalue of a cluster $I$ 
decomposed into child clusters $C_n$ for $n=1,\ldots,m$, then with high 
probability the sign pattern of $\hat{V}\hat{V}^{T}$ reveals the community 
structure, with $(\hat{V}\hat{V}^{T})_{ij}>0$ if $i,j \in C_n$ for some 
cluster, and $(\hat{V}\hat{V}^{T})_{ij}<0$ if $i,j$ belong to different 
clusters. In particular, if the ultrametric graphon has a single hierarchical 
level and two communities at the first level, the sign pattern of the Fiedler 
vector of a sampled random network separates the clusters with high probability. The following theorem formalizes this observation and establishes the sign structure of the empirical spectral projector as a rigorous community detection criterion.
\begin{thm}[Sign structure of spectral projectors for ultrametric graphons] 
\label{thm:signstructureofspectralprojectors}
Let $W$ be an ultrametric graphon and let $L_d^k$ denote its associated 
deterministic ultrametric Laplacian at resolution $k$. Let $I$ be an internal 
node of the ultrametric tree associated with $W$, with child clusters
\[
C_1,\dots,C_m,
\]
and let $\lambda(I)$ be the eigenvalue of $L_d^k$ associated with this node, 
with multiplicity $m-1$. Let $L_r^k$ be the random Laplacian obtained by 
sampling a network with $N_k$ vertices from $W$, and let
\[
\hat{V} \in \mathbb{R}^{N_k \times (m-1)}
\]
be a matrix whose columns form an orthonormal basis of the eigenspace of 
$L_r^k$ associated with the $m-1$ eigenvalues closest to $\lambda(I)$. 
Denote by $\hat{E} := \hat{V}\hat{V}^\top$ the corresponding spectral 
projector. Then, for sufficiently large $N_k$, with high probability, the 
sign pattern of the entries of $\hat{E}$ reveals the community structure 
induced by the child clusters $\{C_n\}_{n=1}^m$ in the following sense:
\begin{itemize}
\item if $i,j \in C_n$ for some $n$, then $(\hat{E})_{ij} > 0$,
\item if $i \in C_n$ and $j \in C_{n'}$ with $n \neq n'$, then $(\hat{E})_{ij} < 0$.
\end{itemize}
Moreover, in the special case where the ultrametric graphon has a single 
hierarchical level with exactly two clusters, the eigenvalue $\lambda(I)$ 
has multiplicity one, and the Fiedler vector of the sampled random Laplacian 
$L_r^k$ separates the two clusters by its sign with high probability.
\end{thm}

\begin{figure}[H]
    \centering
    \includegraphics[width=\textwidth]{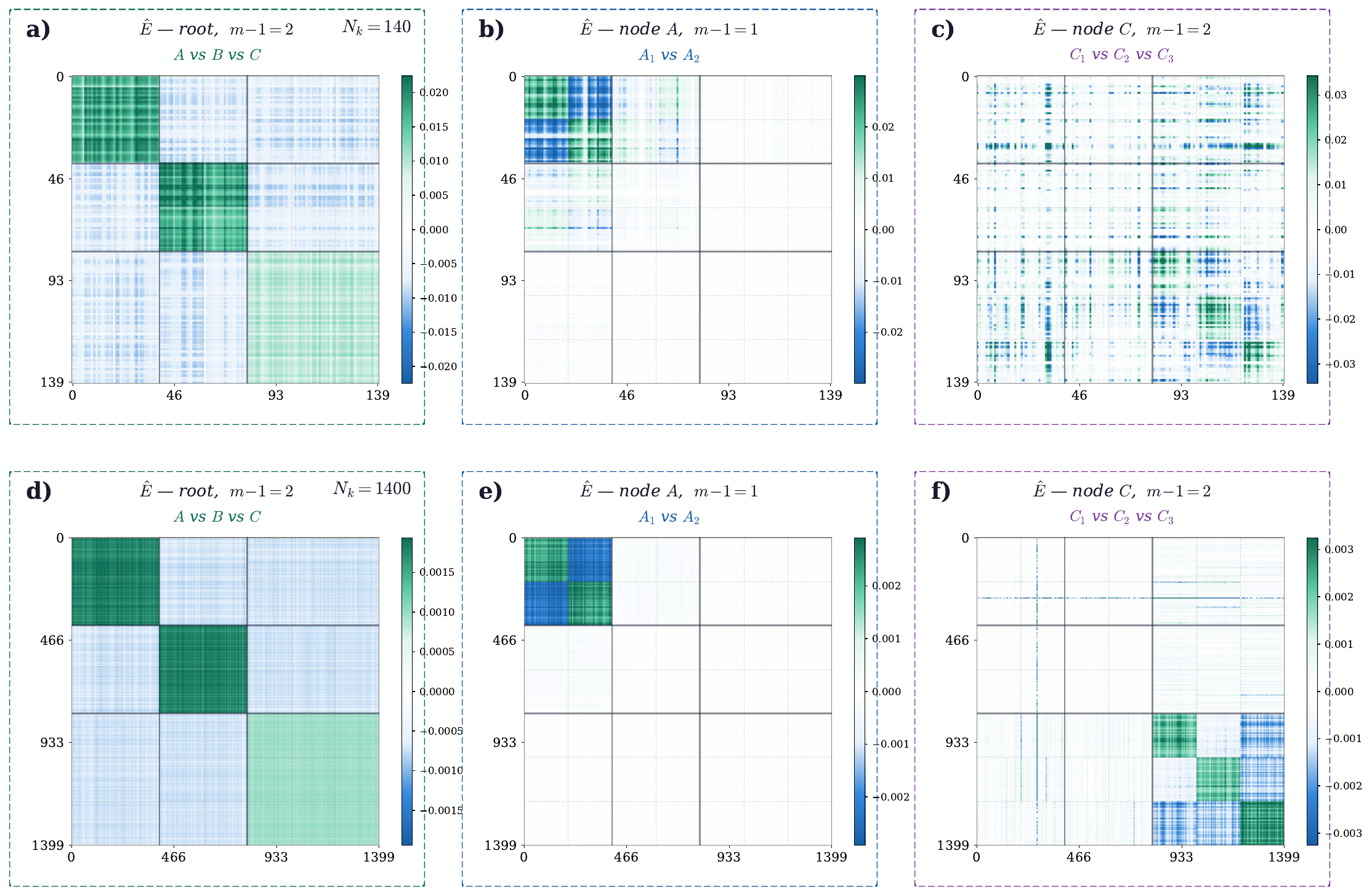}

    \caption{Sign structure of the empirical spectral projectors 
    $\hat{E} = \hat{V}\hat{V}^\top$ for three internal nodes of the 
    ultrametric tree of Example \ref{example:graphonultra}, sampled from $L_r^k$ 
    (Theorem~\ref{thm:signstructureofspectralprojectors}).
    The colormap ranges from blue (negative entries, $\hat{E}_{ij} < 0$, 
    nodes in different child clusters) to green (positive entries, 
    $\hat{E}_{ij} > 0$, nodes in the same child cluster).
    \textbf{Top row} ($N_k = 140$): \textbf{a)} the projector associated with the root 
    detects $A$ vs $B$ vs $C$ very good; \textbf{b) }the projector for node $A$ 
    separates $A_1$ vs $A_2$; \textbf{c)} the projector for node $C$ shows partial 
    detection of $C_1$ vs $C_2$ vs $C_3$, reflecting the small spectral 
    gap $\delta$ at this level (Corollary~\ref{cor:eigenvectorultrametricgraphon}).
    \textbf{Bottom row} ($N_k = 1400$): as $N_k$ increases all three 
    projectors \textbf{d)-f)}  converge to the theoretical values $E_I$ given by 
    Proposition \ref{prop:eigenprojectorsultra}, with the block sign structure becoming sharply 
    visible across all hierarchy levels.}
    \label{fig:projector_signs}
\end{figure}

\section{Spectral threshold in spectral community detection}

%We now want to generalize the results presented in \cite{PhaseTransitionsSpectralCD}, where a phase transition was shown when the community detection algorithm described above using the Fiedler vector is employed, we descried the results of that work: Let $G$ be an undirected graph with $n=n_1+n_2$ vertices, composed of two connected subnetworks $G_1,G_2$ of sizes $n_1,n_2$. Let \(L=-L_{-}\) be the graph Laplacian of \(G\) and $L^{1}=-L_{-}^{1},L^{2}=-L_{-}^2$  the graph Laplacians of the subgraphs $G_1,G_2$ respectively. Inter-community edges are added independently with probability $p\in[0,1]$. Let $y=(y_1^\top,y_2^\top)^\top$ be the Fiedler vector of $L_{-}$, normalized by $\mathbf{1}_{n_1}^\top y_1+\mathbf{1}_{n_2}^\top y_2=0$ and $\|y\|_2=1$, with eigenvalue \(\lambda_2(L_{-})\). Assume $n_1,n_2\to\infty$ with $n_1/n_2\to c\in(0,\infty)$. 

%With probability one, a phase transition occurs at a critical value $p^*$:
%\begin{itemize}
%\item \emph{(Case 1: detectable)} If $p<p^*$, then $y_1,y_2$ converge to constant
%vectors with opposite signs and spectral clustering correctly separates $G_1$ and $G_2$ and 
%\[\frac{\lambda_2(L_{-})}{n}\rightarrow p,\quad a.s.\]
%\item \emph{(Case 2: undetectable)} If $p>p^*$, then $\mathbf{1}_{n_1}^\top y_1\to0$ and $\mathbf{1}_{n_2}^\top y_2\to0$.
%\end{itemize}

%In the balanced case $n_1=n_2=n/2$, the critical threshold satisfies
%\[
%p^* \;\xrightarrow{\mathrm{a.s.}}\;
%\frac{\lambda_2(L_{-}^{1})+\lambda_2(L_{-}^{2})-|\lambda_2(L_{-}^{1})-\lambda_2(L_{-}^{2})|}{n}.
%\]

In this section we use the ultrametric graphon theory introduced in the 
previous sections to give a generalization of the results presented in 
\cite{PhaseTransitionsSpectralCD}. We show the existence of a detectability threshold and give an 
interpretation of this threshold in terms of the structure of the 
communities.

Let \(W:[0,1]^2\rightarrow [0,1]\) be an ultrametric graphon. Define the \emph{Fiedler matrix} of its deterministic Laplacian $L_d^k$ 
as $E_F := E_1$, where $E_1$ is the spectral projector attached to the 
largest nonzero eigenvalue of $L_d^k$,
\begin{equation} \label{eq:maximumfiedler}
\lambda_2(L_d^k)
=
\max_{\substack{x^\top x = 1 \\ \mathbf{1}^\top x = 0}}
x^\top L_d^k\, x.
\end{equation}
The multiplicity of $\lambda_2(L_d^k)$ may be greater than one. If 
$w$ is non-increasing, then
\[
\lambda_2(L_d^k) = -m([0,1])\,w(h([0,1])) = -N_k\,w(h([0,1])),
\]
where $w(h([0,1]))$ is the inter-cluster connection probability. The 
following inequality holds in this case:
\[
\lambda_2(L_d^k) > -N_k\,w(h([0,1])) - m(I_i)
\bigl(w(h(I_i)) - w(h([0,1]))\bigr) =: \lambda_i,
\]
where $I_i$ is a child interval of $[0,1]$ and $\lambda_i$ its attached 
eigenvalue. If for some child interval $I_i$ it holds that 
$w(h([0,1])) > w(h(I_i))$, then
\[
-N_k\,w(h([0,1])) < -N_k\,w(h([0,1])) - m(I_i)
\bigl(w(h(I_i)) - w(h([0,1]))\bigr).
\]
Therefore, the maximum in \eqref{eq:maximumfiedler} is no longer attained 
at the eigenvalue of the root interval $[0,1]$ but rather at the eigenvalue 
of some child interval:
\[
\lambda_2(L_d^k)
= -N_k\,w(h([0,1])) - m(I_i)\bigl(w(h(I_i))-w(h([0,1]))\bigr)
\]
for some $I_i$ satisfying $w(h([0,1])) > w(h(I_i))$. The attached 
spectral projector (Fiedler matrix) has support in $I_i\times I_i$ 
(see \eqref{eq:projectorgraphon}), eliminating any possibility of 
detecting the full community structure in the sense of 
Theorem~\ref{thm:signstructureofspectralprojectors}.
The threshold is given by $p^* = \min_i w(h(I_i))$. Detectability holds 
when $w(h([0,1])) < p^*$ and fails when $w(h([0,1])) > p^*$. In the 
detectable case, Corollary~\ref{cor:eigenvalueultrametricgraphon} gives
\[
\frac{|\hat{\lambda}_2|}{N_k} \rightarrow w(h([0,1])), \qquad \text{a.s.}
\]
The collapse of  the Fiedler matrix is exemplified in Figure \ref{fig:detectability_threshold}.

This recovers the results of \cite{PhaseTransitionsSpectralCD}, where the 
authors study the particular case of the SBM with two communities of equal 
size, corresponding to the ultrametric graphon with root $[0,1]$ and two 
children $[0,\tfrac{1}{2})$ and $[\tfrac{1}{2},1]$, where 
$p^* = \min\{p_1, p_2\}$. The present result allows for 
an arbitrary number of communities of heterogeneous sizes.

\begin{figure}[H]
    \centering
\includegraphics[width=\textwidth]{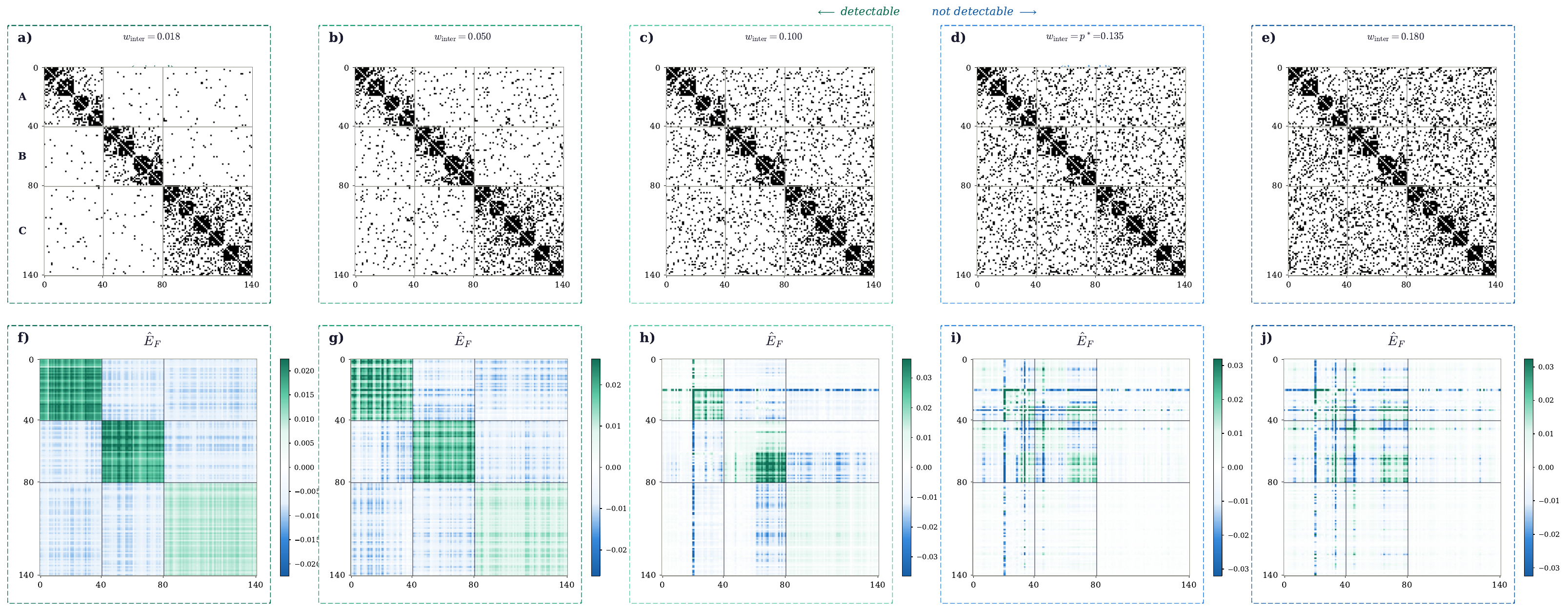}
    \caption{Detectability threshold for the ultrametric graphon of 
    Example~\ref{example:graphonultra}, with first  partition 
    $[0,1] = A \sqcup B \sqcup C$, with $\mu(A)=\mu(B)=\frac{4}{14}$ 
    and $\mu(C)=\frac{6}{14}$. The inter-cluster connection probability 
    $w_\mathrm{inter} = w(h([0,1]))$ is varied while the intra-cluster 
    structure is held fixed, with threshold 
    $p^* = \min\{w(h(A)),\, w(h(B)),\, w(h(C))\} = 0.135$.
    \textbf{Top row~(a)--(e):} adjacency pixel plots of a sampled graph 
    $G \sim L_r^k$ with $N_k = 140$.
    \textbf{Bottom panel~(f)--(j):} empirical Fiedler matrix 
    $\hat{E}_F = \hat{V}\hat{V}^\top$, where $\hat{V} \in \mathbb{R}^{N_k \times 2}$ 
    collects the two Fiedler vectors of $L_r^k$. 
    For $w_\mathrm{inter} < p^*$ (panels a--c, f--h) the sign structure 
    of $\hat{E}_F$ correctly identifies the three communities $A$, $B$, $C$. 
    At the threshold (panels d, i) the Fiedler matrix loses the global 
    block structure, and beyond the threshold (panels e, j) it collapses 
    to the support of a single community, making detection impossible.}
    \label{fig:detectability_threshold}
\end{figure}
Since spectral methods based on the 
leading eigenvectors and its eigenvectors capture only the coarsest level of the community 
hierarchy, the natural setting for this analysis is the single-step 
partition. We therefore restrict attention to the case 
$[0,1]=\bigsqcup\, I_i$, where $I_i$ are the children intervals of 
$[0,1]$, as shown in Figure~\ref{fig:treeonelevelgraphon}.
\begin{figure}[H]
    \centering
    \includegraphics[width=0.7\textwidth]{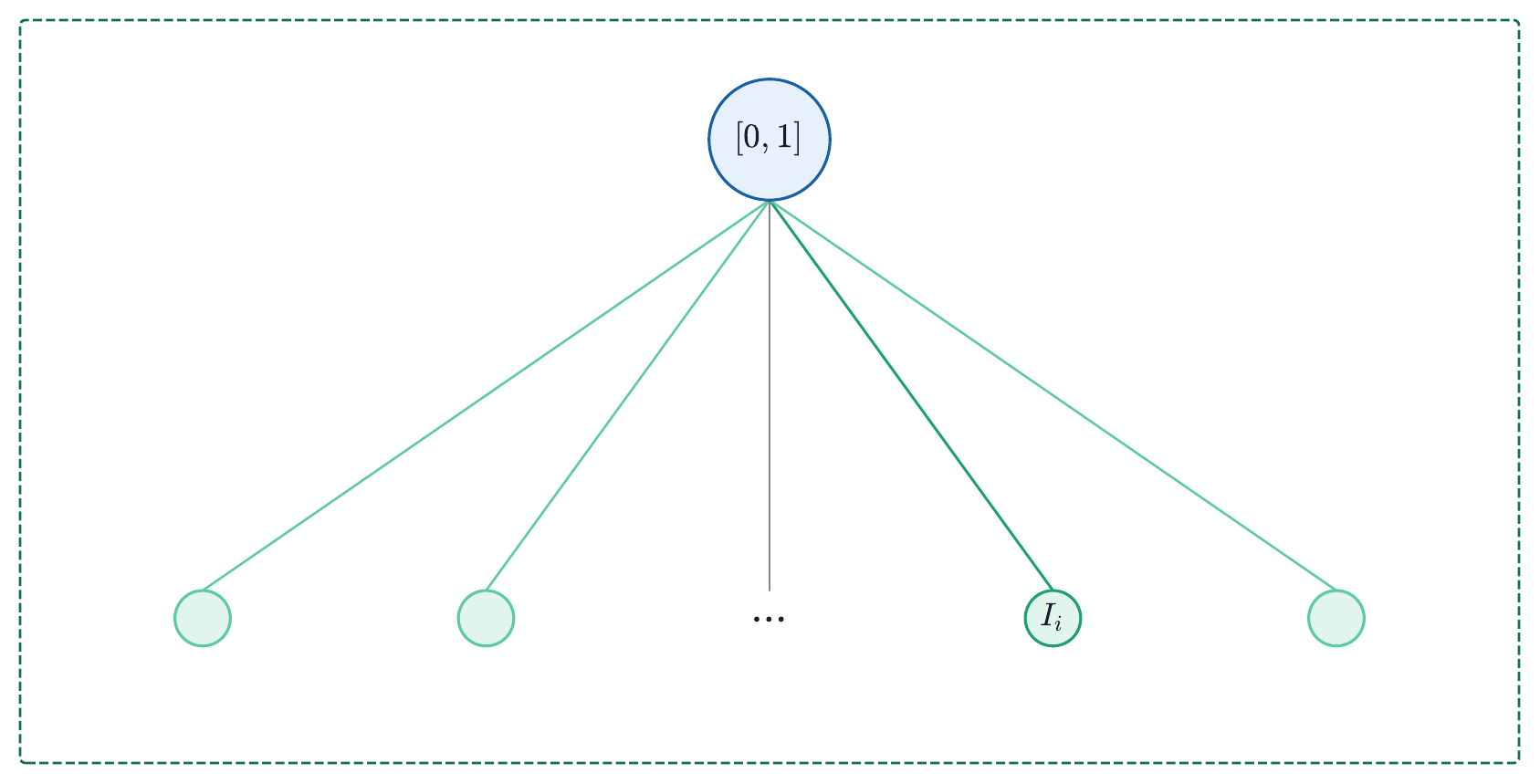}
    \caption{Topological tree of a single step partition.}
    \label{fig:treeonelevelgraphon}
\end{figure}

Instead of requiring $W(x,y) = w(h(I_i))$ for $x,y \in I_i \times I_i$, 
we now consider a broader family of graphons in order to model more general 
intra-community graph distributions.

\begin{definition}
We say that the graphon $W:[0,1]\times [0,1]\to[0,1]$ is a 
\emph{one-level hierarchical graphon} if
\begin{enumerate}
\item for each $i$, $W(x,y)=W_i(x,y)$ for all $x,y\in I_i$, where $W_i$ 
is an arbitrary continuous symmetric function;
\item $W(x,y)=w(h([0,1]))$ whenever $x\in I_i$, $y\in I_j$ with $i\neq j$.
\end{enumerate}
\end{definition}

A one-level hierarchical graphon models a family of graphs $\{G_i\}$, 
where $w(h([0,1]))$ is the inter-community connection probability and 
$W_i:=W|_{I_i\times I_i}$ is the intra-community edge probability of $G_i$. 
We show the existence of a threshold $p^*$ such that if $w(h([0,1]))<p^*$ 
then detectability is possible, that is, there exists a Fiedler matrix 
satisfying the properties of Theorem~\ref{thm:signstructureofspectralprojectors}. 
On the other hand, if $w(h([0,1]))>p^*$, the detection of the real number of communities becomes 
impossible: the Fiedler matrix and the spectral gap no longer carry 
information about the community structure.

First we show that the space
\[
\mathcal{V}_{\mathrm{root}}:=\left\{ \psi: \psi|_{I_i} \, \text{is constant 
for all } i,\, \sum_{i} m(I_i)\psi|_{I_i}=0\right\}
\]
remains an eigenspace of the one-level hierarchical Laplacian.
\begin{proposition}
\label{prop:eigenvaluequasiultragrpahon}
Let $[0,1]=\bigsqcup_{i} I_i$ be a partition into intervals with 
Lebesgue measure $\mu(I_i)>0$. Let $N_k=kN\ge 2$ where $N$ is the LCM 
of the denominators of the lengths $\mu(I_i)$ and $k\ge 1$. Let $X$ be 
the set of $N_k$ total sampled points such that $m(I_i)=N_k\,\mu(I_i)$, 
where $m$ is the counting measure on $X$. Assume that 
$W:[0,1]\times[0,1]\to[0,1]$ is a one-level hierarchical graphon. Define 
the Laplacian operator
\[
(Lf)(x):=\sum_{y\in X} W(x,y)\bigl(f(y)-f(x)\bigr), \qquad x\in X.
\]
Then every $f\in \mathcal{V}_{\mathrm{root}}$ is an eigenvector of $L$ 
with eigenvalue $-w(h([0,1]))\,m(X)$:
\[
Lf = -w(h([0,1]))\,m(X)\,f.
\]
\end{proposition}

\begin{proof}
Fix $f\in \mathcal{V}_{\mathrm{root}}$. By definition there exist constants 
$a_1,\dots,a_K$ such that $f(x)=a_i$ for all $x\in X\cap I_i$ and 
$\sum_{i} m(I_i)a_i=0$. Let $x\in X\cap I_i$. Splitting the sum defining 
the Laplacian yields
\[
(Lf)(x)
=
\sum_{y\in X\cap I_i} W_i(x,y)\bigl(a_i-a_i\bigr)
+
\sum_{j\neq i}\sum_{y\in X\cap I_j} W(x,y)\bigl(a_j-a_i\bigr).
\]
The first term vanishes identically, so the intra-community kernels $W_i$ 
do not contribute. For $j\neq i$, the one-level assumption implies 
$W(x,y)=w(h([0,1]))$ for all $y\in X\cap I_j$, hence
\[
(Lf)(x)
=
\sum_{j\neq i}\sum_{y\in X\cap I_j} w(h([0,1]))\,(a_j-a_i)
=
w(h([0,1]))\sum_{j\neq i} m(I_j)\,(a_j-a_i).
\]
Using $\sum_{j} m(I_j)a_j=0$, we have $\sum_{j\neq i}m(I_j)a_j=-m(I_i)a_i$ 
and $\sum_{j\neq i}m(I_j)=m(X)-m(I_i)$. Hence,
\[
(Lf)(x)
=
w(h([0,1]))\bigl(-m(I_i)a_i - a_i(m(X)-m(I_i))\bigr)
=
-w(h([0,1]))\,m(X)\,a_i.
\]
Since $f(x)=a_i$ for all $x\in X\cap I_i$, it follows that 
$(Lf)(x)=-w(h([0,1]))\,m(X)\,f(x)$ for every $x\in X$, and therefore 
$Lf=-w(h([0,1]))\,m(X)\,f$.
\end{proof}

The proof of the following proposition can be found in 
\cite{vizuete2021laplacian}; this result characterizes the limiting 
behavior of the spectral gap of the random Laplacians attached to a 
continuous graphon $W$.

\begin{prop}
\label{prop:limitspectralgapgraphon}
Let $W$ be a continuous graphon. Let $M$ be the number of isolated 
eigenvalues of the Laplacian operator $L$ of $W$ in the interval 
$[0,\delta_W]$, where $\delta_W:=\inf_{x\in[0,1]}\int_0^1 W(x,y)\,dy$, 
counted with their multiplicities, and let
\[
0=\kappa_1\ge\kappa_2\ge\cdots\ge\kappa_M
\]
be these eigenvalues. Define $\rho=\kappa_2$ if $M\ge 2$, and 
$\rho=\delta_W$ if $M=1$. Then
\[
\lim_{N\to\infty}\frac{-\hat{\lambda}_2(L_r^N)}{N}=\rho,\qquad\text{a.s.,}
\]
where $L_r^N$ is the $N$-vertex random Laplacian attached to $W$.
\end{prop}

Let $W:[0,1]\times[0,1]\to[0,1]$ be a continuous graphon. Since  
Proposition~\ref{prop:eigenvalueaprox} holds for arbitrary graphons it follows that if 
$\lim_{N\to\infty}\frac{-\hat{\lambda}_2(L_r^N)}{N}=\rho$ a.s., then
\[
\lim_{N\to\infty}\frac{-\lambda_2(L_d^N)}{N}=\rho,
\]
where $L_d^N$ is the $N$-vertex deterministic Laplacian attached to $W$. 
The following is the main result of this section.

\begin{thm}
\label{thm:phasetransition}
Let $[0,1]=\bigsqcup_{i} I_i$ be a partition into intervals with Lebesgue 
measure $\mu(I_i)>0$. Let $N_k=kN\ge 2$ where $N$ is the LCM of the 
denominators of the lengths $\mu(I_i)$ and $k\ge 1$. Let $X$ be the set 
of $N_k$ total sampled points such that $m(I_i)=N_k\,\mu(I_i)$, where $m$ 
is the counting measure on $X$.

Let $W:[0,1]\times[0,1]\to[0,1]$ be a one-level hierarchical graphon. Let 
$L_d^k(W_i)$ be the deterministic Laplacian attached to the restricted 
graphon $W|_{I_i\times I_i}$ and the $N_k\mu(I_i)$ sampled points. Denote 
by $\rho_i:=\lim_{k\to\infty}\frac{-\lambda_2(L_d^k(W_i))}{N_k\mu(I_i)}$.
\begin{enumerate}
    \item If $w(h([0,1]))<\min_i\rho_i$, then for sufficiently large $k$ 
    the Fiedler matrix $E_F$ attached to $L_d^k(W)$ coincides with the 
    orthogonal projector onto $\mathcal{V}_{\mathrm{root}}$, and therefore 
    the random realization $\hat{E}_F$ satisfies the properties of 
    Theorem~\ref{thm:signstructureofspectralprojectors}.

    \item If $w(h([0,1]))>\min_i\rho_i$, then $E_F(x,y)=0$ whenever 
    $x\notin X\cap I_i$ or $y\notin X\cap I_i$ for some $i$.
\end{enumerate}
\end{thm}

\begin{proof}
Let $f:X\cap I_i\to\mathbb{R}$ be an eigenfunction of $L_d^k(W_i)$ with 
eigenvalue $\lambda_f<0$. Since $\lambda_f<0$, the orthogonality 
$f\perp\mathbf{1}_{I_i}$ holds, that is, $\sum_{y\in I_i}f(y)m(y)=0$.

We extend $f$ to $X$ by setting $f(x)=0$ for $x\notin X\cap I_i$. For 
$x\in X\cap I_i$,
\begin{equation*}
\begin{split}
(L_d^k(W)f)(x)
&=\sum_{y\in X\cap I_i} W_i(x,y)\bigl(f(y)-f(x)\bigr)
+\sum_{j\neq i}\sum_{y\in X\cap I_j} w(h([0,1]))\bigl(f(y)-f(x)\bigr)\\
&=\sum_{y\in X\cap I_i} W_i(x,y)\bigl(f(y)-f(x)\bigr)
-f(x)\,w(h([0,1]))\,m(X\setminus X\cap I_i)\\
&=f(x)\bigl(\lambda_f - w(h([0,1]))\,m(X\setminus X\cap I_i)\bigr).
\end{split}
\end{equation*}
For $x\notin X\cap I_i$,
\[
(L_d^k(W)f)(x)=w(h([0,1]))\sum_{y\in I_i}f(y)m(y)=0.
\]
Therefore $f$ is an eigenfunction of $L_d^k(W)$ with eigenvalue 
$\lambda_f - w(h([0,1]))\,m(X\setminus X\cap I_i)$.

As shown above, $-w(h([0,1]))\,m(X)=-w(h([0,1]))\,N_k$ is an eigenvalue 
of $L_d^k(W)$ with multiplicity equal to the number of intervals minus one. 
Together with the eigenvalue $0$ attached to $\mathbf{1}=(1,\dots,1)$ and 
the displaced eigenvalues $\lambda_f - w(h([0,1]))\,m(X\setminus X\cap I_i)$ 
of the Laplacians $L_d^k(W_i)$, this gives a complete list of eigenvalues 
of $L_d^k(W)$.

Assume $w(h([0,1]))<\min_i\rho_i$. Then for sufficiently large $k$,
\[
w(h([0,1]))<\min_i\frac{-\lambda_2(L_d^k(W_i))}{N_k\mu(I_i)},
\]
and the following chain of inequalities holds for every $i$:
\begin{equation*}
\begin{split}
-w(h([0,1]))&>\frac{\lambda_2(L_d^k(W_i))}{N_k\mu(I_i)}\\
-w(h([0,1]))\,\mu(I_i)&>\frac{\lambda_2(L_d^k(W_i))}{N_k}\\
-w(h([0,1]))&>\frac{\lambda_2(L_d^k(W_i))}{N_k}-w(h([0,1]))(1-\mu(I_i))\\
-w(h([0,1]))\,N_k&>\lambda_2(L_d^k(W_i))-w(h([0,1]))(N_k-N_k\mu(I_i)).
\end{split}
\end{equation*}
By Proposition~\ref{prop:eigenvaluequasiultragrpahon}, this implies
\[
\lambda_2(L_d^k)
=\max_{\substack{x^\top x=1\\\mathbf{1}^\top x=0}}x^\top L_d^k\,x
=-w(h([0,1]))\,N_k,
\]
so $E_F$ coincides with the orthogonal projector onto $\mathcal{V}_{\mathrm{root}}$.

On the other hand, if $w(h([0,1]))>\min_i\rho_i$, then for sufficiently 
large $k$ there exists an index $i$ such that
\[
-w(h([0,1]))\,N_k<\lambda_2(L_d^k(W_i))-w(h([0,1]))(N_k-N_k\mu(I_i)),
\]
and
\[
\lambda_2(L_d^k)
=\max_{\substack{x^\top x=1\\\mathbf{1}^\top x=0}}x^\top L_d^k\,x
=\lambda_2(L_d^k(W_i))-w(h([0,1]))(N_k-N_k\mu(I_i)).
\]
Since the eigenvectors attached to this eigenvalue are all supported in 
$I_i^2$, the projector satisfies $E_F(x,y)=0$ whenever $x\notin X\cap I_i$ 
or $y\notin X\cap I_i$, as stated.
\end{proof}

We now discuss the meaning of this result. The value $w(h([0,1]))$ 
represents the density of connections between the graphs $G_i$. If 
$\rho^*:=\min_i\rho_i$, the conditions in 
Theorem~\ref{thm:phasetransition} determine how large the inter-community 
density must be compared to $\rho^*$ in order to have spectral 
detectability, that is, in order to:
\begin{enumerate}
    \item Detect the number of communities via the number of similar 
    eigenvalues: for a sufficiently large number of vertices, the sampled 
    graphs should have approximately as many normalized eigenvalues close 
    to $-w(h([0,1]))$ as there are communities minus one, that is, 
    $\frac{\hat{\lambda}_i}{N_k}\to -w(h([0,1]))$ for the corresponding 
    indices $i$.
    \item Identify the community structure via the sign pattern as 
    described in Theorem~\ref{thm:signstructureofspectralprojectors}.
\end{enumerate}
Moreover, many spectral algorithms such as $k$-means depend on selecting 
the first $k$ largest eigenvectors; as established, if the inter-community 
density is not small enough, those eigenvectors do not carry information 
about the community structure (Case~2, 
Theorem~\ref{thm:phasetransition}).

We extend the analysis of \cite{PhaseTransitionsSpectralCD} by 
interpreting the threshold $\rho^*$.

Recall that for an undirected graph $G=(V,E)$ and a subset $S\subset V$, 
we denote by $\delta(S)$ the set of edges with one endpoint in $S$ and the 
other in $V\setminus S$. The \emph{edge conductance} of $S$ is defined as
\[
\phi(S)\;:=\;\frac{|\delta(S)|}{\mathrm{vol}(S)},
\qquad
\mathrm{vol}(S):=\sum_{v\in S}\deg(v),
\]
and the conductance of the graph is
\[
\phi(G)\;:=\;\min_{S:\,\mathrm{vol}(S)\le |E|}\phi(S).
\]
The quantity $\phi(G)$ measures the presence of bottlenecks in the graph, 
that is, subsets of vertices whose boundary is small relative to their 
total degree. Let $\lambda_2(L)$ denote the largest nonzero eigenvalue of 
$L=A-D$. Cheeger's inequality gives
\[
\frac{d_{\min}}{2}\,\phi(G)^2
\;\le\;
|\lambda_2(L)|
\;\le\;
2\,d_{\max}\,\phi(G).
\]In particular, the magnitude of the largest nonzero eigenvalue of $L$ is 
controlled by the conductance and the degree: small conductance corresponds 
to a small spectral gap, reflecting the existence of a bottleneck. On the 
other hand, small $d_{\max}$, that is, low local edge density, also 
corresponds to a small spectral gap.

Since
\[
\rho_i=\lim_{k\to\infty}\frac{-\lambda_2(L_d^k(W_i))}{n_i},
\]
where $n_i=N_k\mu(I_i)$ is the number of vertices of the subgraph $G_i$, 
the threshold admits a clear interpretation in terms of the structural 
connectivity of the sub-networks.

Indeed, Cheeger-type inequalities relate the magnitude of the second 
eigenvalue of the Laplacian to the conductance and the degree distribution 
of the graph. After normalization by the number of vertices, $\rho_i$ 
quantifies an effective spectral gap per vertex, which is sensitive both 
to the presence of bottlenecks (small conductance) and to the absence of 
highly connected vertices (bounded local degree).

Small values of $\rho_i$ indicate weak global cohesion, either due to 
narrow cuts separating large portions of the sub-network or due to locally 
sparse connectivity, whereas larger values of $\rho_i$ correspond to 
graphs whose connectivity remains robust in the large-size limit.

Therefore, the effectiveness of spectral detectability relies on the 
contrast between how strongly each community is internally connected and 
the inter-community edge density: once the latter exceeds $\min_i\rho_i$, 
the spectral community structure is no longer detectable.
\section{Random walks on hierarchical community networks.}

Laplacian dynamics on graphs are a well-studied object in network science, 
with applications ranging from diffusion processes and centrality measures 
to epidemic spreading \cite{MasudaPorterLambiotte2017, Lambiotte2015}. 
Moreover, this dynamics has been studied in the context of community 
detection via the Markov stability framework \cite{Lambiotte2015, Schaub2023}. 
Motivated by this, we study dynamical variables of the continuous-time 
Markov chain attached to an ultrametric graphon and analyze how the 
hierarchical community structure determines their large-graph behavior. 
To this end, we construct a limiting pseudo-inverse Laplacian operator 
and establish its almost sure convergence; we remark that the convergence 
of pseudo-inverse Laplacians is of independent mathematical interest, 
beyond its role in the analysis of the dynamical variables attached to the Laplacian dynamic. 
The main results show that the commute times and hitting times associated 
to this dynamics lose all information about the hierarchical community 
structure in the large-graph limit.

Let $G=(V,E)$ be an undirected graph with $n$ vertices. Let 
$A=(A_{ij})_{i,j=1}^{n}$ be the adjacency matrix and let 
$d_i:=\sum_{j=1}^n A_{ij}$ be the degree of vertex $v_i$. Let $D$ denote 
the degree matrix with diagonal entries $d_{ii}=d_i$.

A natural continuous-time Markov chain (CTMC) on $G$, also called 
Laplacian dynamics or diffusion on $G$, is defined by the combinatorial 
Laplacian $L=A-D$ as generator: the transition probability from vertex $i$ 
to vertex $j$ at time $t\geq 0$ is given by
\[
P_{ij}(t) = \bigl(e^{tL}\bigr)_{ij}.
\]
The mean first passage time (MFPT) $m_{ik}$ is defined as the expected 
time for the CTMC starting at vertex $i$ to reach vertex $k$, with 
$m_{ii}=0$ by convention. The commute time between vertices $i$ and $j$ 
is defined as
\[
C_{ij} := m_{ij} + m_{ij}.
\]
The MFPT and commute times are closely related to $L$ via its 
Moore--Penrose pseudo-inverse. The pseudo-inverse of $L$ with eigenvalues 
$0=\lambda_1\geq\lambda_2\geq\dots\geq\lambda_{n}$ is defined by
\[
L^{+}=V\Lambda^{+}V^{\top},
\qquad
\Lambda^{+}=\mathrm{diag}\!\left(0,-\frac{1}{\lambda_2},\dots,
-\frac{1}{\lambda_{n}}\right),
\]
where the columns of $V$ form an orthonormal eigenbasis of $L$. The 
hitting times satisfy \cite{vanKampen2007}
\[
m(i|k) = \frac{L^+_{kk} - L^+_{ik}}{\pi_k},
\]
where $\pi_k = 1/|V|$ is the stationary distribution of the CTMC. 
The commute time is therefore
\[
C(i,j) = m_{ij} + m_{ji}.
\]

Let $W:[0,1]^2\to[0,1]$ be an ultrametric graphon. Denote 
by $(L_d^k)^+$ the pseudo-inverse of the deterministic Laplacian $L_d^k$ 
and by $(L_r^k)^+$ the pseudo-inverse of its random realization $L_r^k$.

Throughout this section we make the following assumption.
\vspace{5pt}
\noindent\textbf{Hypothesis 1.} The probability of internal edge connection 
in the communities of the last level $M$ is $1$. That is, $W(x,y)=w(d(x,y))=1$ 
for $x,y\in I_i^{(M)}$, where $I_i^{(M)}$ is any interval at the last 
level.
\vspace{5pt}

We now explain the motivation behind this assumption. In hierarchical 
assortative communities, nodes within the same community are more likely 
to be connected, so the deepest communities in the hierarchy tend to be 
the most densely connected. As noted in Example~\ref{example:graphonultra}, 
if the tree is deep enough or $w$ increases quickly enough as 
$d(\cdot,\cdot)\to 0$, the edge probability at the deepest level $M$ is 
approximately $1$. Hypothesis~1 isolates the effect of inter-community 
bridges, which is our main object of interest in the ultrametric 
hierarchical regime for random walks. Moreover, this hypothesis reduces 
the complexity of the statistical analysis. The extension to dense random 
internal structure in the deepest communities can be treated using general 
graphons and is left for future work.

We construct a limiting object for the sequence of pseudo-inverse 
Laplacians $(L_d^k)^+$. To this end, we recall the standard identification 
between matrices and Hilbert--Schmidt operators. Let $A=(A_{ij})_{i,j=1}^{N}$ 
be an $N\times N$ matrix and let $\{I_i\}_{i=1}^N$ be the uniform partition 
of $[0,1]$ into intervals of length $1/N$. The matrix $A$ can be identified 
with an integral operator
\[
L^2([0,1])\ni f\mapsto T(A)f(x)=\int_0^{1}A(x,y)f(y)\,dy,
\]
where $A(x,y)=A_{ij}$ if and only if $x\in I_i$ and $y\in I_j$. For a 
kernel $K\in L^2([0,1]\times[0,1])$ it is known that
\[
\|T(K)\|_{HS}=\|K\|_{L^2},
\]
that is, the Hilbert--Schmidt norm of the integral operator equals the 
$L^2$ norm of its kernel. Therefore,
\[
\|T(A)\|_{HS}^2=\|A\|_{L^2}^2=\int_0^1\int_0^1|A(x,y)|^2\,dx\,dy
=\frac{1}{N^2}\sum_{i,j}|A_{ij}|^2=\frac{1}{N^2}\|A\|_F^2,
\]
and hence $\|T(A)\|_{HS}=\frac{1}{N}\|A\|_F$. The entries of $A$ can be 
recovered from the operator via the $L^2([0,1])$ inner product:
\[
\langle\sqrt{N}\,\mathbf{1}_{I_i},\,T(A)\sqrt{N}\,\mathbf{1}_{I_j}
\rangle_{L^2}=\frac{1}{N}A_{ij}.
\]
Conversely, given a kernel $A\in L^2([0,1]\times[0,1])$, define the matrix 
$A_N=((A_N)_{ij})_{i,j=1}^{N}$ by
\[
(A_N)_{ij}=\frac{1}{|I_i||I_j|}\int_{I_i\times I_j}A(x,y)\,dx\,dy.
\]
The matrix $A_N$ has associated kernel $A_N(x,y)=(A_N)_{ij}$ for 
$x\in I_i$, $y\in I_j$, and it is known that $A_N(x,y)\to A(x,y)$ in 
$L^2$. Therefore,
\[
\|T(A)-T(A_N)\|_{HS}=\|A-A_N\|_{L^2}\to 0
\]
as $N\to\infty$. The operator $T(A_N)$ is referred to as the 
discretization of $T(A)$ at level $N$.

\begin{prop}\label{prop:operatorprojectorconvergence}
Let $W:[0,1]^2\to[0,1]$ be an ultrametric graphon and let $L_d^k$ be the 
attached deterministic ultrametric Laplacian. If $E_I^k$ is the spectral 
projector of $L_d^k$ attached to an internal node $I$ at level 
$\ell<M$, then there exists 
$E_I\in L^2([0,1]\times[0,1])$ such that
\[
\|T(N_k\,E_I^k)-T(E_I)\|_{HS}\to 0
\]
as $k\to\infty$.
\end{prop}

\begin{proof}
Define $E_I:[0,1]^2\to\mathbb{R}$ by
\[
E_I(x,y)=\begin{cases}
-\dfrac{1}{\mu(I)}, & \text{if $x$ and $y$ are in different children of $I$,}\\[6pt]
\dfrac{1}{\mu(J)}-\dfrac{1}{\mu(I)}, & \text{if $x,y\in J$ for some child $J\in C(I)$.}
\end{cases}
\]
Given two intervals $I_i^{(M)}$ and $I_j^{(M)}$ at the last level, the 
average of $E_I(x,y)$ satisfies
\[
\frac{1}{N_k^2}\int_{I_i^{(M)}\times I_j^{(M)}}E_I(x,y)\,dx\,dy
=E_I(x_i,y_j)
\]
for any $x_i\in I_i^{(M)}$ and $y_j\in I_j^{(M)}$. The matrix $N_k E_I^k$ 
has associated integral kernel $N_k E_I^k(x,y)=(N_k E_I^k)_{ij}$ for 
$x\in I_i^{(M)}$ and $y\in I_j^{(M)}$. This kernel satisfies
\[
E_I^k(x,y)=\begin{cases}
-\dfrac{1}{m(I)}, & \text{if $x$ and $y$ are in different children of $I$,}\\[6pt]
\dfrac{1}{m(J)}-\dfrac{1}{m(I)}, & \text{if $x,y\in J$ for some child $J\in C(I)$,}
\end{cases}
\]
and consequently $N_k E_I^k(x,y)=E_I(x,y)$ since $\mu=\frac{m}{N_k}$. 
Therefore $T(N_k E_I^k)$ is the discretization of $T(E_I)$ at level $N_k$, 
and $\|T(N_k E_I^k)-T(E_I)\|_{HS}\to 0$ as $k\to\infty$.
\end{proof}

The following theorem relates the random spectral projectors of $L_r^k$ 
with the limiting operators $T(E_I)$.

\begin{thm}\label{thm:randomspectralconvergence}
Let $W:[0,1]^2\to[0,1]$ be an ultrametric graphon and let $M$ denote the 
last level of the attached discretizations. Let $L_r^k$ be the associated 
random ultrametric Laplacian. Let $\hat{v}_1,\dots,\hat{v}_m$ be 
eigenvectors whose eigenvalues $\hat{\lambda}_{n_1}(k),\dots,
\hat{\lambda}_{n_m}(k)$ form the spectral cluster associated with the 
eigenvalue $\lambda(I)$ of multiplicity $m$ of $L_d^k$ where the level of \(I\) satisfies 
$\ell<M$. Let 
$\hat{V}=[\hat{v}_1,\dots,\hat{v}_m]$ and $\hat{E}_I^k:=\hat{V}\hat{V}^\top$. 
Then
\[
\|T(N_k\,\hat{E}_I^k)-T(E_I)\|_{HS}\to 0\quad\text{a.s.}
\]
as $k\to\infty$.
\end{thm}

\begin{proof}
The triangle inequality gives
\[
\|T(N_k\hat{E}_I^k)-T(E_I)\|_{HS}
\leq\|T(N_k\hat{E}_I^k)-T(N_k E_I^k)\|_{HS}
+\|T(N_k E_I^k)-T(E_I)\|_{HS}.
\]
By the relation between the Hilbert--Schmidt and Frobenius norms,
\[
\|T(N_k\hat{E}_I^k)-T(N_k E_I^k)\|_{HS}
=\frac{1}{N_k}\|N_k\hat{E}_I^k-N_k E_I^k\|_F.
\]
By Proposition~\ref{prop:eigenvectoraproximation}, there exists $\delta>0$ 
and, for any $\gamma\in(0,\frac{1}{2})$, a constant $C=C(\gamma)$ such that
\[
\|T(N_k\hat{E}_I^k)-T(N_k E_I^k)\|_{HS}
\leq\frac{2\sqrt{2}\,m}{\delta}N_k^{\gamma-\frac{1}{2}},
\]
with probability at least $1-2N_ke^{-CN_k^{2\gamma}}$. Together with 
$\|T(N_k E_I^k)-T(E_I)\|_{HS}\to 0$, this gives the desired result.
\end{proof}

It is important to emphasize that $\frac{2\sqrt{2}\,m}{\delta}
N_k^{\gamma-\frac{1}{2}}\to 0$ as $k\to\infty$, since the multiplicity 
$m$ does not depend on $k$: the node $I$ is assumed to be at level 
$\ell<M$. By Hypothesis~1, the graphon is fully connected in the deepest 
communities. Hence, the spectral projector attached to the eigenvalue 
$\lambda(I_i^{(M)})$ with multiplicity $N_k\mu(I_i^{(M)})$ of both 
$L_d^k$ and $L_r^k$ is
\[
E^k_{I_i^{(M)}}(x,y)=\delta_{xy}-\frac{1}{m_{i,M}},
\]
where $m_{i,M}:=m(I_i^{(M)})$. In matrix form,
\[
E^k_{I_i^{(M)}}=I_n-\frac{1}{m_{i,M}}\mathbf{1}_{I_i^{(M)}}
\mathbf{1}_{I_i^{(M)}}^\top,
\]
and for a vector $v\in\mathbb{R}^{m_{i,M}}$,
\[
(E^k_{I_i^{(M)}}v)_a=v_a-\frac{1}{m_{i,M}}\sum_{b=1}^{m_{i,M}}v_b.
\]
In $L^2([0,1])$, the orthogonal projector onto 
$\mathcal{V}:=\{f\in L^2(I_i^{(M)}):\int_{I_i^{(M)}}f(x)\,dx=0\}$ is
\[
P_{I_i^{(M)}}f(x)=f(x)-\int_{I_i^{(M)}}f(x)\,dx.
\]
Each sampled point of $I_i^{(M)}$ is identified with a sub-interval 
$I_{a,i}^{(M)}\subset I_i^{(M)}$ of length $\frac{1}{N_k}$. Define 
$\mathcal{X}_{i,M}\subset L^2([0,1])$ as the finite-dimensional space 
spanned by the indicator functions $\sqrt{m_{i,M}}\,\mathbf{1}_{I_{a,i}^{(M)}}$; 
this space is isometric to $\mathbb{R}^{m_{i,M}}$. For 
$f\in\mathcal{X}_{i,M}$,
\[
f=\sqrt{m_{i,M}}\sum_b v_b\,\mathbf{1}_{I_{b,i}^{(M)}}
\]
and
\[
P_{I_i^{(M)}}f(x)=\sqrt{m_{i,M}}\,v_a
-\frac{\sqrt{m_{i,M}}}{m_{i,M}}\sum_b v_b.
\]
Since $\mathcal{X}_{i,M}$ is isometric to $\mathbb{R}^{m_{i,M}}$ via 
$U_{i,M}:\mathbb{R}^{m_{i,M}}\to\mathcal{X}_{i,M}$, where 
$U_{i,M}v(x)=\sqrt{m_{i,M}}\,\mathbf{1}_{I_{a,i}^{(M)}}(x)$, we conclude
\[
P_{I_i^{(M)}}\big|_{\mathcal{X}_{i,M}}=U_{i,M}E^k_{I_i^{(M)}}U_{i,M}^{-1}.
\]
In particular, the entries of $E^k_{I_i^{(M)}}$ are recovered via the 
inner product:
\[
\langle\sqrt{m_{i,M}}\,\mathbf{1}_{I_{a,i}^{(M)}},\,
P_{I_i^{(M)}}\sqrt{m_{i,M}}\,\mathbf{1}_{I_{b,i}^{(M)}}\rangle_{L^2}
=\langle e_a,E^k_{I_i^{(M)}}e_b\rangle_{\mathbb{R}^{m_{i,M}}}
=(E^k_{I_i^{(M)}})_{ab}.
\]
For a given Laplacian $L$, its pseudo-inverse is constructed via its 
eigenvalue decomposition. In order to construct a limiting object, it is 
therefore important to analyze the behavior of the inverses of the 
eigenvalues, which we address next.

\begin{prop}
\label{prop:inverseeigenvalueconvergenve}
Let $N_k\geq 2$ and let $W$ be an ultrametric graphon. Consider its
attached random Laplacian $L_r^{k}$ with eigenvalues
$0=\hat{\lambda}_1\geq\hat{\lambda}_2\geq \dots \geq \hat{\lambda}_n$.
For all $\gamma \in (0,\frac{1}{2})$, there is a constant $C=C(\gamma)$
and $c_0>0$, independent of $k$, such that
\[
\left|
\frac{N_k}{|\hat{\lambda}_{i}|}-\left(\sum_{I_m\in \gamma_r(I_n)}
\mu(I_m)\Bigl(w(h(I_m))-w(h(I_{F(m)}))\Bigr)\right)^{-1}
\right|
\le
\frac{2}{c_0^{2}}\,N_k^{\gamma-\frac{1}{2}},
\]
with probability at least $1-2N_ke^{-CN_k^{2\gamma}}$, for some
interval $I_n$.
\end{prop}
\begin{proof}
By Proposition~\ref{prop:eigenvalueaprox}, for every
$\gamma\in(0,\tfrac{1}{2})$, there exists a constant $C$ independent
of $k$ such that
\[
\left|
\frac{\lambda(I_n,k)}{N_k}-\frac{\hat{\lambda}_{n_i}(k)}{N_k}
\right|
\le
N_k^{\gamma-\frac{1}{2}},
\]
with probability at least $1-2N_ke^{-C_1 N_k^{2\gamma}}$, where
\[
\frac{\lambda(I_n,k)}{N_k}=\sum_{I_m\in \gamma_r(I_n)}
\mu(I_m)\Bigl(w(h(I_m))-w(h(I_{F(m)}))\Bigr).
\]
This implies
\[
|\hat{\lambda}_{n_i}(k)-\lambda(I_n,k)|
\le
N_k^{\gamma+\frac{1}{2}}.
\]
Moreover, let $c_0>0$ be a constant satisfying
\[
|\lambda(I_n,k)|\ge c_0 N_k,
\qquad
|\hat{\lambda}_{n_i}(k)|\ge \tfrac{c_0}{2}N_k
\]
for all sufficiently large $k$. Consequently,
\[
\left|
\frac{1}{\hat{\lambda}_{n_i}(k)}-\frac{1}{\lambda(I_n,k)}
\right|
=
\frac{|\hat{\lambda}_{n_i}(k)-\lambda(I_n,k)|}
{|\hat{\lambda}_{n_i}(k)\,\lambda(I_n,k)|}
\le
\frac{2}{c_0^{2}}\,N_k^{\gamma-\frac{3}{2}}.
\]
After multiplying both sides by $N_k>0$ we obtain the result.
\end{proof}

For brevity we denote
\[
\nu(I):=\sum_{I_m\in \gamma_r(I)}
\mu(I_m)\Bigl(w(h(I_m))-w(h(I_{F(m)}))\Bigr).
\]
\begin{definition}
Let $W$ be an ultrametric graphon satisfying Hypothesis~1. Let $T(E_I)$
be the integral operator defined in
Proposition~\ref{prop:operatorprojectorconvergence}, for an internal
node $I$ with level $\ell<N$. And let
$P_{I_i^{(N)}}f(x)=f(x)\mathbf{1}_{I_i^{(N)}}-\int_{I_i^{(N)}}f(x)\,dx$.
The \emph{pseudo-inverse Laplacian operator}
$L_W^{+}:L^2([0,1])\rightarrow L^2([0,1])$ is defined as
\[
L^2([0,1])\ni f\mapsto L_W^{+}f(x)
=\sum_{i}\frac{1}{\nu(I_i^{(N)})}P_{I_i^{(N)}}f(x)
+\sum_{\ell<N}\frac{1}{\nu(I_{\ell})}T(E_{I_{\ell}})f(x),
\]
where the first sum runs over the indices of the intervals at the last
level, and the second sum runs over all internal nodes of level
$\ell<N$.
\end{definition}

The relationship between $L_W^+$ and $L_d^k$ is direct. Every sampled
point $x_i\in [0,1]$ can be identified with a sub-interval of one of
the intervals at the last level. This interval has length
$\frac{1}{N_k}$ and will be denoted by $I(x_i,k)$. Let
$\mathcal{X}_{N}$ be the space generated by the functions
$\sqrt{N_k}\,\mathbf{1}_{I(x_i,k)}(x)$. This space, as noted before,
is isometric to $\mathbb{R}^{N_k}$. Since $T(N_k\,E_{I_{\ell}}^k)$ is
the discretization of $T(E_{I_{\ell}})$ at level $N_k$, the following
holds:
\begin{equation*}
\begin{split}
\bigl\langle\sqrt{N_k}\,\mathbf{1}_{I(x_i,k)},\,
L_W^{+}\sqrt{N_k}\,\mathbf{1}_{I(x_j,k)}\bigr\rangle_{L^2}
&=\sum_i\frac{1}{\nu(I_i^{(N)})}\bigl(E^k_{I_i^{(N)}}\bigr)_{ij}
+\sum_{\ell<N}\frac{1}{\nu(I_{\ell})}\bigl(E_{\ell}^k\bigr)_{ij}\\
&=\sum_i\frac{N_k}{\lambda(I_i^{(N)})}\bigl(E^k_{I_i^{(N)}}\bigr)_{ij}
+\sum_{\ell<N}\frac{N_k}{\lambda(I_{\ell})}\bigl(E_{\ell}^k\bigr)_{ij}\\
&=N_k\bigl((L_d^k)^+\bigr)_{ij}.
\end{split}
\end{equation*}
Since the inner product is bilinear, we obtain
\[
\bigl((L_d^k)^+\bigr)_{ij}
=\bigl\langle\mathbf{1}_{I(x_i,k)},\,
L_W^{+}\mathbf{1}_{I(x_j,k)}\bigr\rangle_{L^2},\quad k>0.
\]
Each spectral projector attached to $L_r^+$ has an associated operator
$T(N_k\,\hat{E}_I^k)$. The matrix representation in the
finite-dimensional space $\mathcal{X}_N$ of this operator is:
\[
\bigl\langle\sqrt{N_k}\,\mathbf{1}_{I(x_i,k)},\,
T(N_k\,\hat{E}_I^k)\sqrt{N_k}\,\mathbf{1}_{I(x_j,k)}\bigr\rangle_{L^2}
=\bigl(\hat{E}_I^k\bigr)_{ij}.
\]
Therefore, the operator
\[
T\bigl((L_r^k)^+\bigr)
:=\sum_{i}\frac{N_k}{\hat\lambda(I_i^{(N)})}P_{I_i^{(N)}}f(x)
+\sum_{\ell<N}\frac{N_k}{\hat\lambda(I_{\ell})}
T(N_k\,\hat{E}_{I_{\ell}}^k)f(x)
\]
has as its matrix representation in $\mathcal{X}_N$ the matrix
$(L_r^k)^+$:
\[
\bigl\langle\sqrt{N_k}\,\mathbf{1}_{I(x_i,k)},\,
T\bigl((L_r^k)^+\bigr)\sqrt{N_k}\,\mathbf{1}_{I(x_j,k)}\bigr\rangle_{L^2}
=N_k\bigl((L_r^k)^+\bigr)_{ij}.
\]

\begin{thm}\label{thm:convergencepseudoinverserandom}
Let $W$ be an ultrametric graphon satisfying Hypothesis~1. Then
\[
\|T\bigl((L_r^k)^+\bigr)-L_W^+\|_{L^2}\rightarrow 0,\quad\text{a.s.}
\]
as $k\rightarrow\infty$.
\end{thm}

\begin{proof}
By adding and subtracting
$\sum_{\ell<N}\frac{1}{\nu(I_\ell)}T(N_k\,\hat{E}_{I_{\ell}}^k)$,
we obtain
\begin{equation*}
\begin{split}
\|T\bigl((L_r^k)^+\bigr)-L_W^+\|_{L^2}
&\leq \sum_i\left|\frac{N_k}{\hat\lambda(I_i^{(N)})}
-\frac{1}{\nu(I_i^{(N)})}\right|
\|P_{I_i^{(N)}}\|_{L^2}\\
&\quad+\sum_{\ell<N}\|T(N_k\,\hat{E}_{I_{\ell}}^k)\|_{L^2}
\left|\frac{N_k}{\hat\lambda(I_\ell)}
-\frac{1}{\nu(I_\ell)}\right|\\
&\quad+\sum_{\ell<N}\frac{1}{\nu(I_{\ell})}
\|T(N_k\,\hat{E}_{I_{\ell}}^k)-T(E_{I_\ell})\|_{L^2}.
\end{split}
\end{equation*}
The sequences $\|P_{I_i^{(N)}}\|_{L^2}$ and
$\|T(N_k\,\hat{E}_{I_{\ell}}^k)\|_{L^2}$ are convergent and therefore
bounded. The space $L^2([0,1])$ satisfies
$\|\cdot\|_{L^2}\leq\|\cdot\|_{HS}$; consequently, since the RHS is a
finite sum, after applying
Proposition~\ref{prop:inverseeigenvalueconvergenve} and
Corollary~\ref{cor:eigenvalueultrametricgraphon} the limit follows.
\end{proof}

 As explained in the introduction of this section, many dynamic variables 
are related to the pseudo-inverse $L^+$. Two important variables for the 
dynamics generated by the CTMC associated with a given Laplacian are the 
hitting times and the commute times. The mean first passage time attached 
to the generator $\frac{1}{N_k}L_d^k$ is
\[
m_{ij}^{d,k}=N_k^2\bigl((L_d^k)^+_{jj}-(L_d^k)^+_{ij}\bigr)
=\langle e_i,\,N_k^2(L_d^k)^+(e_i-e_j)\rangle,
\]
where $e_i$ is the $i$-th canonical basis vector of $\mathbb{R}^{N_k}$. 
Analogously, the mean first passage time attached to the generator 
$\frac{1}{N_k}L_r^k$ is
\[
m_{ij}^{r,k}=N_k^2\bigl((L_r^k)^+_{jj}-(L_r^k)^+_{ij}\bigr)
=\langle e_i,\,N_k^2(L_r^k)^+(e_i-e_j)\rangle.
\]
This inner product can be expressed as an inner product in $L^2$ via the 
attached operators:
\[
N_k\langle e_i,N_k(L_r^k)^+(e_i-e_j)\rangle_{\mathbb{R}^{N_k}}
=N_k\langle\sqrt{N_k}\,\mathbf{1}_{I(x_i,k)},\,T((L_r^k)^+)
\bigl(\sqrt{N_k}\,\mathbf{1}_{I(x_i,k)}-\sqrt{N_k}\,\mathbf{1}_{I(x_j,k)}\bigr)
\rangle_{L^2},
\]
and
\[
N_k\langle e_i,N_k(L_d^k)^+(e_i-e_j)\rangle_{\mathbb{R}^{N_k}}
=N_k\langle\sqrt{N_k}\,\mathbf{1}_{I(x_i,k)},\,T(L_W^+)
\bigl(\sqrt{N_k}\,\mathbf{1}_{I(x_i,k)}-\sqrt{N_k}\,\mathbf{1}_{I(x_j,k)}\bigr)
\rangle_{L^2}.
\]
Therefore,
\[
\left|\frac{1}{N_k}m_{ij}^{r,k}-\frac{1}{N_k}m_{ij}^{d,k}\right|
\leq\sqrt{2}\,\|T((L_r^k)^+)-T(L_W^+)\|_{L^2},
\]
where by Theorem~\ref{thm:convergencepseudoinverserandom} the right-hand 
side tends to zero a.s. In a similar way,
\[
\left|\frac{1}{N_k}C_{ij}^{r,k}-\frac{1}{N_k}C_{ij}^{d,k}\right|
\leq\sqrt{2}\,\|T((L_r^k)^+)-T(L_W^+)\|_{L^2},
\]
where $C_{ij}=m_{ij}+m_{ji}$ is the commute time of the CTMC.

\begin{thm}
\label{thm:convergencecommutetimes}
Let $W$ be an ultrametric graphon satisfying Hypothesis~1. Then
\[
\left|\frac{1}{N_k}m_{ij}^{r,k}-\frac{1}{\nu(I_{F(j)})}\right|
\to 0, \quad\text{a.s.,}
\]
and
\[
\left|\frac{1}{N_k}C_{ij}^{r,k}-\left(\frac{1}{\nu(I_{F(i)})}
+\frac{1}{\nu(I_{F(j)})}\right)\right|\to 0, \quad\text{a.s.,}
\]
as $k\to\infty$.
\end{thm}

\begin{proof}
 The mean first passage time of operator of the CTMC with infinitesimal generator $\frac{1}{N_k}L_d^k$  satisfies
\[
m_{ij}^{d,k} = N_k^2\left((L_d^k)^+_{jj} - (L_d^k)^+_{ij}\right),
\]
since $\frac{1}{N_k}L_d^k$ has pseudo-inverse $N_k(L_d^k)^+$. A direct 
computation using the explicit eigenvalue decomposition of the ultrametric 
Laplacian yields: for sampled points $x_i,x_j\in [0,1]$ let $I(x_i,k)$ and $I(x_j,k)$ be the sub-intervals of the corresponding intervals at the last level, with length
$\frac{1}{N_k}$ and \(x_i\neq x_j\), then
\[
m_{ij}^{d,k}=\frac{1}{\nu(I_{F(j)})}\cdot\frac{1}{\mu(I(x_j,k))}
+\sum_{T:\,I_{F(j)}\subseteq T\subsetneq I_n}
\frac{1}{\mu(T^+)}\left(\frac{1}{\nu(T)}-\frac{1}{\nu(T)}\right).
\]
Therefore, 
\[\frac{m_{ij}^{d,k}}{N_k}=\frac{1}{\nu(I_{F(j)})}\cdot\frac{1}{m(I(x_j,k))}
+\sum_{T:\,I_{F(j)}\subseteq T\subsetneq I_n}
\frac{1}{m(T^+)}\left(\frac{1}{\nu(T)}-\frac{1}{\nu(T)}\right).\]
As \(N_k\rightarrow\infty\) we obtain

\[\lim_{N_k\rightarrow\infty}\frac{m_{ij}^{d,k}}{N_k}=\frac{1}{\nu(I_{F(j)})},\]
since 
\[
\left|\frac{1}{N_k}m_{ij}^{r,k}-\frac{1}{N_k}m_{ij}^{d,k}\right|
\leq\sqrt{2}\,\|T((L_r^k)^+)-T(L_W^+)\|_{L^2}\to 0, \quad\text{a.s.}
\]
Where the right-hand side tends to zero a.s. by Theorem \ref{thm:convergencepseudoinverserandom}, the desired limit follows. Similarly, 
\[
\left|\frac{1}{N_k}C_{ij}^{r,k}-\left(\frac{1}{\nu(I_{F(i)})}
+\frac{1}{\nu(I_{F(j)})}\right)\right|\to 0, \quad\text{a.s.,}
\]
as stated.
\end{proof}

The eigenvalue \(\nu(I_{F(i)})\) rewrites as \(\nu(I_{F(i)})=\int_{0}^{1}W(x_i,y)dy\), for \(x_i\in I_F(i)\) i.e. the expected degree in \(I_i\). Therefore, the convergence in Theorem \ref{thm:convergencecommutetimes} shows how the commute times collapse to the sum of the inverses of the expected degrees. 

The commute distance is often regarded as a global measure of connectivity because it accounts for all possible paths between two vertices. Since it is derived from the pseudo-inverse of the Laplacian, it incorporates information from the entire graph rather than relying solely on shortest connections. \newline

Intuitively, one might expect that vertices connected through many alternative routes should be considered “close”, while vertices separated by structural bottlenecks should appear “far apart”. This intuition leads to the following commonly assumed property: Vertices belonging to the same cluster of a graph should have small commute distance, whereas vertices in different clusters should have a comparatively large commute distance as mentioned in \cite{vonluxburg2014hitting}.

However, as we prove, this behavior does not necessarily persist in hierarchical community graphs. Our result show how the commute distance loses the information about the ultrametric structure as the number of vertices increase; the commute times collapse to a quantity that only regard local information, namely, the sum of the inverse of the degrees. For example, in a regular tree with regular size communities in all levels, the degree is homogeneous over all vertices, and therefore, even if two edges belong to two different communities, the commute times matrix converges to a constant matrix \(a\textbf{1}\textbf{1}^{T}\).   

A similar result have been also studied in the setting of discrete time random walks \cite{vonluxburg2014hitting}. Our result extend the collapse phenomena to the continuous time setting in hierarchical community networks, and more interestingly our proof is completely independent, since its based on a description of the spectrum of the Hierarchical pseudo-inverse Laplacian for hierarchical community networks.

\section{Hierarchical community structure and its impact on the
stability of epidemics in the SIS model}

We now study a deterministic SIS epidemic model evolving on a network
represented by a connected graph $G$. We assume homogeneous infection
and recovery rates across the population. The dynamics of the $i$-th
subpopulation are given by
\begin{equation}
\label{eq:SISequationmodel}
\dot{x}_i(t) = -\delta x_i(t) + \beta \sum_{j=1}^N a_{ij} x_j(t)
\bigl(1 - x_i(t)\bigr),
\end{equation}
where $x_i(t) \in [0,1]$ denotes the fraction of infected individuals
in subpopulation $i$ at time $t$, $\delta > 0$ is the recovery rate,
and $\beta > 0$ is the infection rate.

Let $x(t) = [x_1(t),\dots,x_N(t)]^T$ denote the state vector of the
system. In matrix form, the dynamics can be written as
\begin{equation}
\dot{x}(t) = (\beta A - \delta I)x(t) - \beta X(t)Ax(t),
\end{equation}
where $A$ is the adjacency matrix of the graph, $I$ is the identity
matrix, and $X(t)=\mathrm{diag}(x_1(t),\dots,x_N(t))$.

The equilibrium $x=0$ (disease-free state) is globally asymptotically
stable provided the following spectral condition holds:
\begin{equation}\label{eq:thresholdSIS}
\lambda_1(A)\,\frac{\beta}{\delta} < 1,
\end{equation}
where $\lambda_1(A)$ denotes the largest eigenvalue of the adjacency
matrix of a graph $G=(V,E)$.

Following~\cite{vizuete2020graphon_sis}, when analyzing sequences of
graphs approximated by graphons, it is natural to consider families of
graphs indexed by their size $N$. In this setting, the epidemiological
parameters may also depend on $N$, and we shall explicitly write
$\beta_N$ and $\delta_N$ when needed.

As an application of the graphon framework, we derive a condition
ensuring convergence to the disease-free state for graphs sampled from
a given graphon $W$, expressed directly in terms of structural
properties of $W$.

Denote by $d_{\max}$ the maximum degree of $G$ and let
$\bar{d}=\frac{1}{|V|}\sum_{i\in V}d_i$ be the average degree. Since
$\lambda_1(A)=\max_{\|x\|_2=1}x^\top Ax$, it can be shown that
$\lambda_1(A)$ satisfies
\begin{equation}
\label{eq:inequalitydegreemaximumeigenvalue}
\bar{d} \;\le\; \lambda_1(A) \;\le\; d_{\max}.
\end{equation}
This expression allows us to obtain two sufficient conditions, one for
stability and one for instability. Using
Equation~\eqref{eq:thresholdSIS} we obtain:
\[
\bar{d}>\frac{\delta}{\beta}
\;\implies\;
\lambda_1(A)>\frac{\delta}{\beta},
\qquad
\text{(sufficient condition for the endemic regime),}
\]
\[
\frac{\delta}{\beta}>d_{\max}
\;\implies\;
\frac{\delta}{\beta}>\lambda_1(A),
\qquad
\text{(sufficient condition for disease-free stability).}
\]
The following result can be found in~\cite{vizuete2021laplacian}. 

\begin{lem}
\label{lem:convergenvemaxdegreegraphon}
    Given a graphon \(W\), and a positive number \(\nu<e^{-1}\), with probability at least \(1-\nu\) the normalized degrees of the graphs \(G_d^N\) and \(G_r^N\) sampled from \(W\) satisfy: 
    \begin{equation}
        \max_{i=1,\dots,N} \left|\frac{d_d(i)}{N}-\frac{d_r(i)}{N}\right|\leq \sqrt{\frac{\log(2N/\nu)}{N}}.
    \end{equation}
    for \(N\) large enough. 
\end{lem}
The next result follows from these observations.

\begin{thm}\label{thm:SISinequalityultra}
Let $W$ be an ultrametric graphon. Let $G_r^k$ be the attached
sequence of random graphs of size $N_k>0$. Let $A_r^k$ be the
associated random adjacency matrix sequence. For any $\eta>0$:

\noindent(i) If
\[
\sum_i\mu(I_i^{(M)})\nu(I_i^{(M)})
\;>\;
\frac{1}{N_k}\frac{\delta_{N_k}}{\beta_{N_k}}+\eta,
\]
for all sufficiently large $k$, then the disease-free equilibrium is
unstable (endemic regime).

\noindent(ii) If
\[
\frac{1}{N_k}\frac{\delta_{N_k}}{\beta_{N_k}}
\;>\;
\max_i\nu(I_i^{(M)})+\eta,
\]
for all sufficiently large $k$, then the disease-free equilibrium is
globally asymptotically stable.
\end{thm}

\begin{proof}
Fix $\eta>0$. For any graphon,
Theorem~\ref{teorem:convergencegraphons} implies that
$\frac{\bar{d}(G_r^k)}{N_k}$ converges a.s.\ to
$\int_{[0,1]^2}W(d(x,y))\,dx\,dy$. By
Lemma~\ref{lem:convergenvemaxdegreegraphon} the maximum degree can be
approximated via $\max_i\nu(I_i^{(M)})$; therefore for a given
$0<\nu<e^{-1}$ there exists a sufficiently large $k_0$ such that
$\sqrt{\frac{\log(2N_k/\nu)}{N_k}}<\frac{\eta}{2}$, and for all
$k\ge k_0$ we have
\[
\left|
\frac{\bar{d}(G_r^k)}{N_k}
-\sum_{i}\mu(I_i^{(M)})\nu(I_i^{(M)})
\right|<\eta/2,
\qquad
\left|
\frac{d_{\max}(G_r^k)}{N_k}-\max_i\nu(I_i^{(M)})
\right|<\eta/2,
\]
with probability at least $1-\nu$.

\smallskip
\noindent\emph{(Endemic regime).}
Assume that
\[
\sum_{i}\mu(I_i^{(M)})\nu(I_i^{(M)})
\;>\;
\frac{1}{N_k}\frac{\delta_{N_k}}{\beta_{N_k}}+\eta.
\]
Then, using the reverse triangle inequality,
\[
\frac{\bar{d}(G_r^k)}{N_k}
\;\ge\;
\sum_{i}\mu(I_i^{(M)})\nu(I_i^{(M)})
-\left|
\frac{\bar{d}(G_r^k)}{N_k}
-\sum_{i}\mu(I_i^{(M)})\nu(I_i^{(M)})
\right|
\;>\;
\frac{1}{N_k}\frac{\delta_{N_k}}{\beta_{N_k}}+\frac{\eta}{2}
\;>\;
\frac{1}{N_k}\frac{\delta_{N_k}}{\beta_{N_k}}.
\]
Multiplying by $N_k$ yields
$\bar{d}(G_r^k)>\frac{\delta_{N_k}}{\beta_{N_k}}$.
By~\eqref{eq:inequalitydegreemaximumeigenvalue}, we obtain
\[
\lambda_1(A_r^k)>\frac{\delta_{N_k}}{\beta_{N_k}},
\]
i.e.\ the disease-free equilibrium is unstable.

\smallskip
\noindent\emph{(Disease-free regime).}
Assume that
\[
\frac{1}{N_k}\frac{\delta_{N_k}}{\beta_{N_k}}
\;>\;
\max_i\nu(I_i^{(M)})+\eta.
\]
Again by the reverse triangle inequality,
\[
\frac{d_{\max}(G_r^k)}{N_k}
\;\le\;
\max_i\nu(I_i^{(M)})
+\left|
\frac{d_{\max}(G_r^k)}{N_k}-\max_i\nu(I_i^{(M)})
\right|
\;<\;
\frac{1}{N_k}\frac{\delta_{N_k}}{\beta_{N_k}}-\frac{\eta}{2}
\;<\;
\frac{1}{N_k}\frac{\delta_{N_k}}{\beta_{N_k}}.
\]
Multiplying by $N_k$ yields
$d_{\max}(G_r^k)<\frac{\delta_{N_k}}{\beta_{N_k}}$.
By~\eqref{eq:inequalitydegreemaximumeigenvalue}, we obtain
\[
\lambda_1(A_r^k)<\frac{\delta_{N_k}}{\beta_{N_k}},
\]
hence the disease-free equilibrium is globally asymptotically stable.
\end{proof}
 
From the behavior of the average degree of $A_r^k$ we then have that
$\lambda_1(A_r^k)=O(N_k)$, that is the maximum eigenvalue grows
linearly in the dense ultrametric regime. Consequently, the stability
condition
\[
\lambda_1(A_r^k)\frac{\beta_{N_k}}{\delta_{N_k}}<1
\]
is naturally of the scale
\[
N_k\frac{\beta_{N_k}}{\delta_{N_k}},
\]
so that meaningful comparison with $1$ requires the quantity
$N_k\frac{\beta_{N_k}}{\delta_{N_k}}$ to remain finite. Recall that
$\mu(I)$ for an interval in the filtration is the fraction of vertices
in cluster $I$, that is, $\mu(I)=\frac{m(I)}{N_k}$ and $w(h(I))$ is
the probability of connection between clusters inside the node $I$.
Inequality~\eqref{eq:inequalitydegreemaximumeigenvalue} alone fails to
describe the influence of these quantities. It is clear, for instance,
that a decrease in the maximum degree (i.e., a reduction in the maximum
number of contacts between individuals) increases the likelihood of
reaching a disease-free regime. However, since information about the
cluster structure is not captured in general, no intervention strategy based on
the interplay between cluster sizes and their connectivity can be
formulated.

Theorem~\ref{thm:SISinequalityultra} allows us to see the influence of
the hierarchical community structure of the contact network and its
influence on how the disease spreads over time. Inequality~$(i)$ of
Theorem~\ref{thm:SISinequalityultra} implies
\[
N_k\frac{\beta_{N_k}}{\delta_{N_k}}<\frac{1}{\max_i \nu(I_i^{(M)})+\eta}.
\]
First notice that every value $\nu$ depends on the sequence of clusters
containing the cluster $I_i^{(M)}$:
\begin{equation*}
\begin{split}
\nu(I_i^{(M)})&=\sum_{I_m\in \gamma_r(I)}
\mu(I_m)\Bigl(w(h(I_m))-w(h(I_{F(m)}))\Bigr)\\
&=\mu(I_i^{(M)})w(h(I_{i}^{(M)}))+\sum_{I_m\in \gamma_r(I_i^{(M)})\setminus [0,1]}
w(h(I_{F(m)}))\bigl(\mu(I_{F(m)})-\mu(I_m)\bigr)
\end{split}
\end{equation*}
therefore the degree magnitude depends on the overall contributions of
size of clusters and edge connection in all scales. The contribution of
each term
\[w(h(I_{F(m)}))\bigl(\mu(I_{F(m)})-\mu(I_m)\bigr),\]
depend not only on the edge connection between communities, but also on
the relative size of those communities. We see that reducing the
connectivity $w(h(I_{F(m)}))$ will reduce the overall degree, but such
reduction is affected by the relative size of the children communities
inside $I_{F(m)}$; if there is a very large child community $I_m$ the
reduction of $w(h(I_{F(m)}))$ has to be greater since the factor is
multiplied by $(\mu(I_{F(m)})-\mu(I_m))$. On the other hand, if
$I_{F(m)}$ is fragmented into smaller communities, and hence a bigger
factor $(\mu(I_{F(m)})-\mu(I_m))$, the reduction of the edge
connection between them has a greater impact towards the disease-free
regime.

Therefore detecting the best connected sub-communities (the ones with greater average degree) and controlling the connection between communities in the correct scale leads to a reduction of \(\max_i \nu(I_i^{(M)})\), which could be an optimal effort towards the disease free regime. 

On the other hand, inequality~$(ii)$ of
Theorem~\ref{thm:SISinequalityultra} leads to
\[
\frac{1}{\sum_i \mu(I_i^{(M)})\nu(I_i^{(M)})-\eta}
\;<\;
N_k\frac{\beta_{N_k}}{\delta_{N_k}}.
\]
In order to achieve this the average
$\sum_i \mu(I_i^{(M)})\nu(I_i^{(M)})$ has to be large enough;
therefore a failed lock-down-type strategy, that is, reducing the edge
connection between communities may have faltered due to the lack of
global cooperation between communities. This is different in the
heterogeneous regime: for example, assume that there is an interval
$I_j^{(M)}$ with very large $\nu(I_j^{(M)})$ such that the limit of
the normalized average degree $\bar{d}(G_r^k)$ satisfies
\[
\sum_i \mu(I_i^{(M)})\nu(I_i^{(M)})\approx \mu(I_i^{(M)})\nu(I_j^{(M)}),
\]
and
\[
\max_i \nu(I_i^{(M)})=\nu(I_j^{(M)}).
\]
Consequently, the edges connecting the nodes in community $I_j^{(M)}$
become highly important, and ``global'' cooperation in this sense is a
less effective strategy; the effect of this community dominates.

The global cooperation becomes more important the more regular the
underlying tree is: assume that the underlying tree is a balanced
binary tree, that is $\mu(I)=2^{-n}$ for intervals $I$ at level $n$,
in this case it holds
\[
w(h(I_{F(m)}))\bigl(\mu(I_{F(m)})-\mu(I_m)\bigr)
=\frac{1}{2^{n+1}}w(h(I_{F(m)})),
\]
for $I_m$ at level $n$. Moreover, if $w(h(I))$ is also level-regular,
that is $w(h(I))=w_n$ for all intervals $I$ at level $n$, then global
cooperation becomes the only winning strategy, since the maximum degree
can only be decreased if there is a decline of $w(h(I))=w_n$ for all
intervals at level $n$. While intuitively expected, these results
provide a formal mathematical basis for the role of the topology in
epidemic dynamics: whereas homogeneity points toward global
cooperation, heterogeneity places the responsibility on particular
communities.
\subsubsection{Numerical experiments}

Let $B\in [0,1]$ be the number representing the budget or capacity for
reducing the edges between communities. For $0<\alpha<1$ we can invest
a fraction of $B$ into reducing the edge connection of a specific
community $I\subseteq[0,1]$ attached to an ultrametric graphon
$W:[0,1]\rightarrow[0,1]$ by replacing:
\[
w(d(x,y))=w(h(I)), \quad x,y\in I
\rightarrow
w(d(x,y))=(1-\alpha(B+\varepsilon))w(h(I)), \quad x,y\in I,
\]
where $\varepsilon>0$ is a small parameter forcing the probability to
be greater than zero independently of $B$.

In this way we reduce the probability of connection between the
children clusters of $I$. The aim of the following experiments is to
show numerical evidence of the effect of different strategies of
cooperation in heterogeneous and homogeneous communities. For the
simulations we consider a binary ultrametric graphon constructed from a
rooted tree of fixed depth $L$. Hence, each internal node
$I\subseteq[0,1]$ is divided into two children $I_0$ and $I_1$. In
order to control the homogeneity and heterogeneity of the communities,
we construct such partition randomly. Consider a Dirichlet
distribution:
\[
(p,1-p) \sim \mathrm{Dirichlet}(c,c),
\]
which is equivalent to $p \sim \mathrm{Beta}(c,c)$. If $I=(a,b)$, the
children intervals are defined by
\[
I_0 = (a,\, a + p(b-a)),
\qquad
I_1 = (a + p(b-a),\, b).
\]
The parameter $c>0$ determines the heterogeneity of the masses: large
$c$ produces nearly equal splits (homogeneous case), while small $c$
generates highly unbalanced partitions. For $x,y\in[0,1]$, let
$d(x,y)$ denote the depth of their lowest common ancestor in the tree.
We define the graphon
\[
W(x,y) = w_{d(x,y)},
\]
where $\{w_\ell\}_{\ell=0}^{L}$ is a monotone sequence; in our
experiments we use:
\[
w_\ell = w_{\min} + (w_{\max}-w_{\min})\left(\frac{\ell}{L}\right)^\gamma,
\qquad \gamma>1,
\]
so that connection probabilities increase with hierarchical proximity.
Random graphs of size $M$ are sampled from the graphon $W(x,y)$, where
we connect pairs of vertices independently with probability
$W(x_i,x_j)$. Define $\rho_{\infty}$ to be
\[
\rho_\infty := \frac{1}{N}\sum_{i=1}^N x_i^\ast,
\]
where $x^\ast \in (0,1)^N$ is the non-trivial equilibrium of
equation~\eqref{eq:SISequationmodel}. Thus, $\rho_\infty = 0$ in the
disease-free regime ($\lambda_1(A)\beta/\delta < 1$), while
$\rho_\infty > 0$ in the endemic regime~\cite{nowzari2016analysis}. We
analyze two strategies of cooperation at different levels.\newline
\textit{Global cooperation.}\newline
The budget $B$ will be equally distributed in the communities at a
given level. Therefore, for any community at level $\ell$ we make the
reduction
\[
w(d(x,y))\rightarrow\left(1-\frac{(B+\varepsilon)}{2^{\ell}}\right)w(d(x,y)),
\quad x,y\in I: I \text{ is at level } \ell.
\]
\textit{Largest community intervention.}
Let $I_{\max}$ be the interval satisfying
$\nu(I_{\max})=\max_i\nu(I_i^{(M)})$. The budget $B$ will be invested
into reducing the edge-connectivity of communities containing
$I_{\max}$, thus making the following reduction
\[
w(d(x,y))\rightarrow (1-(B+\varepsilon))w(d(x,y)),
\quad x,y\in I: I\,\text{is at level }\ell\,\text{and }
I\in\gamma_r(I_{\max}).
\]
We sampled $M=260$ points and take $c\in\{1.6,100.0\}$. The depth of
the underlying binary tree is $L=7$, and the parameters of $W$ are
$w_{\max}=0.67$, $w_{\min}=0.03$ and $\gamma=1.8$. For the
experiments we take $\tau:=N_k\frac{\beta_{N_k}}{\delta_{N_k}}$
taking the integer values $n=4,\dots,20$. Define
\[
\tau_{\max}^{\mathrm{crit}}:=\frac{1}{\max_i\nu(I_i^{(M)})},
\qquad
\tau_{\mathrm{avg}}^{\mathrm{crit}}:=\frac{1}{\sum_i\mu(I_i^{(M)})\nu(I_i^{(M)})}.
\]
\begin{figure}[H]
    \centering
    \includegraphics[width=\linewidth]{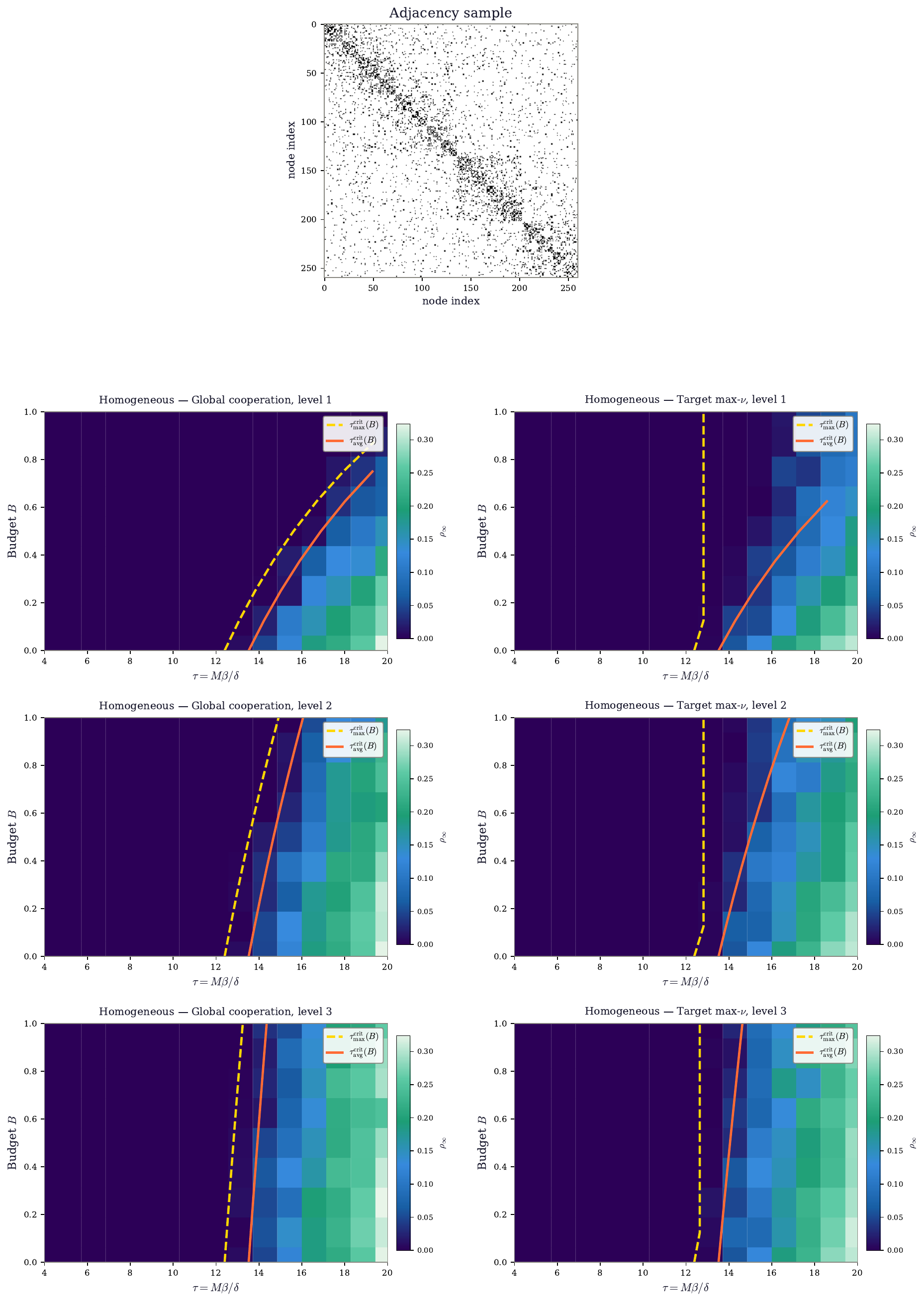}
    \caption{For \(c=100.0\), the resulting graphons are highly homogeneous as shown in the adjacency matrix sample. 
    We present the heat map \(\rho_{\infty}=\rho_{\infty}(\tau,B)\) for the two strategies. The purple area represent all the cases where free-disease is achieved. Left column show the global cooperation intervention on the first three levels, whereas the right column show the largest community intervention. The curves of the values \(\tau_{max}^{crit}\) and \(\tau_{avg}^{crit}\) are presented in each case.}  
    \label{fig:resultshomogeneous}
\end{figure}

\begin{figure}[H]
    \centering
    \includegraphics[width=\linewidth]{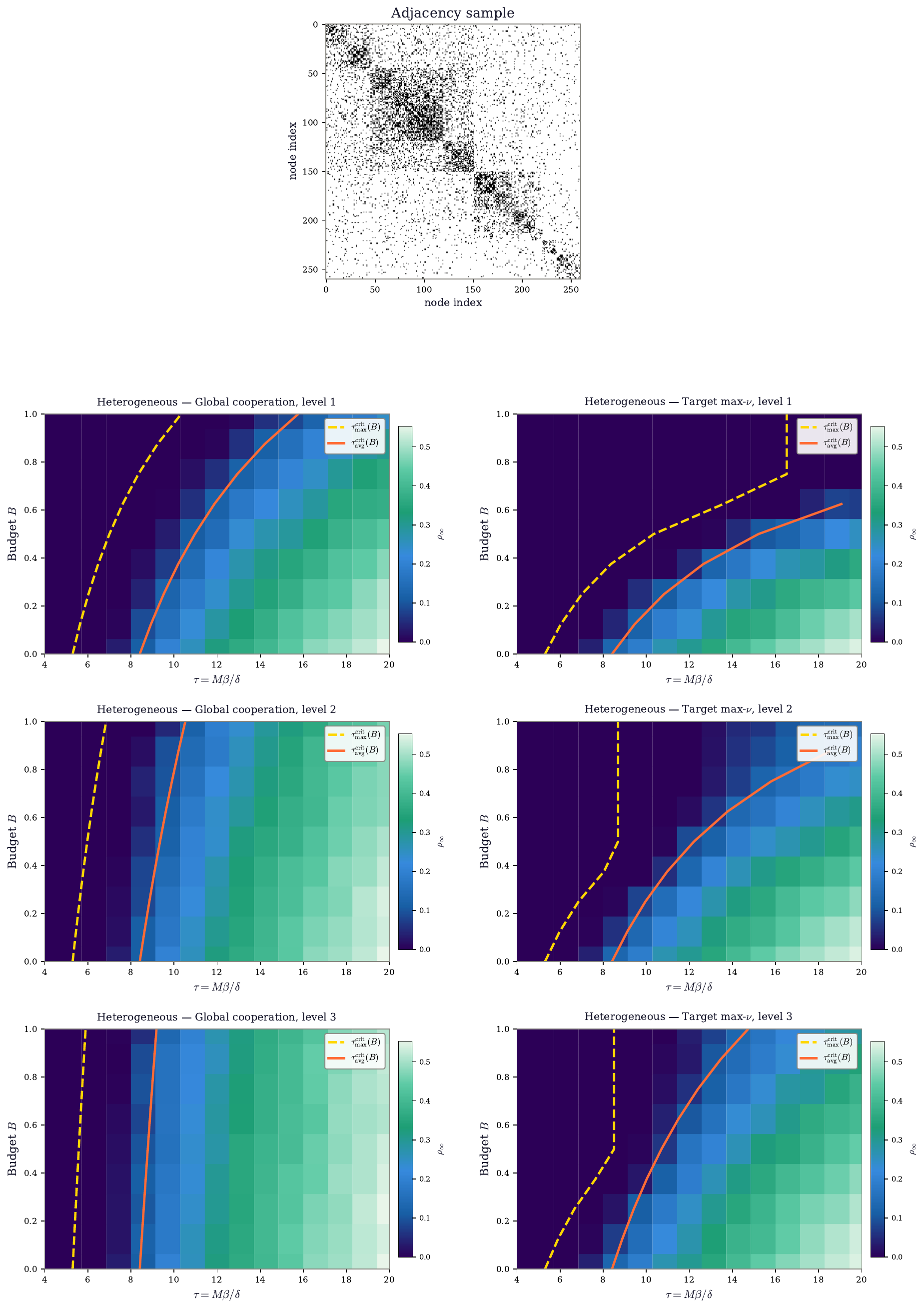}
    \caption{For \(c=1.6\) the sampled graphs have more heterogeneous communities. The heatmap \(\rho_{\infty}=\rho_{\infty}(\tau,B)\) for the two strategies is presented. In this case we see a substantial difference between the two strategies as expected.}
    \label{fig:resultsheterogeneous}
\end{figure}
As expected, for the homogeneous case the global cooperation is more
optimal increasing the possible scenarios where the disease-free
stability is achieved, nevertheless the difference in this case is not
as significant. As shown in Figure~\ref{fig:resultshomogeneous}, the
global cooperation pushes the curve $\tau_{\max}^{\mathrm{crit}}$ more
to the right, implying a greater area of disease-free scenarios. On the
other hand we see that the largest community intervention results in a
bigger gap between the curves and therefore a smaller range of
applicability of Theorem~\ref{thm:SISinequalityultra}. Moreover, since
we are always targeting the community containing $I_{\max}$, which is
fixed, the values of $\tau_{\max}^{\mathrm{crit}}$ become constant at
some point (represented in the plots as a vertical line), maintaining
the area $\tau<\tau_{\max}^{\mathrm{crit}}$ constant and not
optimizing the disease-free region. The global cooperation results in a
better alternative since the average degree is always minimized,
nevertheless the impact is very small for deeper levels.

The heterogeneous case presented in
Figure~\ref{fig:resultsheterogeneous} is completely different.
Unbalanced communities lead to the creation of hotspots (see the
adjacency sample in Figure~\ref{fig:resultsheterogeneous}) with high
degree; therefore, reducing the connectivity per level in the global
cooperation strategy produces only a marginal increase in both
$\tau_{\mathrm{avg}}^{\mathrm{crit}}$ and
$\tau_{\max}^{\mathrm{crit}}$. As before, the performance worsens as
the level becomes deeper. On the other hand, interventions targeting
the largest community lead to better results. Since $I_{\max}$ tends to
be contained in the densest clusters, decreasing the connections of the
clusters along its path $\gamma_r(I_{\max})$ produces a more
pronounced increase in $\tau_{\max}^{\mathrm{crit}}$. In this setting,
the maximum degree effectively persists within communities along the
path $\gamma_r(I_{\max})$ for a significant range of the parameter.
Although the community with the highest average degree eventually shifts
to another cluster, giving rise to the vertical effect, this transition
is delayed.

\section*{Acknowledgements}

 Patrick Bradley is warmly thanked for the many useful conversations we had and for his general advice. Wilson Zúñiga is thanked for introduced me to the concept of graphons. This work is supported by the Deutsche Forschungsgemeinschaft
under project number 469999674.

\bibliography{biblio}

@article{Glasscock2015,
  author  = {Glasscock, Daniel},
  title   = {What is a graphon?},
  journal = {Notices of the American Mathematical Society},
  volume  = {62},
  number  = {1},
  pages   = {46--48},
  year    = {2015}
}

@misc{MoranLedezma2026,
  author       = {Mor{\'a}n Ledezma, Angel Alfredo},
  title        = {Spectral Geometry and Heat Kernels on Phylogenetic Trees},
  year         = {2026},
  howpublished = {arXiv preprint},
  note         = {arXiv:2603.20922},
  doi          = {10.48550/arXiv.2603.20922}
}

@inproceedings{BrinPage1998,
  author    = {Brin, Sergey and Page, Lawrence},
  title     = {Anatomy of a large-scale hypertextual web search engine},
  booktitle = {Proceedings of the Seventh International World Wide Web Conference},
  year      = {1998},
  pages     = {107--117}
}

@article{LangvilleMeyer2004,
  author  = {Langville, Amy N. and Meyer, Carl D.},
  title   = {Deeper inside {PageRank}},
  journal = {Internet Mathematics},
  year    = {2004},
  volume  = {1},
  number  = {3},
  pages   = {335--380},
  doi     = {10.1080/15427951.2004.10129091}
}

@article{LambiottRosvall2012,
  author  = {Lambiotte, Renaud and Rosvall, Martin},
  title   = {Ranking and clustering of nodes in networks with smart
             teleportation},
  journal = {Physical Review E},
  year    = {2012},
  volume  = {85},
  number  = {5},
  pages   = {056107},
  doi     = {10.1103/PhysRevE.85.056107}
}

@article{MasudaPorterLambiotte2017,
  author  = {Masuda, Naoki and Porter, Mason A. and Lambiotte, Renaud},
  title   = {Random walks and diffusion on networks},
  journal = {Physics Reports},
  year    = {2017},
  volume  = {716--717},
  pages   = {1--58},
  doi     = {10.1016/j.physrep.2017.07.007}
}

@article{PetitLambiotte2021,
  author  = {Petit, Julien and Lambiotte, Renaud and Carletti, Timoteo},
  title   = {Random walks on dense graphs and graphons},
  journal = {SIAM Journal on Applied Mathematics},
  year    = {2021},
  volume  = {81},
  number  = {6},
  pages   = {2323--2345},
  doi     = {10.1137/20M1339246}
}

@article{KlimmJonesSchaub2022,
  author  = {Klimm, Florian and Jones, Nick S. and Schaub, Michael T.},
  title   = {Modularity maximization for graphons},
  journal = {SIAM Journal on Applied Mathematics},
  year    = {2022},
  volume  = {82},
  number  = {6},
  doi     = {10.1137/22M1492003}
}

@article{LovaszSzegedy2006,
  author  = {Lov\'{a}sz, L\'{a}szl\'{o} and Szegedy, Bal\'{a}zs},
  title   = {Limits of dense graph sequences},
  journal = {Journal of Combinatorial Theory, Series B},
  year    = {2006},
  volume  = {96},
  number  = {6},
  pages   = {933--957},
  doi     = {10.1016/j.jctb.2006.05.002}
}

@article{BorgsI2008,
  author  = {Borgs, Christian and Chayes, Jennifer T. and
             Lov\'{a}sz, L\'{a}szl\'{o} and S\'{o}s, Vera T. and
             Vesztergombi, Katalin},
  title   = {Convergent sequences of dense graphs {I}: Subgraph
             frequencies, metric properties and testing},
  journal = {Advances in Mathematics},
  year    = {2008},
  volume  = {219},
  number  = {6},
  pages   = {1801--1851},
  doi     = {10.1016/j.aim.2008.07.008}
}

@article{BorgsII2012,
  author  = {Borgs, Christian and Chayes, Jennifer T. and
             Lov\'{a}sz, L\'{a}szl\'{o} and S\'{o}s, Vera T. and
             Vesztergombi, Katalin},
  title   = {Convergent sequences of dense graphs {II}: Multiway cuts
             and statistical physics},
  journal = {Annals of Mathematics},
  year    = {2012},
  volume  = {176},
  number  = {1},
  pages   = {151--219},
  doi     = {10.4007/annals.2012.176.1.2}
}

@book{LovaszBook2012,
  author    = {Lov\'{a}sz, L\'{a}szl\'{o}},
  title     = {Large Networks and Graph Limits},
  series    = {American Mathematical Society Colloquium Publications},
  volume    = {60},
  publisher = {American Mathematical Society},
  address   = {Providence, RI},
  year      = {2012}
}

@article{GaoLuZhou2015,
  author  = {Gao, Chao and Lu, Yu and Zhou, Harrison H.},
  title   = {Rate-optimal graphon estimation},
  journal = {Annals of Statistics},
  year    = {2015},
  volume  = {43},
  number  = {6},
  pages   = {2624--2652},
  doi     = {10.1214/15-AOS1354}
}

@article{KloppTsybakov2017,
  author  = {Klopp, Olga and Tsybakov, Alexandre B. and Verzelen, Nicolas},
  title   = {Oracle inequalities for network models and sparse graphon
             estimation},
  journal = {Annals of Statistics},
  year    = {2017},
  volume  = {45},
  number  = {1},
  pages   = {316--354},
  doi     = {10.1214/16-AOS1454}
}

@article{Civier2019,
  author  = {Civier, Oren and Smith, Robert Elton and Yeh, Chun-Hung
             and Connelly, Alan and Calamante, Fernando},
  title   = {Is removal of weak connections necessary for
             graph-theoretical analysis of dense weighted structural
             connectomes from diffusion {MRI}?},
  journal = {NeuroImage},
  year    = {2019},
  volume  = {194},
  pages   = {68--81},
  doi     = {10.1016/j.neuroimage.2019.02.039}
}

@article{Lambiotte2015,
  author  = {Lambiotte, Renaud and Delvenne, Jean-Charles and Barahona, Mauricio},
  title   = {Random walks, {Markov} processes and the multiscale modular
             organization of complex networks},
  journal = {IEEE Transactions on Network Science and Engineering},
  year    = {2015},
  volume  = {1},
  number  = {2},
  pages   = {76--90},
  doi     = {10.1109/TNSE.2015.2391998}
}

@article{Schaub2012,
  author  = {Schaub, Michael T. and Delvenne, Jean-Charles and
             Yaliraki, Sophia N. and Barahona, Mauricio},
  title   = {Markov dynamics as a zooming lens for multiscale community
             detection: non clique-like communities and the field-of-view limit},
  journal = {PLoS ONE},
  year    = {2012},
  volume  = {7},
  number  = {2},
  pages   = {e32210},
  doi     = {10.1371/journal.pone.0032210}
}

@article{Patelli2020,
  author  = {Patelli, Aurelio and Gabrielli, Andrea and Cimini, Giulio},
  title   = {Generalized {Markov} stability of network communities},
  journal = {Physical Review E},
  year    = {2020},
  volume  = {101},
  number  = {5},
  pages   = {052301},
  doi     = {10.1103/PhysRevE.101.052301}
}

@article{LiZhang2013,
  author  = {Li, Hui-Jia and Zhang, Xiang-Sun},
  title   = {Analysis of stability of community structure across
             multiple hierarchical levels},
  journal = {Europhysics Letters},
  year    = {2013},
  volume  = {103},
  number  = {5},
  pages   = {58002},
  doi     = {10.1209/0295-5075/103/58002}
}

@article{Bonaccorsi2014,
  author  = {Bonaccorsi, Stefano and Ottaviano, Stefania and
             De Pellegrini, Francesco and Socievole, Annalisa and
             Van Mieghem, Piet},
  title   = {Epidemic outbreaks in two-scale community networks},
  journal = {Physical Review E},
  year    = {2014},
  volume  = {90},
  number  = {1},
  pages   = {012810},
  doi     = {10.1103/PhysRevE.90.012810}
}

@article{LiuHu2005,
  author  = {Liu, Zonghua and Hu, Bambi},
  title   = {Epidemic spreading in community networks},
  journal = {Europhysics Letters},
  year    = {2005},
  volume  = {72},
  number  = {2},
  pages   = {315--321},
  doi     = {10.1209/epl/i2004-10550-5}
}

@article{Nadini2018,
  author  = {Nadini, Matthieu and Sun, Kaiyuan and Ubaldi, Enrico and
             Starnini, Michele and Rizzo, Alessandro and Perra, Nicola},
  title   = {Epidemic spreading in modular time-varying networks},
  journal = {Scientific Reports},
  year    = {2018},
  volume  = {8},
  pages   = {2908},
  doi     = {10.1038/s41598-018-20908-x}
}

@article{Odor2021,
  author  = {{\'O}dor, G\'eza},
  title   = {Nonuniversal power-law dynamics of susceptible infected
             recovered models on hierarchical modular networks},
  journal = {Physical Review E},
  year    = {2021},
  volume  = {103},
  number  = {6},
  pages   = {062112},
  doi     = {10.1103/PhysRevE.103.062112}
}

@article{BorzoneMas2023,
  author  = {Borzone Mas, Dalmiro and Scarabotti, Pablo A. and Vaschetto, Pablo A.
             and Alvarenga, Patricio and Vazquez, Martin and Arim, Mat{\'\i}as},
  title   = {On traits matching and the modular organization of food web
             and occurrence networks},
  journal = {Journal of Animal Ecology},
  year    = {2023},
  doi     = {10.1111/1365-2656.13900}
}

@article{Porter2009,
  author  = {Porter, Mason A. and Onnela, Jukka-Pekka and Mucha, Peter J.},
  title   = {Communities in networks},
  journal = {Notices of the American Mathematical Society},
  year    = {2009},
  volume  = {56},
  number  = {11},
  pages   = {1082--1097}
}

@article{GirvanNewman2002,
  author    = {Girvan, Michelle and Newman, Mark E. J.},
  title     = {Community structure in social and biological networks},
  journal   = {Proceedings of the National Academy of Sciences},
  year      = {2002},
  volume    = {99},
  number    = {12},
  pages     = {7821--7826},
  doi       = {10.1073/pnas.122653799}
}

@article{NewmanGirvan2004,
  author    = {Newman, Mark E. J. and Girvan, Michelle},
  title     = {Finding and evaluating community structure in networks},
  journal   = {Physical Review E},
  year      = {2004},
  volume    = {69},
  number    = {2},
  pages     = {026113},
  doi       = {10.1103/PhysRevE.69.026113}
}

@article{Fortunato2010,
  author    = {Fortunato, Santo},
  title     = {Community detection in graphs},
  journal   = {Physics Reports},
  year      = {2010},
  volume    = {486},
  number    = {3--5},
  pages     = {75--174},
  doi       = {10.1016/j.physrep.2009.11.002}
}

@article{Ravasz2002,
  author    = {Ravasz, Erzs\'{e}bet and Somera, Anna L. and Mongru, Dale A.
               and Oltvai, Zolt\'{a}n N. and Barab\'{a}si, Albert-L\'{a}szl\'{o}},
  title     = {Hierarchical organization of modularity in metabolic networks},
  journal   = {Science},
  year      = {2002},
  volume    = {297},
  number    = {5586},
  pages     = {1551--1555},
  doi       = {10.1126/science.1073374}
}

@article{SalesPardo2007,
  author    = {Sales-Pardo, Marta and Guimer\`{a}, Roger and
               Moreira, Andr\'{e} A. and Amaral, Lu\'{i}s A. Nunes},
  title     = {Extracting the hierarchical organization of complex systems},
  journal   = {Proceedings of the National Academy of Sciences},
  year      = {2007},
  volume    = {104},
  number    = {39},
  pages     = {15224--15229},
  doi       = {10.1073/pnas.0703740104}
}

@article{Clauset2008,
  author    = {Clauset, Aaron and Moore, Cristopher and Newman, Mark E. J.},
  title     = {Hierarchical structure and the prediction of missing links in networks},
  journal   = {Nature},
  year      = {2008},
  volume    = {453},
  number    = {7191},
  pages     = {98--101},
  doi       = {10.1038/nature06830}
}

@article{Schaub2023,
  author    = {Schaub, Michael T. and Li, Jiaze and Peel, Leto},
  title     = {Hierarchical community structure in networks},
  journal   = {Physical Review E},
  year      = {2023},
  volume    = {107},
  number    = {5},
  pages     = {054305},
  doi       = {10.1103/PhysRevE.107.054305}
}

@article{Dreveton2025,
  author    = {Dreveton, Maximilien and Kuroda, Daichi and
               Grossglauser, Matthias and Thiran, Patrick},
  title     = {When does bottom-up beat top-down in hierarchical community detection?},
  journal   = {Journal of the American Statistical Association},
  year      = {2025},
  volume    = {00},
  pages     = {1--12},
  doi       = {10.1080/01621459.2025.2569711}
}

@article{Rezvani2022,
  author    = {Rezvani, Mojtaba and Kazemian, Fazeleh Sadat},
  title     = {A survey on hierarchical community detection in large-scale complex networks},
  journal   = {AUT Journal of Mathematics and Computing},
  year      = {2022},
  volume    = {3},
  number    = {2},
  pages     = {173--184},
  doi       = {10.22060/AJMC.2022.21715.1103}
}

@article{Meunier2009,
  author    = {Meunier, David and Lambiotte, Renaud and Fornito, Alex and
               Ersche, Karen D. and Bullmore, Edward T.},
  title     = {Hierarchical modularity in human brain functional networks},
  journal   = {Frontiers in Neuroinformatics},
  year      = {2009},
  volume    = {3},
  pages     = {37},
  doi       = {10.3389/neuro.11.037.2009}
}

@article{Meunier2010,
  author    = {Meunier, David and Lambiotte, Renaud and Bullmore, Edward T.},
  title     = {Modular and hierarchically modular organization of brain networks},
  journal   = {Frontiers in Neuroscience},
  year      = {2010},
  volume    = {4},
  pages     = {200},
  doi       = {10.3389/fnins.2010.00200}
}

@article{Wierzbinski2023,
  author    = {Wierzbinski, Marcin and Falc\'{o}-Roget, Joan and Crimi, Alessandro},
  title     = {Community detection in brain connectomes with hybrid quantum computing},
  journal   = {Scientific Reports},
  year      = {2023},
  volume    = {13},
  pages     = {1--12},
  doi       = {10.1038/s41598-023-35040-2}
}

@article{Simon1962,
  author    = {Simon, Herbert A.},
  title     = {The architecture of complexity},
  journal   = {Proceedings of the American Philosophical Society},
  year      = {1962},
  volume    = {106},
  number    = {6},
  pages     = {467--482}
}

@book{networksnewman,
    author = {Newman, Mark},
    title = {Networks},
    publisher = {Oxford University Press},
    year = {2018},
    month = {07},
    isbn = {9780198805090},
    doi = {10.1093/oso/9780198805090.001.0001},
    url = {https://doi.org/10.1093/oso/9780198805090.001.0001},
}

@article{communitydetection,
title = {A comprehensive review of community detection in graphs},
journal = {Neurocomputing},
volume = {600},
pages = {128169},
year = {2024},
issn = {0925-2312},
doi = {https://doi.org/10.1016/j.neucom.2024.128169},
url = {https://www.sciencedirect.com/science/article/pii/S0925231224009408},
author = {Jiakang Li and Songning Lai and Zhihao Shuai and Yuan Tan and Yifan Jia and Mianyang Yu and Zichen Song and Xiaokang Peng and Ziyang Xu and Yongxin Ni and Haifeng Qiu and Jiayu Yang and Yutong Liu and Yonggang Lu},
}

@article{nowzari2016analysis,
  author    = {Nowzari, Cameron and Preciado, Victor M. and Pappas, George J.},
  title     = {Analysis and Control of Epidemics: A Survey of Spreading Processes on Complex Networks},
  journal   = {IEEE Control Systems Magazine},
  volume    = {36},
  number    = {1},
  pages     = {26--46},
  month     = feb,
  year      = {2016},
  doi       = {10.1109/MCS.2015.2495000}
}

@article{vizuete2020graphon_sis,
  author    = {Vizuete, Renato and Frasca, Paolo and Garin, Federica},
  title     = {Graphon-Based Sensitivity Analysis of {SIS} Epidemics},
  journal   = {IEEE Control Systems Letters},
  volume    = {4},
  number    = {3},
  pages     = {542--547},
  month     = jul,
  year      = {2020},
  doi       = {10.1109/LCSYS.2020.2971021}
}

@article{vonluxburg2014hitting,
  author    = {von Luxburg, Ulrike and Radl, Agnes and Hein, Matthias},
  title     = {Hitting and Commute Times in Large Random Neighborhood Graphs},
  journal   = {Journal of Machine Learning Research},
  volume    = {15},
  pages     = {1751--1798},
  year      = {2014}
}

@article{vizuete2021laplacian,
  author    = {Vizuete, Renato and Garin, Federica and Frasca, Paolo},
  title     = {The {L}aplacian Spectrum of Large Graphs Sampled From Graphons},
  journal   = {IEEE Transactions on Network Science and Engineering},
  volume    = {8},
  number    = {2},
  pages     = {1711--1721},
  month     = {apr--jun},
  year      = {2021},
  doi       = {10.1109/TNSE.2021.3069675}
}

@article{Fiedler1973,
author = {Fiedler, Miroslav},
journal = {Czechoslovak Mathematical Journal},
language = {eng},
number = {2},
pages = {298-305},
publisher = {Institute of Mathematics, Academy of Sciences of the Czech Republic},
title = {Algebraic connectivity of graphs},
url = {http://eudml.org/doc/12723},
volume = {23},
year = {1973},
}

@article{
girvan2002community,
author = {M. Girvan  and M. E. J. Newman },
title = {Community structure in social and biological networks},
journal = {Proceedings of the National Academy of Sciences},
volume = {99},
number = {12},
pages = {7821-7826},
year = {2002},
doi = {10.1073/pnas.122653799},
URL = {https://www.pnas.org/doi/abs/10.1073/pnas.122653799},
eprint = {https://www.pnas.org/doi/pdf/10.1073/pnas.122653799},
}

@article{LOVASZ2006933,
title = {Limits of dense graph sequences},
journal = {Journal of Combinatorial Theory, Series B},
volume = {96},
number = {6},
pages = {933-957},
year = {2006},
issn = {0095-8956},
doi = {https://doi.org/10.1016/j.jctb.2006.05.002},
url = {https://www.sciencedirect.com/science/article/pii/S0095895606000517},
author = {László Lovász and Balázs Szegedy},
}

@article{bradleymoranshape,
    author = {Bradley, Patrick Erik and Ledezma, Angel Moran},
    title = {Hearing shapes via p-adic Laplacians},
    journal = {Journal of Mathematical Physics},
    volume = {64},
    number = {11},
    pages = {113502},
    year = {2023},
    month = {11},
    issn = {0022-2488},
    doi = {10.1063/5.0152374},
    
}

@article{MoranLedezma2025,
  author    = {Mor\'an Ledezma, \'Angel},
  title     = {Time-varying energy landscapes and temperature paths: dynamic transition rates in locally ultrametric complex systems},
  journal   = {Journal of Statistical Mechanics: Theory and Experiment},
  year      = {2025},
  volume    = {2025},
  pages     = {113501},
  doi       = {10.1088/1742-5468/ae120f},
  url       = {https://stacks.iop.org/JSTAT/2025/113501},
  publisher = {IOP Publishing},
  note      = {Open Access}
}

@article{rammal1986ultrametricity,
  author  = {Rammal, R. and Toulouse, G. and Virasoro, M. A.},
  title   = {Ultrametricity for Physicists},
  journal = {Reviews of Modern Physics},
  volume  = {58},
  number  = {3},
  pages   = {765--788},
  year    = {1986}
}

@article{BramburgerHolzer2023,
  author    = {Jason Bramburger and Matt Holzer},
  title     = {Pattern Formation in Random Networks Using Graphons},
  journal   = {SIAM Journal on Mathematical Analysis},
  volume    = {55},
  number    = {3},
  pages     = {2150--2185},
  year      = {2023},
  doi       = {10.1137/21M1455875}
}

@article{daviskahntheorem,
    author = {Yu, Y. and Wang, T. and Samworth, R. J.},
    title = {A useful variant of the Davis–Kahan theorem for statisticians},
    journal = {Biometrika},
    volume = {102},
    number = {2},
    pages = {315-323},
    year = {2014},
    month = {04},
    issn = {0006-3444},
    doi = {10.1093/biomet/asv008},
    url = {https://doi.org/10.1093/biomet/asv008},
    eprint = {https://academic.oup.com/biomet/article-pdf/102/2/315/9642505/asv008.pdf},
}

@article{PhaseTransitionsSpectralCD,
  title        = {Phase Transitions in Spectral Community Detection},
  author       = {Chen, Pin-Yu and Hero, Alfred O.},
  journal      = {IEEE Transactions on Signal Processing},
  year         = {2015},
  volume       = {63},
  number       = {16},
  pages        = {4339--4352},
  doi          = {10.1109/TSP.2015.2442958}
}

@book{vanKampen2007,
  author    = {N. G. van Kampen},
  title     = {Stochastic Processes in Physics and Chemistry},
  edition   = {3},
  year      = {2007},
  publisher = {North-Holland},
  address   = {Amsterdam},
  series    = {North-Holland Personal Library},
  isbn      = {978-0-444-52965-7},
  doi       = {10.1016/B978-0-444-52965-7.X5000-4}
}

@article{ZunigaNetworks,
author={W. {Z\'{u}\~{n}iga-Galindo}},
title={Reaction-diffusion equations on complex networks and {Turing} patterns, via $p$-adic analysis},
year={2020},
journal={Journal of Mathematical Analysis and Applications},
volume={491},
number={1},
pages={124239}
}

@article{Zuniga2022,
title={Ultrametric diffusion, rugged energy landscapes and transition networks},
author={W.A. 
       Zúñiga-Galindo},
journal={Physica A: Statistical Mechanics and its Applications},
volume={597}, 
pages={127221},
year={2022}
}

@article{ZZ2023,
author={Zúñiga-Galindo, W.A. and
Zambrano-Luna, B.A.}, title={Hierarchical
{Wilson–Cowan} Models and
Connection Matrices}, journal={Entropy},
year={2023},
volume={25}, 
pages={949}
}

\end{document}